\documentclass[11pt,oneside]{amsart}

\usepackage{amsmath,amsthm,amssymb,amsfonts,amscd}
\usepackage{url}
\usepackage{color}
\usepackage{setspace}
\usepackage{enumerate}

\usepackage{epsfig}
\usepackage{graphicx}
\usepackage{wrapfig}
\usepackage{mathrsfs}
\usepackage{pinlabel}

\definecolor{grey}{rgb}{0.7,0.7,0.7}

\newenvironment{changemargin}[2]{\begin{list}{}{%
\setlength{\topsep}{0pt}%
\setlength{\leftmargin}{0pt}%
\setlength{\rightmargin}{0pt}%
\setlength{\listparindent}{\parindent}%
\setlength{\itemindent}{\parindent}%
\setlength{\parsep}{0pt plus 1pt}%
\addtolength{\leftmargin}{#1}%
\addtolength{\rightmargin}{#2}%
}\item }{\end{list}}

\theoremstyle{plain}
\newtheorem{theorem}{Theorem}
\newtheorem{proposition}[theorem]{Proposition}

\newtheorem{observation}[theorem]{Observation}
\newtheorem{observations}[theorem]{Observations}
\newtheorem{claim}[theorem]{Claim}
\newtheorem{lemma}[theorem]{Lemma}

\newtheoremstyle{theoremwithref}{}{}{\itshape}{}{\bfseries}{.}{.5em}{#1 #2 #3}
\theoremstyle{theoremwithref}

\theoremstyle{definition}
\newtheorem{definition}[theorem]{Definition}

\newtheorem{remark}[theorem]{Remark}

\newtheorem{conjecture}[theorem]{Conjecture}

\numberwithin{theorem}{section}
\numberwithin{equation}{section}

\newcommand{\D}{\mathrm{d}}

\newcommand{\CP}{\mathcal C}
\newcommand{\CHS}{\mathcal H}
\newcommand{\ACP}{\mathsf{C}}
\newcommand{\ACHS}{\mathsf{H}}

\newcommand{\NN}{\mathbb{N}}
\newcommand{\ZZ}{\mathbb{Z}}
\newcommand{\QQ}{\mathbb{Q}}
\newcommand{\RR}{\mathbb{R}}

\newcommand{\HH}{\mathbb{H}}
\newcommand{\PP}{\mathbb{P}}

\renewcommand{\SS}{\mathbb{S}}

\newcommand{\GL}{\mathrm{GL}}
\newcommand{\SO}{\mathrm{SO}}
\newcommand{\OO}{\mathrm{O}}
\newcommand{\PO}{\mathrm{PO}}

\newcommand{\PSL}{\mathrm{PSL}}
\newcommand{\PGL}{\mathrm{PGL}}

\newcommand{\g}{\mathfrak{g}}

\newcommand{\AdS}{\mathrm{AdS}^3}
\newcommand{\AdSS}{\mathrm{AdS}}

\newcommand{\Adm}{\mathsf{Adm}^+}
\newcommand{\adm}{\mathsf{adm}}

\newcommand{\F}{F}
\newcommand{\T}{\mathfrak{F}}
\newcommand{\Ad}{\operatorname{Ad}}

\newcommand{\Hom}{\mathrm{Hom}}
\newcommand{\spa}{\mathrm{span}}
\newcommand{\ie}{i.e.\ }
\newcommand{\eg}{e.g.\ }
\newcommand{\resp}{resp.\ }

\newcommand{\SQ}{\mathrm{SQ}}

\title{Margulis spacetimes via the arc complex}

\author{Jeffrey Danciger}
\address{Department of Mathematics, The University of Texas at Austin, 1 University Station C1200, Austin, TX 78712, USA}
\email{jdanciger@math.utexas.edu}

\author{Fran\c{c}ois Gu\'eritaud}
\address{CNRS and Universit\'e Lille 1, Laboratoire Paul Painlev\'e, 59655 Villeneuve d'Ascq Cedex, France}
\email{francois.gueritaud@math.univ-lille1.fr}

\author{Fanny Kassel}
\address{CNRS and Universit\'e Lille 1, Laboratoire Paul Painlev\'e, 59655 Villeneuve d'Ascq Cedex, France}
\email{fanny.kassel@math.univ-lille1.fr}

\thanks{J.D. is partially supported by the National Science Foundation under the grant DMS 1103939.
F.G. and F.K. are partially supported by the Agence Nationale de la Recherche under the grants DiscGroup (ANR-11-BS01-013) and ETTT (ANR-09-BLAN-0116-01), and through the Labex CEMPI (ANR-11-LABX-0007-01). The authors also acknowledge support from the GEAR Network, funded by the National Science Foundation under grant numbers DMS 1107452, 1107263, and 1107367 (``RNMS: GEometric structures And Representation varieties").}

\begin{document}

\begin{abstract}
We study \emph{strip deformations} of convex cocompact hyperbolic surfaces, defined by inserting hyperbolic strips along a collection of disjoint geodesic arcs properly embedded in the surface.
We prove that any deformation of the surface that uniformly lengthens all closed geodesics can be realized as a strip deformation, in an essentially unique way.
The infinitesimal version of this result gives a parameterization, by the arc complex, of the moduli space of Margulis spacetimes with fixed convex cocompact linear holonomy.
As an application, we provide a new proof of the tameness of such Margulis spacetimes $M$ by establishing the Crooked Plane Conjecture, which states that $M$ admits a fundamental domain bounded by piecewise linear surfaces called crooked planes.
The noninfinitesimal version gives an analogous theory for complete anti-de Sitter $3$-manifolds.
\end{abstract}

\maketitle
\tableofcontents

\section{Introduction}

The understanding of moduli spaces using simple combinatorial models is a major theme in geometry. 
While coarse models, like the curve complex or pants complex, are used to great effect in the study of the various metrics and compactifications of Teichm\"{u}ller spaces (see \cite{mm99,raf07,bm08,bmns} for instance), parameterizations and/or cellulations can provide insight at both macroscopic and microscopic scales.
One prominent example is Penner's cell decomposition of the decorated Teichm\"{u}ller space of a punctured surface \cite{pen87}, which was generalized in \cite{hazPhD,gl11} and has interesting applications to mapping class groups (see \cite{pen12}).
In this paper we give a parameterization, comparable to Penner's, of the moduli space of certain Lorentzian $3$-manifolds called Margulis spacetimes.

A \emph{Margulis spacetime} is a quotient of the $3$-dimensional Minkowski space $\RR^{2,1}$ by a free group $\Gamma$ acting properly discontinuously by isometries. 
The first examples were constructed by Margulis \cite{mar83,mar84} in 1983, as counterexamples to Milnor's suggestion \cite{mil77} to remove the cocompactness assumption in the Auslander conjecture \cite{aus64}.
Since then many authors, most prominently Charette, Drumm, Goldman, Labourie, and Margulis, have studied their geometry, topology, and deformation theory: see \cite{dru92,dg95,dg99,cg00,gm00,glm09,cdg10,cdg11,cg13}, as well as \cite{dgk13}.
Any Margulis spacetime is determined by a noncompact hyperbolic surface~$S$, with $\pi_1(S) = \Gamma$, and an infinitesimal deformation of~$S$ called a \emph{proper deformation}.
The subset of proper deformations forms a symmetric cone, which we call the \emph{admissible cone}, in the tangent space to the Fricke--Teichm\"uller space of (classes of) complete hyperbolic structures of the same type as~$S$ on the underlying topological surface.
In the case that $S$ is convex cocompact, seminal work of Goldman--Labourie--Margulis \cite{glm09} shows that the admissible cone is open with two opposite, convex components, consisting of the infinitesimal deformations of~$S$ that uniformly expand or uniformly contract the marked length spectrum of~$S$.

In this paper we study a simple geometric construction, called a \emph{strip deformation}, which produces uniformly expanding deformations of~$S$: it is defined by cutting $S$ along finitely many disjoint, properly embedded geodesic arcs, and then gluing in a \emph{hyperbolic strip}, \ie the region between two ultraparallel geodesic lines in~$\HH^2$, at each arc.
An \emph{infinitesimal strip deformation} (Definition~\ref{def:inf-strip-deform}) is the derivative of a path of strip deformations along some fixed arcs as the widths of the strips decrease linearly to zero.
It is easy to see that, as soon as the supporting arcs decompose the surface into disks, an infinitesimal strip deformation lengthens all closed geodesics of~$S$ uniformly (this was observed by Thurston \cite{thu86a} and proved in more detail by Papadopoulos--Th\'eret \cite{pt10}); thus it is a proper deformation.
Our main result (Theorem~\ref{thm:main}) states that all proper deformations of $S$ can be realized as infinitesimal strip deformations, in an essentially unique way: after making some choices about the geometry of the strips, the map from the complex of arc systems on~$S$ to the projectivization of the admissible cone, taking any weighted system of arcs to the corresponding infinitesimal strip deformation, is a homeomorphism.

We note that infinitesimal strip deformations are also used by Goldman--Labourie--Margulis--Minsky in \cite{glmm}.
They construct modified infinitesimal strip deformations along geodesic arcs that accumulate on a geodesic lamination, in order to describe infinitesimal deformations of a surface for which all lengths increase, but not uniformly.

As an application of our main theorem, we give a new proof of the \emph{tameness} of Margulis spacetimes, under the assumption that the associated hyperbolic surface is convex cocompact.
This result was recently established, independently, by Choi--Goldman \cite{cg13} and by the authors \cite{dgk13}.
Here we actually prove the stronger result, named the \emph{Crooked Plane Conjecture} by Drumm--Goldman \cite{dg95}, that any Margulis spacetime admits a fundamental domain bounded by \emph{crooked planes}, piecewise linear surfaces introduced by Drumm \cite{dru92}.
This follows from our main theorem by observing that a strip deformation encodes precise directions for building fundamental domains in~$\RR^{2,1}$ bounded by crooked planes (Section~\ref{subsec:proof-crooked-plane-conj}).
In the case that the free group $\Gamma$ has rank two, the Crooked Plane Conjecture was verified by Charette--Drumm--Goldman \cite{cdg13}.
In particular, when the surface $S$ is a once-holed torus, they found a tiling of the admissible cone according to which triples of isotopy classes of crooked planes embed disjointly; this picture is generalized by our parameterization via strip deformations.

We now state precisely our main results, both in the setting of Margulis spacetimes just discussed, and in the related setting of complete anti-de Sitter $3$-manifolds (Section~\ref{subsec:intro-AdS}).

\subsection{Margulis spacetimes}

The $3$-dimensional Minkowski spa\-ce $\RR^{2,1}$ is the affine space $\RR^3$ endowed with the parallel Lorentzian structure induced by a quadratic form of signature $(2,1)$; its isometry group is $\OO(2,1)\ltimes\RR^3$, acting affinely.
Let $G$ be the group $\PGL_2(\RR)$, acting on the real hyperbolic plane $\HH^2$ by isometries in the usual way, and on the Lie algebra $\g=\mathfrak{pgl}_2(\RR)$ by the adjoint action.
We shall identify $\RR^{2,1}$ with the Lie algebra $\g$ endowed with the Lorentzian structure induced by half its Killing form.
The group of orientation-preserving isometries of~$\RR^{2,1}$ identifies with $G\ltimes\g$, acting on~$\g$ by $(g,w)\cdot v=\Ad(g)v+w$.
Its subgroup preserving the time orientation is $G_0\ltimes\g$, where $G_0=\PSL_2(\RR)$ is the identity component of~$G$.

By \cite{fg83} and \cite{mes90}, if a discrete group $\Gamma$ acts properly discontinuously and freely by isometries on~$\RR^{2,1}$, and if $\Gamma$ is not virtually solvable, then $\Gamma$ is a free group and its action on~$\RR^{2,1}$ is orientation-preserving (see \eg \cite{abe01}) and induces an embedding of~$\Gamma$ into $G\ltimes\g$ with image
\begin{equation}\label{eqn:Gamma-rho-u}
\Gamma^{\rho,u} = \{ (\rho(\gamma),u(\gamma))~|~\gamma\in\Gamma\} \subset G \ltimes \g ,
\end{equation}
where $\rho\in\Hom(\Gamma,G)$ is an injective and  discrete representation and\linebreak $u :\nolinebreak\Gamma\rightarrow\nolinebreak\g$ a $\rho$-cocycle, \ie $u(\gamma_1\gamma_2) = u(\gamma_1) + \Ad(\rho(\gamma_1))\,u(\gamma_2)$ for all $\gamma_1,\gamma_2\in\Gamma$.
By definition, a Margulis spacetime is a manifold $M = \Gamma^{\rho,u}\backslash\RR^{2,1}$ determined by such a proper group action.
Properness is invariant under conjugation by $G\ltimes\g$.
We shall consider conjugate proper actions to be equivalent; in other words, we shall consider Margulis spacetimes to be equivalent if there exists a marked isometry between them.
In particular, we will be interested in holonomies $\rho$ up to conjugacy, \ie as classes in $\Hom(\Gamma, G)/G$, and in $\rho$-cocycles $u$ up to addition of a coboundary, \ie as classes in the cohomology group $H^1_{\rho}(\Gamma,\g) := H^1(\Gamma,\g_{\Ad \rho})$.

Note that for a Margulis spacetime $\Gamma^{\rho,u}\backslash\RR^{2,1}$, the representation~$\rho$ is the holonomy of a noncompact hyperbolic surface $S = \rho(\Gamma)\backslash\HH^2$, and the $\rho$-cocycle $u$ can be interpreted as an infinitesimal deformation of this holonomy, obtained as the derivative at $t=0$ of some smooth path $(\rho_t)_{t\geq 0}$ of representations with $\rho_0=\rho$, in the sense that $\rho_t(\gamma) = e^{tu(\gamma)+o(t)}\rho(\gamma)$ for all $\gamma\in\Gamma$ (see \cite[\S\,2.3]{dgk13} for instance).
Thus the moduli space of Margulis spacetimes projects to the space of noncompact hyperbolic surfaces; describing the fiber above~$S$ amounts to identifying the \emph{proper deformations} $u$ of~$\rho$, \ie the infinitesimal deformations $u$ of~$\rho$ for which the group $\Gamma^{\rho,u}$ acts properly discontinuously on~$\RR^{2,1}$.

A properness criterion was given by Goldman--Labourie--Margulis \cite{glm09}: suitably interpreted \cite{gm00}, it states that for a convex cocompact representation $\rho\in\Hom(\Gamma,G)$ and a $\rho$-cocycle $u : \Gamma\rightarrow\g$, the group $\Gamma^{\rho,u}$ acts properly discontinuously on~$\RR^{2,1}$ if and only if the infinitesimal deformation $u$ ``uniformly lengthens all closed geodesics'', \ie
\begin{equation}\label{eqn:propcritMink}
\inf_{\gamma\in\Gamma\smallsetminus\{ e\} }\ \frac{\D \lambda_{\gamma}(u)}{\lambda_{\gamma}(\rho)} > 0,
\end{equation}
or ``uniformly contracts all closed geodesics'', \ie \eqref{eqn:propcritMink} holds for $-u$ instead of~$u$.
Here $\lambda_{\gamma} : \Hom(\Gamma,G)\rightarrow\RR_+$ is the function (see \eqref{eqn:def-lambda}) assigning to any representation $\tau$ the hyperbolic translation length of $\tau(\gamma)$.
That the injective and discrete representation $\rho$ is convex cocompact means that $\Gamma$ is finitely generated and that $\rho(\Gamma)$ does not contain any parabolic element; equivalently, $S=\rho(\Gamma)\backslash\HH^2$ is the union of a compact convex set (called the convex core), whose preimage in~$\HH^2$ is the smallest nonempty, closed, $\rho(\Gamma)$-invariant, convex subset of~$\HH^2$, and of finitely many ends of infinite volume (called the funnels).
In \cite{dgk13} we gave a new proof of the Goldman--La\-bourie--Margulis criterion, as well as another equivalent properness criterion in terms of expanding (or contracting) equivariant vector fields on~$\HH^2$.
These criteria are to be extended to arbitrary injective and discrete~$\rho$ (for finitely generated~$\Gamma$) in \cite{glm-parab,dgk-parab}, allowing $\rho(\Gamma)$ to have parabolic elements.

We now fix a convex cocompact hyperbolic surface $S$ (possibly nonorientable) with fundamental group $\Gamma=\pi_1(S)$ and holonomy representation $\rho\in\Hom(\Gamma,G)$.
We shall use the following terminology.

\begin{definition}\label{def:Fricke-Teich}
The \emph{Fricke--Teichm\"uller space} $\T\subset\Hom(\Gamma,G)/G$ of~$\rho$ is the set of conjugacy classes of convex cocompact holonomies of hyperbolic structures on the topological surface underlying $S\simeq\rho(\Gamma)\backslash\HH^2$.
Its tangent space $T_{[\rho]}\T$ identifies with the first cohomology group $H^1_{\rho}(\Gamma,\g)$.
\end{definition}

\begin{definition}\label{def:GLMcone}
The \emph{positive admissible cone} in $T_{[\rho]}\T\simeq H^1_{\rho}(\Gamma,\g)$ is the subset of classes of $\rho$-cocycles $u$ satisfying \eqref{eqn:propcritMink}.
The \emph{admissible cone} is the union of the positive admissible cone and of its opposite.
The projectivization of the admissible cone, a subset of $\PP(H^1_{\rho}(\Gamma,\g))$, will be denoted~$\adm(\rho)$.
\end{definition}

\noindent
The positive admissible cone is an open, convex cone in the finite-dimensional vector space $T_{[\rho]}\T\simeq H^1_{\rho}(\Gamma,\g)$.

We now describe the fundamental objects of the paper, namely strip deformations, which will be used to parameterize $\adm(\rho)$ (Theorem~\ref{thm:main}).

\subsection{The arc complex and strip deformations}\label{subsec:arc-complex}

We call \emph{arc} of~$S$ any nontrivial isotopy class of embedded lines in~$S$ for which each end exits in a funnel; we shall denote by $\mathscr{A}$ the set of arcs of~$S$.
A \emph{geodesic arc} is a geodesic representative of an arc. 
The following notion was first introduced by Thurston \cite[proof of Lem.\,3.4]{thu86a}.

\begin{definition}\label{def:strip-deform}
A \emph{strip deformation} of the hyperbolic surface~$S$ along a geodesic arc $\underline{\alpha}$ is a new hyperbolic surface that is obtained from~$S$ by cutting along~$\underline{\alpha}$ and gluing in (without any shearing) a \emph{strip}, the region in~$\HH^2$ bounded by two ultraparallel geodesics.
A strip deformation of~$S$ along a collection of pairwise disjoint and nonisotopic geodesic arcs $\underline{\alpha}_0, \ldots, \underline{\alpha}_k$ is a hyperbolic surface obtained by simultaneously performing this operation for each geodesic arc $\underline{\alpha}_i$, where $0\leq i\leq k$.
(Note that the operations commute since the $\underline{\alpha}_i$ are disjoint.)
We shall also say that the holonomy representation of the resulting surface (defined up to conjugation) is a strip deformation of the holonomy representation $\rho$ of~$S$.
\end{definition}

The nonshearing condition in Definition~\ref{def:strip-deform} means that the strip at the arc~$\underline{\alpha}_i$ is inserted so that the two endpoints of the most narrow cross section of the strip are identified with the two preimages of a single point $p_{\alpha_i}\in\underline{\alpha}_i$ (see Figure~\ref{fig:strip}).
This point $p_{\alpha_i}\in\underline{\alpha}_i$ is called the \emph{waist} of the strip.
The thickness of the strip at its most narrow cross section is called the \emph{width} of the strip.
In the above definition, the waist and width of each strip may be chosen arbitrarily.

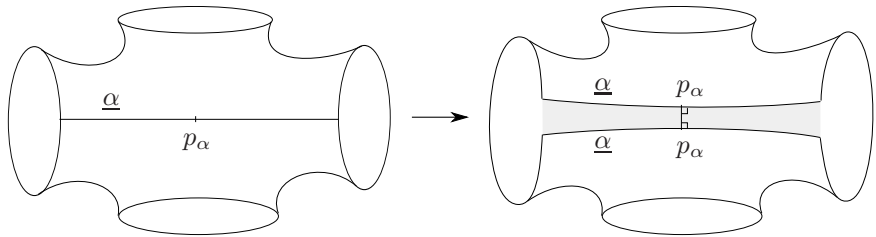
\begin{figure}[ht!]
\centering
\labellist
\small\hair 2pt
\pinlabel $p_{\alpha}$ [u] at 183 90
\pinlabel $p_{\alpha}$ [t] at 660 95
\pinlabel $p_{\alpha}$ [b] at 660 128
\pinlabel $\underline{\alpha}$ [b] at 100 115
\pinlabel $\underline{\alpha}$ [b] at 575 130
\pinlabel $\underline{\alpha}$ [t] at 575 100
\endlabellist
\includegraphics[scale=0.4]{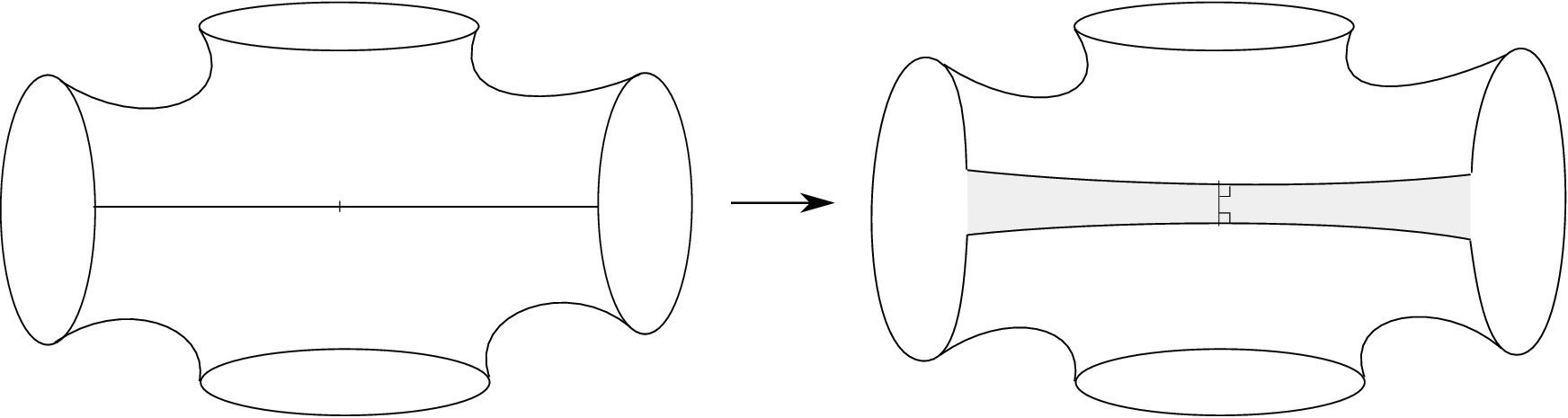}
\caption{A strip deformation along one arc in a four-holed sphere}
\label{fig:strip}
\end{figure}

We shall also use the infinitesimal version of this construction:

\begin{definition}\label{def:inf-strip-deform}
An \emph{infinitesimal strip deformation} of~$S$ is the class in $H^1_{\rho}(\Gamma,\g)$ of a $\rho$-cocycle $u : \Gamma\rightarrow\g$ obtained as the derivative at $t=0$ of a path $t\mapsto \rho_t \in\Hom(\Gamma,G)$ of strip deformations of $\rho_0=\rho$, along fixed geodesic arcs $\underline{\alpha}_0,\dots,\underline{\alpha}_k$, with fixed waists, and such that the widths of the strips, measured at the waists, are of the form $m_i t$ for some fixed numbers $m_0,\dots,m_k>0$; these numbers are called the \emph{widths} of the infinitesimal strip deformation.
\end{definition}

Our parameterization of the admissible cone by strip deformations depends on certain choices: for each arc $\alpha\in\mathscr{A}$, we fix
\begin{itemize}
  \item a geodesic representative $\underline{\alpha}$ of~$\alpha$,
  \item a point $p_{\alpha}\in\underline{\alpha}$ (the \emph{waist}),
  \item a positive number $m_{\alpha}>0$ (the \emph{width}).
\end{itemize}
We require that the $\underline{\alpha}$ intersect minimally, meaning that the representatives $\underline{\alpha}_1$ and $\underline{\alpha}_2$ of two arcs $\alpha_1$ and~$\alpha_2$ always have smallest possible intersection number (including ideal intersection points).
This can be achieved by choosing the representatives $\underline{\alpha}$ to intersect the boundary of the convex core orthogonally, but we do not require this.

For any arc $\alpha\in\mathscr{A}$, we define ${\boldsymbol f}(\alpha)\in H^1_{\rho}(\Gamma,\g)$ to be the infinitesimal strip deformation of~$\rho$ along~$\underline{\alpha}$ with waist~$p_{\alpha}$ and width~$m_{\alpha}$.
Recall that the \emph{arc complex} $\overline{X}$ of~$S$ is the simplicial complex with vertex set~$\mathscr{A}$ and with one $k$-dimensional simplex for each collection of $k+1$ pairwise homotopically disjoint arcs.
Top-dimensional cells of $\overline{X}$ correspond to so-called \emph{hyperideal triangulations} of~$S$ (see Section~\ref{subsec:angle-convex-core}).
The map $\alpha\mapsto{\boldsymbol f}(\alpha)$ extends by barycentric interpolation to a map ${\boldsymbol f}:\overline{X}\rightarrow H^1_{\rho}(\Gamma,\g)$.
By postcomposing with the projectivization map $H^1_{\rho}(\Gamma,\g)\rightarrow\PP(H^1_{\rho}(\Gamma,\g))$, we obtain a map
$$f : \overline{X} \longrightarrow \PP\big(H^1_{\rho}(\Gamma,\g)\big).$$
Let $X$ be the complex of \emph{arc systems} of~$S$, \ie the subset of~$\overline{X}$ obtained by removing all cells corresponding to collections of arcs that do \emph{not} subdivide the surface $S$ into topological disks.
For instance, no vertex of~$\overline{X}$ is in~$X$, but the interior of any top-dimensional cell is (see Section~\ref{sec:ex} for more examples).
By work of Penner \cite{pen87} on the decorated Teichm\"uller space, $X$ is homeomorphic to an open ball of dimension $3 |\chi|-1=\dim(\T)-1$, where $\chi$ is the Euler characteristic of~$S$.
Our main result is that any point of the positive admissible cone is realized as an infinitesimal strip deformation, in a \emph{unique} way given our choice of $(\underline{\alpha},p_{\alpha},m_{\alpha})_{\alpha\in\mathscr{A}}$:

\begin{theorem}\label{thm:main}
The map $f$ restricts to a homeomorphism between $X$ and the projectivized admissible cone $\adm(\rho)$.
\end{theorem}

It is natural to wonder about the image of ${\boldsymbol f}$ in $H^1_{\rho}(\Gamma,\g)$, before projectivization.
Since $\adm(\rho)=f(X)$ is convex, it seems reasonable to hope that ${\boldsymbol f}(X)$ should appear as the boundary of a convex object in $H^1_{\rho}(\Gamma,\g)$: thus the following conjecture would provide a concrete realization of $X$ as part of the boundary of the convex hull of a natural discrete subset in a finite-dimensional vector space.

\begin{conjecture} \label{conj}
There exists a choice of minimally intersecting geodesic representatives~$\underline{\alpha}$ and waists~$p_{\alpha}$, for $\alpha\in\mathscr{A}$, such that if all the widths $m_{\alpha}$ equal~$1$, then ${\boldsymbol f}(X)$ is a convex hypersurface in $H^1_{\rho}(\Gamma,\g)$.
\end{conjecture}

\subsection{Fundamental domains for Margulis spacetimes}

In 1992, Drumm \cite{dru92} introduced piecewise linear surfaces in the $3$-dimensional Minkowski space $\RR^{2,1}$ called \emph{crooked planes} (see \cite{cg00}).
The \emph{Crooked Plane Conjecture} of Drumm--Goldman states that any Margulis spacetime should admit a fundamental domain in~$\RR^{2,1}$ bounded by finitely many crooked planes.
Charette--Drumm--Goldman \cite{cdg10,cdg13} proved this conjecture in the special case that the fundamental group is a free group of rank two.
Here we give a proof of the Crooked Plane Conjecture in the general case that the linear holonomy is convex cocompact. 

\begin{theorem}\label{thm:crooked-planes}
Any discrete subgroup of $\OO(2,1)\ltimes\RR^3$ acting properly discontinuously and freely on~$\RR^{2,1}$, with convex cocompact linear part, admits a fundamental domain in~$\RR^{2,1}$ bounded by finitely many crooked planes.
\end{theorem}

This is an easy consequence of Theorem~\ref{thm:main}: the idea is to interpret infinitesimal strip deformations as motions of crooked planes making them disjoint in~$\RR^{2,1}$ (see Section~\ref{subsec:proof-crooked-plane-conj}).
Theorem~\ref{thm:crooked-planes} provides a new proof of the tameness of Margulis spacetimes with convex cocompact linear holonomy, independent from the original proofs given in \cite{cg13,dgk13}.

\subsection{Strip deformations and anti-de Sitter $3$-manifolds}\label{subsec:intro-AdS}

In \cite{dgk13} we showed that, in a precise sense, Margulis spacetimes behave like ``infinitesimal analogues'' or ``renormalized limits'' of complete $\AdSS$ manifolds, which are quotients of the negatively-curved \emph{anti-de Sitter space} $\AdS$.
Following this point of view further, we now derive analogues of Theorems \ref{thm:main} and~\ref{thm:crooked-planes} for $\AdSS$ manifolds.

The anti-de Sitter space $\AdS = \PO(2,2)/\OO(2,1)$ is a model space for Lorentzian manifolds of constant negative curvature.
It can be realized as the set of negative points in $\PP^3(\RR)$ with respect to a quadratic form of signature $(2,2)$; its isometry group is $\PO(2,2)$.
Equivalently, $\AdS$ can be realized as the identity component $G_0=\PSL_2(\RR)$ of $G=\PGL_2(\RR)$, endowed with the biinvariant Lorentzian structure induced by half the Killing form of $\g=\mathfrak{pgl}_2(\RR)=\mathfrak{psl}_2(\RR)$; the group of orientation-preserving isometries then identifies with
$$(G\times G)_+ := \{ (g_1,g_2)\in G\times G~|~g_1g_2\in G_0\} ,$$
acting on~$G_0$ by right and left multiplication: $(g_1,g_2)\cdot g=g_2gg_1^{-1}$.

By \cite{kr85}, any torsion-free discrete subgroup of $(G\times G)_+$ acting properly~dis\-continuously on $\AdS$ is, up to switching the two factors of $G\times G$, of the form
$$\Gamma^{\rho,j} = \{ (\rho(\gamma), j(\gamma))~|~\gamma\in\Gamma\} \subset G\times G$$
where $\Gamma$ is a discrete group and $\rho,j \in\Hom(\Gamma,G)$ are two representations with $j$ injective and discrete.
Suppose that $\Gamma$ is finitely generated.
By \cite{kasPhD,gk13}, a necessary and sufficient condition for the action of $\Gamma^{\rho,j}$ on $\AdS$ to be properly discontinuous is that (up to switching the two factors) $j$ be injective and discrete and $\rho$ be ``uniformly contracting'' with respect to~$j$, in the sense that there exists a $(j,\rho)$-equivariant Lipschitz map $\HH^2\rightarrow\HH^2$ with Lipschitz constant $<1$, or equivalently that
\begin{equation}\label{eqn:propcritAdS}
\inf_{\gamma\in\Gamma\smallsetminus\{ e\} }\ \frac{\lambda_{\gamma}(j)}{\lambda_{\gamma}(\rho)} > 1,
\end{equation}
where $\lambda_{\gamma} : \Hom(\Gamma,G)\rightarrow\RR_+$ is the hyperbolic translation length function of~$\gamma$ as above, see \eqref{eqn:def-lambda}.
One should view \eqref{eqn:propcritMink} as the derivative of \eqref{eqn:propcritAdS} as $j$ tends to~$\rho$ with derivative~$u$.
If $\Gamma$ is the fundamental group of a compact~sur\-face and both $\rho,j$ are injective and discrete, then \eqref{eqn:propcritAdS} is never satisfied \cite{thu86a}.

Suppose that $\rho$ is convex cocompact, of infinite covolume, and let $S$ be the hyperbolic surface $\rho(\Gamma)\backslash\HH^2$, with holonomy~$\rho$.
Let $\Adm(\rho)$ be the subset of the Fricke--Teichm\"uller space $\T$ of~$\rho$ (Definition~\ref{def:Fricke-Teich}) consisting of classes of convex cocompact representations $j\in\Hom(\Gamma,G)$ that are ``uniformly longer'' than~$\rho$, namely that satisfy \eqref{eqn:propcritAdS}.
As in Section~\ref{subsec:arc-complex}, for each arc $\alpha\in\mathscr{A}$ of~$S$ we fix a geodesic representative $\underline{\alpha}$ of~$\alpha$, a point $p_{\alpha}\in\underline{\alpha}$, and a positive number $m_{\alpha}>0$, and we require that the $\underline{\alpha}$ intersect minimally.
Let $\F(\alpha)\in\T$ be the class of the strip deformation of~$\rho$ along~$\underline{\alpha}$ with waist~$p_{\alpha}$ and width~$m_{\alpha}$.
Since the vertices of a cell of the arc complex $\overline{X}$ correspond to disjoint arcs, the cut-and-paste operations along them do not interfere and the map $\alpha\mapsto\F(\alpha)$ naturally extends to a map $\F : \overline{\mathrm{C}X}\rightarrow\T$, where $\overline{\mathrm{C}X}$ is the abstract cone over the arc complex~$\overline{X}$, with the property that
\begin{equation}\label{eqn:f-deriv-F}
{\boldsymbol f}(x) = \frac{\D}{\D t}\Big|_{t=0}\, F(tx) \,\in\, T_{[\rho]}\T \,\simeq\, H^1_{\rho}(\Gamma,\g)
\end{equation}
for all $x\in\overline{X}$.
Recall that $\overline{\mathrm{C}X}$ is the quotient of $\overline{X}\times\RR_+$ by the equivalence relation $(x,0)\sim (x',0)$ for all $x,x'\in\overline{X}$: we abbreviate $(x,t)$ as $tx$.
Let $\mathrm{C}X\subset\overline{\mathrm{C}X}$ be the abstract \emph{open} cone over~$X$, equal to the image of $X\times\RR_+^{\ast}$.
We prove the following ``macroscopic'' version of Theorem~\ref{thm:main}.

\begin{theorem}\label{thm:main-macro}
For convex cocompact~$\rho$, the map $\F$ restricts to a homeomorphism between $\mathrm{C}X$ and $\Adm(\rho)$.
\end{theorem}
\noindent In other words, any ``uniformly lengthening'' deformation of~$\rho$ can be realized as a strip deformation, and the realization is unique once the geodesic representatives~$\underline{\alpha}$, waists~$p_{\alpha}$, and widths~$m_{\alpha}$ are fixed for all arcs $\alpha\in\mathscr{A}$.

Note that the situation is very different when $\Gamma$ is the fundamental group of a compact surface: as mentioned above, in this case $j$ is Fuchsian and $\rho$ is necessarily non-Fuchsian \cite{thu86a}, up to switching the two factors.
As proved independently in \cite{gkw13} and \cite{dt13}, the subset $\Adm(\rho)$ of the Fricke--Teichm\"uller space (\ie the classical Teichm\"uller space in the orientable case) consisting of representations $j$ ``uniformly longer'' than~$\rho$ is always nonempty.
It would be interesting to obtain a parameterization of $\Adm(\rho)$ by some simple combinatorial object in this situation as well.
See the recent paper \cite{tho14} for an approach via harmonic maps.

By analogy with the Minkowski setting, it is natural to ask whether a free, properly discontinuous action on $\AdS$ admits a fundamental domain bounded by nice polyhedral surfaces.
In Section~\ref{sec:AdS}, we introduce piecewise geodesic surfaces in $\AdS$ that we call \emph{AdS crooked planes}, and prove:

\begin{theorem}\label{thm:AdS-crooked-planes}
Let $\rho,j\in\Hom(\Gamma,G)$ be the holonomy representations of two convex cocompact hyperbolic structures on a fixed surface and assume that $\Gamma^{\rho,j}$ acts properly discontinuously on $\AdS$.
Then $\Gamma^{\rho,j}$ admits a fundamental domain bounded by finitely many $\AdSS$ crooked planes.
\end{theorem}

Theorem~\ref{thm:AdS-crooked-planes} provides a new proof of the tameness (obtained in \cite{dgk13}) of complete $\AdSS$ $3$-manifolds of finite type, in this special case.

In contrast with Theorem~\ref{thm:AdS-crooked-planes}, in \cite{dgk-crooked-AdS} we construct examples of pairs $(\rho,j)$ with $j$ convex cocompact and $\rho$ noninjective or nondiscrete such that the group $\Gamma^{\rho,j}$ acts properly discontinuously on $\AdS$ but does \emph{not} admit any fundamental domain bounded by disjoint crooked planes.
It would be interesting to determine exactly which proper actions admit such fundamental domains.
The examples of \cite{dgk-crooked-AdS} build on a disjointness criterion that we establish there for $\AdSS$ crooked planes (see Proposition~\ref{prop:disjoint-CP-AdS} of the current paper for a sufficient condition).

\subsection{Organization of the paper}

Section~\ref{sec:metric-estim} is devoted to some basic estimates for infinitesimal strip deformations.
These allow us, in Section~\ref{sec:reduction}, to reduce the proofs of Theorems~\ref{thm:main} and~\ref{thm:main-macro} to Claim~\ref{claim:main}, about the behavior of the map $f$ at faces of codimension zero and one.
We prove Claim~\ref{claim:main} in Section~\ref{sec:codim01}, after introducing some formalism in Section~\ref{sec:formalism}.
In Section~\ref{sec:ex} we give some basic examples of the tiling of the admissible cone produced by Theorem~\ref{thm:main}.
Finally, Sections \ref{sec:crooked-planes} and~\ref{sec:AdS} are devoted to the proofs of Theorems~\ref{thm:crooked-planes} (the Crooked Plane Conjecture) and~\ref{thm:AdS-crooked-planes} (its anti-de Sitter counterpart) using strip deformations.
In Appendix~\ref{sec:remarks} we make some remarks about the choices involved in the definition of the map~$f$, in relation with Conjecture~\ref{conj}; this appendix is not needed anywhere in the paper.

\subsection*{Acknowledgments}

We would like to thank Thierry Barbot, Virginie Cha\-rette, Todd Drumm, and Bill Goldman for interesting discussions related to this work, as well as Fran\c{c}ois Labourie and Yair Minsky for igniting remarks.
We are grateful to the Institut Henri Poincar\'e in Paris and to the Centre de Recherches Math\'ematiques in Montreal for giving us the opportunity to work together in stimulating environments.

\section{Metric estimates for (infinitesimal) strip deformations}\label{sec:metric-estim}

In this section we provide some estimates for the effect, on curve lengths, of a (possibly infinitesimal) strip deformation, especially when supported on long arcs.

Let $S$ be a convex cocompact hyperbolic surface of infinite volume, with fundamental group $\Gamma=\pi_1(S)$ and holonomy representation $\rho\in\Hom(\Gamma,G)$.
Let $\T\subset\Hom(\Gamma,G)/G$ be the corresponding Fricke--Teichm\"uller space (Definition~\ref{def:Fricke-Teich}), whose tangent space $T_{[\rho]}\T$ identifies with the cohomology group $H^1_{\rho}(\Gamma,\g)$.
For any $\gamma\in\Gamma$ and any $\tau\in\Hom(\Gamma,G)$ we set
\begin{equation}\label{eqn:def-lambda}
\lambda_{\gamma}(\tau) := \inf_{p\in\HH^2} \, d(p,\tau(\gamma)\cdot p),
\end{equation}
where $d$ is the hyperbolic metric on~$\HH^2$: this is the translation length of $\tau(\gamma)$ if $\tau(\gamma)\in G$ is hyperbolic, and $0$ otherwise.
We thus obtain a function $\lambda_{\gamma} : \T\rightarrow\RR_+$, whose differential is denoted $\D\lambda_{\gamma} : T\T\rightarrow\RR$.

As in Section~\ref{subsec:arc-complex}, for each arc $\alpha\in\mathscr{A}$ of~$S$ we fix a geodesic representative $\underline{\alpha}$ of~$\alpha$, a point $p_{\alpha}\in\underline{\alpha}$, and a positive number $m_{\alpha}>0$, such that the lifts of the $\underline{\alpha}$ have minimal intersection numbers in~$\HH^2 \cup \partial_{\infty}\HH^2$.

\begin{remark}\label{rem:connected-choices}
The space of all such choices of $(\underline{\alpha},p_{\alpha},m_{\alpha})_{\alpha\in\mathscr{A}}$ is connected.
\end{remark}

\noindent
Indeed, one system of minimally intersecting geodesic representatives of the arcs is given by the geodesics orthogonal to the boundary of the convex core, and any other system is obtained from this one by pushing the endpoints at infinity of the geodesics forward by an isotopy of the circles at infinity of the funnels.

Let $\overline{X}$ be the arc complex of~$S$.
The goal of this section is to set up some notation and establish estimates for the map ${\boldsymbol f} : \overline{X}\rightarrow H^1_{\rho}(\Gamma,\g)$ of Section~\ref{subsec:arc-complex} and its projectivization $f : \overline{X}\rightarrow\PP(H^1_{\rho}(\Gamma,\g))$, which realize infinitesimal strip deformations with respect to our choice of $(\underline{\alpha},p_{\alpha},m_{\alpha})_{\alpha\in\mathscr{A}}$.

\subsection{Variation of length of geodesics under strip deformations}\label{subsec:var-length}

For $x\in\overline{X}$, we denote by $|x|\subset S$ the \emph{support} of~$x$, \ie the union of the geodesic arcs $\underline{\alpha}$ corresponding to the vertices of the smallest cell of~$\overline{X}$ containing~$x$.
Let $W_x: |x| \to \RR_+$ be the \emph{strip width function} mapping any $p\in\underline{\alpha}\subset |x|$ to
\begin{equation}\label{eqn:strip-width-function}
W_x(p) := w_{\alpha} \cosh d(p,p_{\alpha}),
\end{equation}
where $w_{\alpha}$ is the width (as measured at the narrowest point) of the strip to be inserted along~$\underline{\alpha}$, and the distance $d(\cdot, \cdot)$ is implicitly measured along~$\underline{\alpha}$.
Note that $w_{\alpha}$ is $m_{\alpha}$ times the weight of $\alpha$ in the barycentric expression for $x\in\overline{X}$, and $W_x(p)$ is the length of the path crossing the strip at~$p$ at constant distance from the waist segment.
For $p\in \HH^2$, we denote by $\measuredangle_p$ the measure of angles of geodesics at~$p$, valued in $[0,\pi/2]$.
We first make the following elementary observation.

\begin{observation}\label{obs:length-deriv}
For any $\gamma\in\Gamma\smallsetminus\{ e\}$ and any $x\in\overline{X}$,
\begin{equation}\label{eqn:effect}
\D\lambda_{\gamma}\big({\boldsymbol f}(x)\big) = \sum_{p\in \underline{\smash{\gamma}}\cap |x|} W_x(p) \, \sin \measuredangle_p(\underline{\smash{\gamma}},|x|) \, \geq 0,
\end{equation}
where $\underline{\smash{\gamma}}$ is the geodesic representative of $\gamma$ on~$S$.
\end{observation}

Formula \eqref{eqn:effect} is analogous to the cosine formula expressing the effect of an earthquake on the length of a closed geodesic (see \cite{ker83}), and is proved similarly.
The difference is that the angle $\measuredangle_p(\underline{\smash{\gamma}},|x|)$ in the formula for earthquakes is replaced by its complement to $\pi/2$ in the formula for strip deformations, changing the cosine into a sine.
In general, strip deformations of~$\rho$ should be thought of as analogues of earthquakes, where instead of sliding against itself, the surface is pushed apart in a direction orthogonal to each geodesic arc of the support.
In \eqref{eqn:effect}, unlike in the formula for earthquakes, the contribution of each intersection point $p$ depends, not only on the angle, but also (via $W_x$) on $p$ itself.

\begin{proof}[Proof of Observation~\ref{obs:length-deriv}]
Up to passing to a double covering, we may assume that $S$ is orientable and $\rho\in\Hom(\Gamma,G)$ takes values in $G_0=\PSL_2(\RR)$.
By linearity, it is sufficient to prove the formula when $x$ is a vertex $\alpha\in\mathscr{A}$ of the arc complex~$\overline{X}$ and the corresponding strip width $w_{\alpha}$ is~$1$.
For $t\in\RR$, we set
$$a_t: = \! \begin{pmatrix} e^{t/2} & 0 \\ 0 & e^{-t/2}\end{pmatrix} ,\
b_t := \! \begin{pmatrix} \cosh\frac{t}{2} & \sinh\frac{t}{2} \\ \sinh\frac{t}{2} & \cosh\frac{t}{2} \end{pmatrix} ,\
r_t := \! \begin{pmatrix} \;\;\, \cos\frac{t}{2} & \sin\frac{t}{2} \\ -\sin\frac{t}{2} & \cos\frac{t}{2} \end{pmatrix} ,$$
where all matrices are understood to be elements of $G_0=\PSL_2(\RR)$.
Up to conjugation we may assume that $\rho(\gamma)=a_{\lambda_{\gamma}(\rho)}$.
Suppose the oriented geodesic loop $\underline{\smash{\gamma}}$ crosses the geodesic representative $\underline{\alpha}$ at points $q_1,\dots,q_k$, in this order.
For $1\leq i \leq k$, we use the following notation:
\begin{itemize}
  \item $\ell_i>0$ is the distance between $q_{i-1}$ and $q_i$ along~$\underline{\smash{\gamma}}$, with the convention that $q_0=q_k$,
  \item $d_i\in\RR$ is the signed distance from $q_i$ to the waist $p_{\alpha}$ along~$\underline{\alpha}$, for the orientation of $\underline{\alpha}$ towards the left of $\underline{\smash{\gamma}}$ at~$q_i$,
  \item $\theta_i\in (0,\pi)$ is the angle, at~$q_i$, between the oriented geodesics $\underline{\smash{\gamma}}$ and~$\underline{\alpha}$, for the orientation of~$\underline{\alpha}$ towards the left of $\underline{\smash{\gamma}}$ at~$q_i$.
\end{itemize}
Consider the map $F : \overline{\mathrm{C}X}\rightarrow\T\subset\Hom(\Gamma,G)/G$ of Section~\ref{subsec:intro-AdS}.
Then $F(tx)$ lifts to a homomorphism $\Gamma\rightarrow G_0$ sending $\gamma$ to
\begin{equation}\label{eqn:F(tx)(gamma)}
a_{\ell_1} \big(r_{\theta_1} a_{d_1} \, b_t \, a_{-d_1} r_{-\theta_1}\big) \, a_{\ell_2} (\dots) \, a_{\ell_k} \big(r_{\theta_k} a_{d_k} \, b_t \, a_{-d_k} r_{-\theta_k}\big) \in G_0 .
\end{equation}
Note that, ignoring nondiagonal entries,
$$\left . \frac{\D}{\D t} \right |_{t=0} \big(r_{\theta_i} a_{d_i} \, b_t \, a_{-d_i} r_{-\theta_i}\big) = \frac{\sin\theta_i \cdot \cosh d_i}{2} \begin{pmatrix} 1 & * \\ * & - 1 \end{pmatrix} .$$
Therefore, if we set $\ell:=\ell_1+\dots+\ell_k=\lambda_{\gamma}(\rho)$, then \eqref{eqn:F(tx)(gamma)} is equal to
\begin{eqnarray*}
& \rho(\gamma) + t \left( \sum_{i=1}^k \textstyle{\frac{\sin \theta_i \cdot \cosh d_i}{2}} \, a_{\ell_1+\dots+\ell_i} \begin{pmatrix} 1 & * \\ * & -1 \end{pmatrix} \, a_{\ell_{i+1}+\dots+\ell_k} \right) \! + o(t) \\
= & \begin{pmatrix} e^{\ell/2} & 0 \\ 0 & e^{-\ell/2} \end{pmatrix} + \frac{t}{2} \left( \sum_{i=1}^k \sin \theta_i \cdot \cosh d_i  \begin{pmatrix} e^{\ell/2} & * \\ * & -e^{-\ell/2} \end{pmatrix} \right) + o(t) .
\end{eqnarray*}
By differentiating the formula $\lambda(g)=2\,\mathrm{arccosh}(|\mathrm{tr}(g)|/2)$ for hyperbolic $g\in G_0=\PSL_2(\RR)$ (where $\lambda(g)$ is the translation length of $g$ in~$\HH^2$), we find
$$\D\lambda_{\gamma}\big( {\boldsymbol f}(x)\big) = \sum_{i=1}^k  \sin \theta_i \cdot \cosh d_i = \sum_{i=1}^k W_x(q_i) \, \sin \measuredangle_{q_i}(\underline{\smash{\gamma}},|x|).\qedhere$$
\end{proof}

\subsection{Angles at the boundary of the convex core}\label{subsec:angle-convex-core}

Let $\Delta\subset\mathscr{A}$ be a hyperideal triangulation of~$S$, \ie a set of $3|\chi|$ arcs corresponding to a top-dimensional face of the arc complex~$\overline{X}$ (here $\chi\in -\NN$ is the Euler characteristic of~$S$).
Then $\Delta$ divides $S$ into $2|\chi|$ connected components (hyperideal triangles).
Let $\partial S$ denote the geodesic boundary of the convex core of~$S$.

\begin{proposition}\label{prop:noslant}
For any choice of minimally intersecting geodesic representatives $(\underline{\alpha})_{\alpha\in\mathscr{A}}$, there exists $\theta_0>0$ such that all the $\underline{\alpha}$ intersect $\partial S$ at an angle $\geq\theta_0$ (measured in $[0,\pi/2]$).
Moreover, $\theta_0$ can be taken to depend continuously on the holonomy~$\rho$ and on the choice of geodesic representatives of the arcs of any fixed hyperideal triangulation $\Delta$ of~$S$.
\end{proposition}

\begin{proof}
This is a consequence of the minimal intersection numbers of the~$\underline{\alpha}$.
Indeed, fix a hyperideal triangulation $\Delta$ of~$S$, consider one component $\underline{\smash{\eta}}$ of $\partial S$, and choose one geodesic arc $\underline{\smash{\beta}}$ of $\Delta$ that exits $\underline{\smash{\eta}}$ at one end.
In the universal cover~$\HH^2$, a lift $\tilde{\eta}$ of~$\underline{\smash{\eta}}$ is intersected by a collection $\{ \tilde{\beta}_i\}_{i\in\ZZ}$ of naturally ordered, pairwise disjoint lifts of~$\underline{\smash{\beta}}$.
Each $\tilde{\beta}_i$ escapes~through two components ($\tilde{\eta}$ and another one) of the lift of $\partial S$; let $\beta^*_i$ denote the geodesic arc orthogonal to these two components.
The $\beta^*_i$ are lifts of the geodesic arc $\beta^*$ of~$S$ which is in the same class as~$\underline{\smash{\beta}}$ but intersects $\partial S$ at right angles.

Consider another geodesic arc $\underline{\alpha}\neq\underline{\smash{\beta}}$ of~$S$, which also crosses~$\underline{\smash{\eta}}$.
Let~$\alpha^*$ be the geodesic representative of $\alpha$ that is orthogonal to $\partial S$.
In~$\HH^2$, a lift of $\alpha^*$ that intersects $\tilde{\eta}$ lies entirely between $\beta^*_i$ and $\beta^*_{i+1}$, for some $i\in \ZZ$.
By minimality of the intersection numbers, the corresponding lift of $\underline{\alpha}$ lies entirely between $\tilde{\beta}_i$ and~$\tilde{\beta}_{i+1}$.
Since the angles at which the $\tilde{\beta}_i$ intersect~$\tilde{\eta}$ are all the same, the angle at which $\underline{\alpha}$ intersects~$\underline{\smash{\eta}}$ is bounded away from~$0$, independently of~$\alpha$.
We conclude by repeating for all boundary components~$\underline{\smash{\eta}}$.
\end{proof}

\subsection{The unit-peripheral normalization}

Given a choice $(\underline{\alpha},p_{\alpha})_{\alpha\in\mathscr{A}}$ of geodesic representatives and waists of the arcs of~$S$, we now discuss a specific choice of widths $(m_{\alpha})_{\alpha\in\mathscr{A}}$, which we call the unit-peripheral normalization.

\begin{remark}\label{rem:change-m-alpha}
For any systems $(m_{\alpha})_{\alpha\in\mathscr{A}}$ and $(m'_{\alpha})_{\alpha\in\mathscr{A}}$ of widths, there is a cell-preserving, cellwise projective homeomorphism $h : \overline{X}\rightarrow\overline{X}$ such that
$$f_m = f_{m'} \circ h,$$
where $f_m : \overline{X}\rightarrow\PP(H^1_{\rho}(\Gamma,\g))$ (\resp $f_{m'} : \overline{X}\rightarrow\PP(H^1_{\rho}(\Gamma,\g))$) is the map defined in Section~\ref{subsec:arc-complex} with respect to $(\underline{\alpha},p_{\alpha},m_{\alpha})_{\alpha\in\mathscr{A}}$ (\resp to $(\underline{\alpha},p_{\alpha},m'_{\alpha})_{\alpha\in\mathscr{A}}$). Such a map~$h$ preserves~$X$.
\end{remark}

Choosing all the widths $m_{\alpha}$ so that
\begin{equation}\label{eqn:normal}
\sum_{p \in |x| \cap \partial S} W_x(p) \, \sin \measuredangle_p(\partial S,|x|) = 1
\end{equation}
for all $x\in\overline{X}$ (or equivalently for all vertices $x=\alpha\in\mathscr{A}$) corresponds,~by Observation~\ref{obs:length-deriv}, to asking all infinitesimal strip deformations associated~with the choice of $(\underline{\alpha},p_{\alpha},m_{\alpha})_{\alpha\in\mathscr{A}}$ to increase the total length of $\partial S$ at unit rate.
In this case ${\boldsymbol f}(\overline{X})$ is contained in some affine hyperplane of $T_{[\rho]}{\T}=H^1_{\rho}(\Gamma,\g)$ (namely $\sum_{i=1}^m \D\lambda_{\gamma_i}=1$ where $\gamma_1,\dots,\gamma_m\in\Gamma$ correspond to the connected components of $\partial S$).

\begin{definition}\label{def:unit-peripheral-norm}
We call \eqref{eqn:normal} the \emph{unit-peripheral normalization}.
\end{definition}

In the unit-peripheral normalization, by \eqref{eqn:effect}, the bound $\theta_0$ of Proposition~\ref{prop:noslant} satisfies
\begin{equation}\label{eqn:boundwidth}
W_x(p) \leq 1/\sin \theta_0
\end{equation}
for all $x\in\overline{X}$ and $p\in\partial S$.
By convexity of $\cosh$ in the formula \eqref{eqn:strip-width-function}, the inequality \eqref{eqn:boundwidth} holds in fact for all $p$ in the convex core of~$S$.

In Section~\ref{subsec:f-bounded} we shall bound the length variation of closed geodesics under infinitesimal strip deformations using Observation~\ref{obs:length-deriv} and the following.

\begin{proposition}\label{prop:technical}
In the unit-peripheral normalization, there exists $K>0$ (depending on~$\rho$ and on the choice of geodesic representatives $(\underline{\alpha})_{\alpha\in\mathscr{A}}$ and waists $p_{\alpha}\in\underline{\alpha}$) such that for any $\gamma\in\Gamma\smallsetminus\{ e\}$ and any $x\in\overline{X}$,
$$\sum_{p\in\underline{\smash{\gamma}}\cap |x|} W_x(p) \leq K \, \vartheta_{|x|}(\gamma) \, \lambda_{\gamma}(\rho),$$
where $\underline{\smash{\gamma}}$ is the closed geodesic of~$S$ corresponding to~$\gamma$ and
$$\vartheta_{|x|}(\gamma) := \max_{p\in\underline{\smash{\gamma}}\cap |x|} \, \measuredangle_p(\underline{\smash{\gamma}}, |x|).$$
\end{proposition}

\begin{proof}
Fix $\gamma\in\Gamma\smallsetminus\{ e\}$.
Any lift $\tilde{\gamma}$ of $\underline{\smash{\gamma}}$ to~$\HH^2$ intersects the preimage of $|x|$ in a collection of points $\{p_i\}_{i\in\ZZ}$, naturally ordered along~$\tilde{\gamma}$, so that $\rho(\gamma)$ takes $p_i$ to $p_{i+m}$ for all $i\in\ZZ$. 
Let $\ell_i \subset \HH^2$ be the lifted geodesic arc of $|x|$ that contains~$p_i$, and let $\tilde{W}_x$ be the lift of the function $W_x$ to~$\HH^2$.
Recall that the support $|x|$ consists of at most $3N$ geodesic arcs, where $N:=|\chi|$ is the absolute value of the Euler characteristic of~$S$.
It is enough to find a constant $K'\geq 0$, independent of~$\gamma$, such that for any $i\in\ZZ$,
\begin{equation}\label{eqn:bound-Rx-K'}
\tilde{W}_x(p_i) \leq K' \, \vartheta_{|x|}(\gamma) \, d(p_i, p_{i+6N}) ;
\end{equation}
indeed, the result will follow (with $K=6NK'$) by adding up for $1\leq i \leq m$.

Let $\Omega\subset\HH^2$ be the preimage of the convex core of~$S$.
We first observe the existence of a constant $D>0$, independent of $\gamma$ and~$i$, such that if $\ell_i$ and $\ell_{i+6N}$ exit the same boundary component $\tilde{\eta}$ of~$\Omega$, then $d(p_i,p_{i+6N})\geq\nolinebreak D$.
Indeed, $|x|$ has at most $6N$ half-arcs exiting any boundary component of~$\Omega$; therefore, if $\ell_i$ and $\ell_{i+6N}$ both exit~$\tilde{\eta}$, then $\ell_j=\rho(\gamma')\cdot\ell_i$ for some integer $i<\nolinebreak j\leq i+6N$ and some element $\gamma'\in\Gamma$ stabilizing~$\tilde{\eta}$. 
By Proposition~\ref{prop:noslant}, this implies that the shortest distance from $\ell_i$ to~$\ell_j$ is bounded from below, independently of $\gamma,i,j$; therefore, so is $d(p_i, p_{i+6N})$. 

Let us prove the existence of $K'\geq 0$ such that \eqref{eqn:bound-Rx-K'} holds for all $i\in\ZZ$ with $d(p_i, p_{i+6N})\geq D$.
For any $\varepsilon>0$, there exists $\kappa>0$ (independent of $\gamma$ and~$i$) such that $\ell_i$ remains $\varepsilon$-close to $\tilde{\gamma}$ for at least
$$|\log (\vartheta_{|x|}(\gamma))| - \kappa$$
units of length to the left and right of~$p_i$: this follows from the fact that if two geodesics $\ell,\ell'$ of~$\HH^2$ intersect at an angle $\vartheta\in (0,\pi/2]$ at a point~$p$, then
\begin{equation}\label{eqn:right-angled-triangle}
e^{d(q,p)} \geq \sinh d(q,p) = \frac{\sinh d(q,\ell')}{\sin\vartheta} \geq \frac{\sinh d(q,\ell')}{\vartheta}
\end{equation}
for all $q\in\ell$.
If $\varepsilon$ is small enough (independently of $\gamma$ and~$i$), then in particular~$\ell_i$ cannot exit~$\Omega$ on this interval around $p_i$: otherwise $\tilde{\gamma}$ would exit it as well, by Proposition~\ref{prop:noslant}.
But on this interval, the maximum of $\tilde{W}_x|_{\ell_i}$ is at least $\tilde{W}_x(p_i)\,e^{|\log (\vartheta_{|x|}(\gamma))|-\kappa}/2$, by definition of~$W_x$. 
Using \eqref{eqn:boundwidth}, it follows that
$$\tilde{W}_x(p_i) \leq \vartheta_{|x|}(\gamma) \frac{2e^{\kappa}}{\sin \theta_0},$$
and so \eqref{eqn:bound-Rx-K'} holds with $K'=2e^{\kappa}/(D\sin\theta_0)$ when $d(p_i, p_{i+6N})\geq D$.

We now treat the case of integers $i\in\ZZ$ such that $d(p_i, p_{i+6N})<D$.
Using \eqref{eqn:right-angled-triangle}, we see that the distance from $p_i$ to the line~$\ell_{i+6N}$ is at most
$$\delta_{\gamma,i} := 2\vartheta_{|x|}(\gamma)\,d(p_i, p_{i+6N}) \in (0,\pi D) .$$
For any $\varepsilon>0$ there exists $\kappa>0$ (independent of $\gamma$ and~$i$) such that $\ell_i$ remains $\varepsilon$-close to $\ell_{i+6N}$ for at least
$$|\log \delta_{\gamma,i}| - \kappa$$
units of length to the right and left of~$p_i$: this follows from the fact that if $\ell,\ell'$ are two disjoint geodesics of~$\HH^2$ and if $p\in\ell$ is closest to~$\ell'$, then
$$\sinh d(q,\ell') = \cosh d(q,p) \cdot \sinh d(\ell,\ell')$$
for all $q\in\ell$, and $\delta\mapsto(\sinh\delta)/\delta$ is bounded on $(0,\pi D)$.
If $\varepsilon$ is small enough (independently of $\gamma$ and~$i$), then in particular~$\ell_i$ cannot exit~$\Omega$ on this interval around $p_i$: otherwise $\ell_{i+6N}$ would exit the same boundary component of~$\Omega$, by Proposition~\ref{prop:noslant}, which would contradict the definition of~$D$.
But on this interval, the maximum of $\tilde{W}_x|_{\ell_i}$ is at least $\tilde{W}_x(p_i)\,e^{|\log \delta_{\gamma,i}|-\kappa}/2$, by definition of~$W_x$. 
Using \eqref{eqn:boundwidth}, it follows that
$$\tilde{W}_x(p_i) \leq \delta_{\gamma,i} \, \frac{2e^{\kappa}}{\sin \theta_0},$$
and so \eqref{eqn:bound-Rx-K'} holds with $K'=4e^{\kappa}/\sin\theta_0$ when $d(p_i, p_{i+6N})<D$.
\end{proof}

\subsection{Boundedness of the map~$\boldsymbol f$}\label{subsec:f-bounded}

Observation~\ref{obs:length-deriv} and Proposition~\ref{prop:technical} imply that in the unit-peripheral normalization \eqref{eqn:normal},
\begin{equation} \label{eqn:biggestangle}
0 \leq \frac{\D\lambda_{\gamma}\big({\boldsymbol f}(x)\big)}{\lambda_{\gamma}(\rho)} \leq K \bigg( \max_{p\in\underline{\smash{\gamma}}\cap |x|} \measuredangle_p(\underline{\smash{\gamma}},|x|)\bigg)^2
\end{equation}
for all $\gamma\in\Gamma\smallsetminus\{ e\}$ and $x\in\overline{X}$, where $\underline{\smash{\gamma}}$ is the geodesic representative of $\gamma$ on~$S$.
The square on the right-hand side can be interpreted by saying that small intersection angles in $\underline{\smash{\gamma}}\cap |x|$ weaken the effect of the strip deformation ${\boldsymbol f}(x)$ on the length of~$\gamma$ in two different ways: first, by spreading out the intersection points along~$\underline{\smash{\gamma}}$ (Proposition~\ref{prop:technical}); second, by lessening the effect of each (Observation~\ref{obs:length-deriv}).

Recall that in the unit-peripheral normalization, the set ${\boldsymbol f}(\overline{X})$ is contained in some affine hyperplane of $T_{[\rho]}\T=H^1_{\rho}(\Gamma,\g)$ that does not contain~$0$.

\begin{proposition} \label{prop:unit-peripheral}
In the unit-peripheral normalization, the set ${\boldsymbol f}(\overline{X})$ is bounded in $H^1_{\rho}(\Gamma,\g)$; its projectivization $f(\overline{X})$ is contained in an affine chart of $\PP(H^1_{\rho}(\Gamma,\g))$.
\end{proposition}

\begin{proof}
There is a family $\{\gamma_1, \dots, \gamma_{3N}\} \subset\Gamma\smallsetminus\{ e\}$ such that the $\D\lambda_{\gamma_i}$ form a dual basis of $T_{[\rho]}\T$.
We then apply \eqref{eqn:biggestangle} to the~$\gamma_i$.
\end{proof}

The set $f(\overline{X})\subset\PP(H^1_{\rho}(\Gamma,\g))$ does not depend on the choice of normalization of widths $m_{\alpha}$, by Remark~\ref{rem:change-m-alpha}.

\subsection{Lengthening the boundary of~$S$}

To conclude this section, we show that a strip deformation with large weight (\ie belonging to $F(tX)$ for some large $t>0$) lengthens the boundary of~$S$ by a large amount, regardless of the supporting arcs.
This will be needed in Section~\ref{subsec:mainsteps123} in the proof that the map $F$ is proper (Proposition~\ref{prop:mainsteps}.(3)).

\begin{lemma}\label{lem:increase-bdry-length}
In the unit-peripheral normalization, there exists $K'\geq 0$ (depending on~$\rho$ and on the choice of geodesic representatives $(\underline{\alpha})_{\alpha\in\mathscr{A}}$) such that for any $t>0$ and any $x\in X$, the total boundary length of the convex core for $F(tx)\in\T$ is at least $2\log t-K'$.
\end{lemma}

\noindent As can be seen in the proof by considering one very long arc, $\log t$ is actually the optimal order of magnitude.

\begin{proof}
Up to passing to a double covering, we may assume that $S$ is orientable and $\rho\in\Hom(\Gamma,G)$ takes values in $G_0=\PSL_2(\RR)$.
By definition \eqref{eqn:normal} of the unit-peripheral normalization, we can find a point $p\in |x|\cap\partial S$ such that
$$W_x(p) \geq \frac{1}{6N},$$
where $3N=\dim(\T)$ as before.
Let $\underline{\alpha}$ be the arc of $|x|$ through~$p$.
In the upper half-plane model of~$\HH^2$, we may assume that $p$ lifts to $\sqrt{-1}$ and $\underline{\alpha}$ to $(0,\infty)$. 
Taking the basepoint of $\pi_1(S)$ at~$p$, the holonomy of the boundary loop $\gamma$ through~$p$ (suitably oriented) then has the form 
$$\rho(\gamma) = \begin{pmatrix} a & b \\ b & c \end{pmatrix}$$
with $a,b,c>0$ and $ad-bc=1$, which we see as an element of $G_0=\PSL_2(\RR)$.
By Proposition~\ref{prop:noslant}, we can bound $b$ below by some $b_0>0$ independent of $p$ and~$\alpha$.
Let $w$ be the width (measured at the waist as usual) of the strip inserted along~$\underline{\alpha}$ when producing the deformation $F(1x)$ from~$S$: by definition, $W_x(p)=w \cosh(d(p,p_\alpha))$, where the distance is measured along~$\underline{\alpha}$.
If only one end of~$\underline{\alpha}$ exits through the boundary loop~$\gamma$, then inserting a strip of width $tw$, as in the deformation $F(tx)$, corresponds to multiplying $\rho(\gamma)$ by
$$g := \begin{pmatrix} e^{d(p,p_{\alpha})/2} & 0 \\ 0 & e^{-d(p,p_{\alpha})/2} \end{pmatrix} \begin{pmatrix} \cosh \frac{tw}{2} & \sinh \frac{tw}{2} \\ \sinh \frac{tw}{2} & \cosh \frac{tw}{2} \end{pmatrix} \begin{pmatrix} e^{-d(p,p_{\alpha})/2} & 0 \\ 0 & e^{d(p,p_{\alpha})/2} \end{pmatrix} ,$$
which increases $\lambda(\rho(\gamma))$.
If the other end of~$\underline{\alpha}$, or some other arcs in the support of~$|x|$, also exit through~$\gamma$, then inserting strips to produce $F(tx)$ increases $\lambda(\rho(\gamma))$ even more, so we may use $\lambda(\rho(\gamma)g)$ as a lower bound for the length of the convex core corresponding to $F(tx)$.
A computation shows
\begin{eqnarray*}
\mathrm{tr}(\rho(\gamma)g) & \geq & b \, \big(e^{d(p,p_{\alpha})} + e^{-d(p,p_{\alpha})}\big) \, \sinh \frac{tw}{2}\\
& \geq & t\, b\, w \cosh(d(p,p_\alpha)) = t \, b \, W_x(p) \ \geq\ t \, \frac{b_0}{6N} \,,
\end{eqnarray*}
and so
$$\lambda_{\gamma}(\F(tx)) \geq 2 \, \mathrm{arccosh} \bigg(\frac{\mathrm{tr}(\rho(\gamma)g)}{2}\bigg) \geq 2\,\log t - K' ,$$
where $K':=2|\log(12N/b_0)|\geq 0$ does not depend on $t$ nor~$\alpha$.
\end{proof}

\section{Reduction of the proof of Theorems \ref{thm:main} and~\ref{thm:main-macro}}\label{sec:reduction}

We continue with the notation of Section~\ref{sec:metric-estim}.
As in Section~\ref{subsec:arc-complex}, let $X$ be the subset of the arc complex~$\overline{X}$ obtained by removing all faces corresponding to collections of arcs that do \emph{not} subdivide the surface $S$ into disks.
By Penner's work \cite{pen87} on the decorated Teichm\"uller space, $X$ is an open ball of dimension $3|\chi|-1$, where $\chi\in -\NN$ is the Euler characteristic of~$S$.
In order to prove Theorems \ref{thm:main} and~\ref{thm:main-macro}, it is sufficient to prove the following proposition.

\begin{proposition}\label{prop:mainsteps}
For any convex cocompact $\rho\in\Hom(\Gamma,G)$,
\begin{enumerate}
  \item the restriction of $f$ to~$X$ (\resp of $\F$ to $\mathrm{C}X$) takes values in the projectivized admissible cone $\adm(\rho)$ (\resp in $\Adm(\rho)$),
  \item $\adm(\rho)$ is an open ball of dimension $3|\chi|-1$, and $\Adm(\rho)$ is open in~$\T$ and homotopically trivial,
  \item the restrictions $f : X\rightarrow\adm(\rho)$ and $\F : \mathrm{C}X\rightarrow\Adm(\rho)$ are proper,
  \item the restrictions $f : X\rightarrow\adm(\rho)$ and $\F : \mathrm{C}X\rightarrow\Adm(\rho)$ are local homeomorphims.
\end{enumerate}
\end{proposition}

Indeed, (3) and~(4) imply that the restrictions $f : X\rightarrow\adm(\rho)$ and $\F : \mathrm{C}X\rightarrow\Adm(\rho)$ are coverings, and (2) implies that these coverings are trivial.
We will prove (1), (2), (3) in Section~\ref{subsec:mainsteps123}, while in Section~\ref{subsec:reduc-mainstep4} we will reduce (4) to a basic claim (Claim~\ref{claim:main}) about the behavior of the map ${\boldsymbol f}$ at faces of codimension $0$ and~$1$ in the arc complex, to be proved in Section~\ref{sec:codim01}.

\subsection{Range and properness of $f$ and~$F$}\label{subsec:mainsteps123}

In this section we prove statements (1), (2), and~(3) of Proposition~\ref{prop:mainsteps}.

\begin{proof}[Proof of Proposition~\ref{prop:mainsteps}.(1)] (See also \cite{pt10}.)
The inclusion $f(X)\subset\adm(\rho)$ means that any infinitesimal strip deformation ${\boldsymbol f}(x)$ of~$\rho$, performed on pairwise disjoint geodesic arcs $\underline{\alpha}_0,\dots,\underline{\alpha}_k$ that subdivide $S=\rho(\Gamma)\backslash\HH^2$ into topological disks, lengthens every curve at a \emph{uniform} rate relative to its length.
(Observation~\ref{obs:length-deriv} gives lengthening, but not uniform lengthening a priori.)
To see that this is true, note that, by compactness, there exist $R>0$ and $\theta\in (0,\pi/2]$ such that any closed geodesic $\underline{\smash{\gamma}}$ of~$S$ must cross one of the $\underline{\alpha}_i$ at least once every $R$ units of length, at an angle $\geq\theta$ (because $\underline{\smash{\gamma}}$ cannot exit the convex core).
If we denote by $\gamma$ the element of~$\Gamma$ corresponding to~$\underline{\smash{\gamma}}$, then Observation~\ref{obs:length-deriv} implies
\begin{equation}\label{eqn:lower-bound-deriv-length}
\frac{\D\lambda_{\gamma}\big({\boldsymbol f}(x)\big)}{\lambda_{\gamma}(\rho)} \geq \frac{\#\left(\underline{\smash{\gamma}}\cap\bigcup_{1\leq i\leq k} \underline{\alpha}_i\right)}{\lambda_{\gamma}(\rho)} \, w \,\sin\theta \geq \frac{w \, \sin\theta}{R} =: \varepsilon > 0,
\end{equation}
where $w>0$ is defined to be the minimum, for $1\leq i\leq k$, of the strip widths $w_{\alpha_i}$ associated with~$x$.
This proves that any infinitesimal strip deformation ${\boldsymbol f}(x)$, for $x\in X$, satisfies \eqref{eqn:propcritMink}, hence $f(X)\subset\adm(\rho)$.

To see that $\F(\mathrm{C}X)\subset\Adm(\rho)$, it is sufficient to establish that $\F(X)\subset\Adm(\rho)$ (where we identify $X$ with $1X\subset \mathrm{C}X$), for we may adjust the widths $w_{\alpha}$ as we wish.
Note that the bounds $\theta$ and~$R$ above can be taken to hold uniformly when $\rho$ and the geodesic representatives $\underline{\alpha}_i$ vary in a compact subset of the deformation space.
We can thus argue by integration of the infinitesimal inequality \eqref{eqn:lower-bound-deriv-length}, using \eqref{eqn:f-deriv-F}.
More precisely, for any $t_0\geq 0$ and $x\in\nolinebreak X$, the vector $\left . \frac{\D}{\D t} \right |_{t=t_0} \F(tx) \in T_{\F(t_0x)}\T$ is realized as an infinitesimal strip deformation of $\F(t_0x)$, for instance along geodesic arcs that bound the strips used to produce $\F(t_0 x)$ from~$\rho$.
Bounds $\theta$ and~$R$ as above hold for these geodesic arcs, independently of $t_0$ and~$x$, as long as $t_0\leq 1$.
Therefore, \eqref{eqn:lower-bound-deriv-length} implies
$$\frac{\D}{\D t}\Big|_{t=t_0}\ \frac{\lambda_{\gamma}(F(tx))}{\lambda_{\gamma}(F(t_0x))} \geq \varepsilon$$
for all $t_0\in [0,1]$, all $\gamma\in\Gamma\smallsetminus\{ e\}$, and all $x\in X$, and so $\lambda_{\gamma}(F(x))\geq e^{\varepsilon}\lambda_{\gamma}(\rho)$ for all $\gamma\in\Gamma$ and all $x\in X$, proving $F(X)\in\Adm(\rho)$.
\end{proof}

\begin{proof}[Proof of Proposition~\ref{prop:mainsteps}.(2)]
To see that $\Adm(\rho)$ is open in the Fricke--Teich\-m\"uller space~$\T$, we note that for any $[j]\in\T$ and $\varepsilon>0$, there exists a neighborhood $B_{\varepsilon}$ of $[j]$ in~$\T$ such that the convex cores of all hyperbolic metrics on~$S$ with holonomies in~$B_{\varepsilon}$ are mutually $(1+\varepsilon)$-bi-Lipschitz.
Thus
$$\frac{1}{1+\varepsilon} \, \leq \, \frac{\lambda_\gamma(j')}{\lambda_\gamma(j)} \, \leq \, 1+\varepsilon$$
for all $\gamma\in \Gamma\smallsetminus \{e\}$ and $[j']\in B_{\varepsilon}$.
In particular, if $j\in\Hom(\Gamma,G)$ satisfies \eqref{eqn:propcritAdS}, then for $\varepsilon$ small enough, any $j'\in\Hom(\Gamma,G)$ with $[j']\in B_{\varepsilon}$ also satisfies \eqref{eqn:propcritAdS}, and so $B_{\varepsilon}\subset\Adm(\rho)$, proving the openness of $\Adm(\rho)$.

The fact that $\adm(\rho)$ is open in $\PP(T_{[\rho]}\T)\simeq\PP(H^1_{\rho}(\Gamma,\g))$ follows from \cite{glm09} or \cite{dgk13}.
Alternatively, we may argue as above: by \eqref{eqn:propcritMink} and linearity of the maps $\mathrm{d}\lambda_\gamma$ for $\gamma\in\Gamma$, it is enough to check that the neighborhood $B_{\varepsilon}$ of $[j]$, corresponding to mutually $(1+\varepsilon)$-bi-Lipschitz convex cores, can be taken to contain a ball of radius $\geq C\varepsilon$ as $\varepsilon \rightarrow 0$, for some $C>0$ independent of~$\varepsilon$ and some smooth Riemannian metric on a neighborhood $U$ of $[\rho]$ in~$\T$.
This in turn follows from the existence of a \emph{smooth} local trivialization of the natural bundle of hyperbolic surfaces above~$U$, and compactness of the convex core.

By construction (see \eqref{eqn:propcritMink}), the positive admissible cone of~$\rho$ is also a \emph{convex} subset of $H^1_{\rho}(\Gamma,\g)$, hence its projectivization $\adm(\rho)$ is an open ball of dimension $3|\chi|-1$.

Finally, we check that $\Adm(\rho)$ is homotopically trivial.
Fix $k\in\NN$ and consider a continuous map $\sigma$ from the sphere $\SS^k$ to $\Adm(\rho)$.
We want to deform $\sigma$ to a constant map, inside the set of continuous maps from $\SS^k$ to $\Adm(\rho)$.
Choose a hyperideal triangulation $\Delta$ of~$S$.
For any $t\geq 0$, let\linebreak $\sigma_t : \SS^k\rightarrow\T$ be the postcomposition of~$\sigma$ with a strip deformation of width~$t$ along all arcs of~$\Delta$ simultaneously (taking for instance geodesic representatives that exit the boundary of the convex core of $S=\rho(\Gamma)\backslash\HH^2$ perpendicularly).
Then $\sigma_t(\SS^k) \subset \Adm(\rho)$ by Proposition~\ref{prop:mainsteps}.(1), and for large $t$ the convex core of any hyperbolic metric corresponding to some $\sigma_t(q)\in\T$, for $q\in\SS^k$, looks in fact like a collection of near-ideal triangles connected by long, thin isthmi of length $t+O(1)$, according to the combinatorics of the dual trivalent graph of~$\Delta$ (see Figure~\ref{fig:isthmi}).
\begin{figure}[ht!]
\centering
\labellist
\small\hair 2pt
\pinlabel $\sigma(q)$ at 60 25
\pinlabel $O(e^{-t/2})$ at 330 92
\pinlabel $t+O(1)$ at 330 36
\pinlabel $\sigma_t(q)$ at 330 10
\endlabellist
\includegraphics[width=11cm]{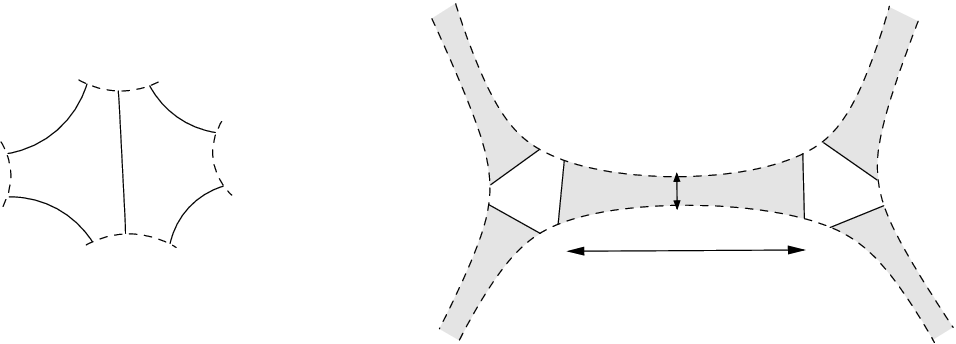}
\caption{The convex core of a convex cocompact hyperbolic surface (here $\sigma(q)$ for $q\in\SS^k$), before and after inserting strips of width~$t$ for large $t>0$. The strips are shaded.}
\label{fig:isthmi}
\end{figure}
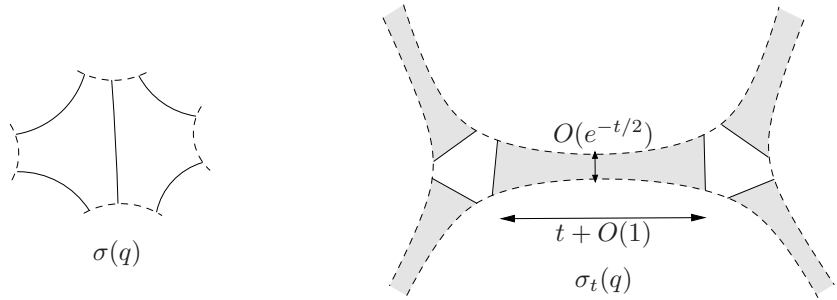
Here the error term $O(1)$ is uniform in $q\in\SS^k$.
The thickness of each isthm (\ie the smallest length of a geodesic arc across the convex core in the isotopy class of the strip) is $O(e^{-t/2})$.
These thicknesses form coordinates for the Fricke--Teichm\"uller space $\T$ of~$S$.
There exists $\varepsilon>0$, depending on~$\rho$, such that when these coordinates are all $\leq \varepsilon$, then the metric is in $\Adm(\rho)$.
Choose $t$ large enough to make all isthm thicknesses $\leq\varepsilon$ in all the metrics $\sigma_t(q)$ for $q$ ranging in~$\SS^k$, then interpolate linearly to thicknesses $(\varepsilon, ... ,\varepsilon)$.
This proves that $\sigma$ is homotopically trivial.
\end{proof}

In particular, $\Adm(\rho)$ is connected and simply connected.
Theorem~\ref{thm:main-macro} will imply that it is actually a ball.

\begin{proof}[Proof of Proposition~\ref{prop:mainsteps}.(3)]
To establish the properness of the restriction $f : X \to \adm(\rho)$, we consider a sequence $(x_n)$ escaping to infinity in~$X$, such that $(f(x_n))$ converges to some class $[u]\in\PP(H^1_{\rho}(\Gamma,\g))$; we must show that $[u]$ does not lie in $\adm(\rho)$.
By Remark~\ref{rem:change-m-alpha}, up to replacing the sequence $(x_n)$ with $(h(x_n))$ for some cell-preserving, cellwise projective homeomorphism $h : \overline{X}\rightarrow\overline{X}$, so that $|x_n|=|h(x_n)|$ for all~$n$, we may assume that we are working in the unit-peripheral normalization.

Suppose that, up to passing to a subsequence, the supports $|x_n|$ stabilize; then $(x_n)$ converges to some point $x\in\overline{X}\smallsetminus\nolinebreak X$, up to passing again to a subsequence.
By construction, the restriction of $f$ to any cell of the arc complex~$\overline{X}$ is continuous: in particular, $[u]=f(x)$.
Since $x\notin X$, the~support $|x|$ fails to decompose the surface into disks, hence the infinitesimal deformation ${\boldsymbol f}(x)$ fails to lengthen all curves (Observation~\ref{obs:length-deriv}), and $[u]\notin\nolinebreak\adm(\rho)$.

Otherwise, the supports $|x_n|$ diverge.
Up to passing to a subsequence, we may assume that they admit a Hausdorff limit $\Lambda\subset S$, which consists of a nonempty compact lamination together with some isolated leaves escaping in the funnels.
For any $\varepsilon>0$ we can find a simple closed geodesic $\underline{\smash{\gamma}}$ forming angles $\leq\varepsilon/2$ with $\Lambda$, hence $\leq\varepsilon$ with~$|x_n|$ for large~$n$.
Since the closure of ${\boldsymbol f}(\overline{X})$ in $H^1_{\rho}(\Gamma,\g)$ is compact and does not contain~$0$ (Proposition~\ref{prop:unit-peripheral}), the sequence $({\boldsymbol f}(x_n))$ converges to some infinitesimal deformation $u$ in the projective class $[u]$.
By \eqref{eqn:biggestangle}, the corresponding element $\gamma\in\Gamma$ satisfies
$$\frac{\D\lambda_{\gamma}(u)}{\lambda_{\gamma}(\rho)} \leq \varepsilon^2 K.$$
Since this holds for arbitrarily small~$\varepsilon$, we see that $u$ does not satisfy \eqref{eqn:propcritMink}, \ie $[u]\notin\adm(\rho)$. 
Thus $f$ is proper.

We now show that the restriction $\F : \mathrm{C}X\rightarrow\Adm(\rho)$ is proper.
As in the infinitesimal case, up to applying a cellwise linear homeomorphism $\mathrm{C}X\rightarrow \mathrm{C}X$, we may assume that we are working in the unit-peripheral normalization.
Consider a sequence $(t_nx_n)$ escaping to infinity in $\mathrm{C}X$, with $t_n\in\RR_+^{\ast}$ and $x_n\in\nolinebreak X$ for all~$n$; we must show that $\F(t_nx_n)$ escapes to infinity in $\Adm(\rho)$.
If $t_n\rightarrow +\infty$, then $\F(t_n x_n)$ escapes to infinity in~$\T$, because the length of the boundary of the convex core of the corresponding hyperbolic metric on~$S$ goes to infinity by Lemma~\ref{lem:increase-bdry-length}.
Up to passing to a subsequence, we may therefore assume that $(t_n)$ is bounded and that $(\F(t_n x_n))$ is bounded in~$\T$.
We then argue as in the infinitesimal case.
If the supports $|x_n|$ stabilize, then, up to passing to a subsequence, $(t_nx_n)$ converges to $tx$ for some $t\geq 0$ and $x\in\overline{X}\smallsetminus X$; by continuity of~$F$ on each cell of~$\overline{X}$, the sequence $(\F(t_nx_n))$ converges to $\F(tx)$.
Since $x\notin X$ we have $\F(tx)\notin\Adm(\rho)$: indeed, the support $|x|$ is disjoint from some simple closed geodesic~$\underline{\smash{\gamma}}$, hence the corresponding element $\gamma\in\Gamma$ satisfies $\lambda_{\gamma}(F(tx))=\lambda_{\gamma}(\rho)$.
If the supports $|x_n|$ diverge, then for any $\varepsilon>0$ we can find a simple closed geodesic $\underline{\smash{\gamma}}$ forming angles $\leq\varepsilon$ with $|x_n|$ for all large enough~$n$.
Proposition~\ref{prop:technical} then implies
$$\sum_{p\in\underline{\smash{\gamma}}\cap |x_n|} W_{x_n}(p) \leq K \varepsilon \, \lambda_{\gamma}(\rho)$$
for the corresponding element $\gamma\in\Gamma$.
Consider the representative of~$\gamma$ in the metric $\F(t_n x_n)$ that agrees with $\underline{\smash{\gamma}}$ outside of the strips and also includes (nongeodesic) segments crossing each strip at constant distance from the waist.
The length of this representative is exactly
$$\lambda_{\gamma}(\rho) + t_n \sum_{p\in\underline{\smash{\gamma}}\cap |x_n|} \, W_{x_n}(p).$$
Thus the length of $\gamma$ in the metric $\F(t_n x_n)$ is $\leq (1+Kt_n \varepsilon)\,\lambda_{\gamma}(\rho)$, and so any limit $[\rho']\in\T$ of a subsequence of $(\F(t_n x_n))$ satisfies $\lambda_{\gamma}(\rho')\leq (1+Kt\varepsilon)\,\lambda_{\gamma}(\rho)$.
This holds for arbitrarily small~$\varepsilon$, hence $[\rho']\notin\Adm(\rho)$.
\end{proof}

\subsection{Reduction of Proposition~\ref{prop:mainsteps}.(4)}\label{subsec:reduc-mainstep4}

The following claim is a stepping stone to the proof of Proposition~\ref{prop:mainsteps}.(4); it will be proved in Section~\ref{sec:codim01}.
The numbering of the statements (0), (1) refers to the codimension of the faces.

\begin{claim}\label{claim:main}
Let $\Delta,\Delta'$ be two hyperideal triangulations of~$S$ differing by a diagonal switch.
\begin{enumerate}
  \item[(0)] The points $f(\alpha)$, for $\alpha\in\mathscr{A}$ ranging over the $3|\chi|$ edges of~$\Delta$, are the vertices of a nondegenerate, top-dimensional simplex in $\PP(H^1_{\rho}(\Gamma,\g))$, denoted~$f(\Delta)$.
  \item[(1)] The simplices $f(\Delta)$ and $f(\Delta')$ have disjoint interiors in $\PP(H^1_{\rho}(\Gamma,\g))$.
  \item[(2)] There exists a choice of geodesic arcs $\underline{\alpha}$ and waists $p_{\alpha}$ for which $f(\Delta)\cup f(\Delta')$ is convex in $\PP(H^1_{\rho}(\Gamma,\g))$.
\end{enumerate}
\end{claim}
Here, we say that a closed subset of $\PP(H^1_{\rho}(\Gamma,\g))$ is \emph{convex} if it is convex and compact in some affine chart of $\PP(H^1_{\rho}(\Gamma,\g))$.
(Note that the whole image $f(X)$ has compact closure in some affine chart, by Proposition~\ref{prop:unit-peripheral}.) 
We now explain how Proposition~\ref{prop:mainsteps}.(4) follows from Claim~\ref{claim:main}.

\begin{proof}[Proof of Proposition~\ref{prop:mainsteps}.(4) assuming Claim~\ref{claim:main}]
Let us prove that $f$ is~a~lo\-cal homeomorphism near any point $x\in X$.
The simplicial structure of $\overline{X}$ defines a partition of $X$ into \emph{strata}, where the stratum of~$x$ is the unique simplex containing $x$ in its interior.
If $x$ belongs to a top-dimensional stratum of~$X$, then local homeomorphicity is Claim~\ref{claim:main}.(0).
If $x$ belongs to a codimension-$1$ stratum, then local homeomorphicity is Claim~\ref{claim:main}.(1).

Before treating the important case that $x$ belongs to a codimension-$2$ stratum, we first recall some useful classical terminology.
For any stratum $\sigma$ of~$X$, the union of the simplices of~$\overline{X}$ containing the closure $\overline{\sigma}$ of~$\sigma$ is the suspension of~$\overline{\sigma}$ with a simplicial sphere $L_{\sigma}$, called the \emph{link} of~$\sigma$. (That $L_\sigma$ is a sphere follows \eg from \cite{pen87}.)
The map $\boldsymbol{f}$ induces a cellwise affine map
$$\boldsymbol{f}_{\sigma} :\ \mathrm{C} L_{\sigma} \, \longrightarrow \, H^1_{\rho}(\Gamma,\g)/\spa(\boldsymbol{f}(\sigma)) \smallsetminus \{0\} \, \simeq \, \RR^{\mathrm{codim}(\sigma)} \smallsetminus \{0\} ,$$
called the \emph{link map} of $\boldsymbol{f}$ at~$\sigma$.
It also induces a \emph{(positively) projectivized link map}
$$f_\sigma :\ L_\sigma \, \longrightarrow \, \PP^+\big(H^1_{\rho}(\Gamma,\g)/\spa(\boldsymbol{f}(\sigma))\big) \, \simeq \, \SS^{\mathrm{codim}(\sigma)-1}.$$
To prove that $f$ is a local homeomorphism at~$x$, it is sufficient to prove that the projectivized link map of ${\boldsymbol f}$ at the stratum of~$x$ is a homeomorphism.

We now turn to codimension-$2$ strata of~$X$.
These come in two types.
The first type corresponds to decompositions of~$S$ into two hyperideal quadrilaterals and $2|\chi|-4$ hyperideal triangles.
Each quadrilateral can be divided into triangles in two ways, differing by a diagonal exchange.
Fixing the division of one quadrilateral, we find ourselves in a codimension-$1$ situation as above, where we can apply Claim~\ref{claim:main}.(1).
Thus the fact that $f$ is a local homeomorphism at~$x$ follows from the fact that there is no interference between the diagonal exchanges inside the two quadrilaterals.
More precisely, let $\sigma$ be the codimension-$2$ stratum and $\sigma',\sigma''$ the two neighboring codimension-$1$ strata, so that $\overline{\sigma}=\overline{\sigma}'\cap\overline{\sigma}''$.
Then $L_{\sigma}$ is the suspension of $L_{\sigma'}$ and~$L_{\sigma''}$, and $\boldsymbol{f}_{\sigma}=\boldsymbol{f}_{\sigma'}\oplus \boldsymbol{f}_{\sigma''}$.
Since $f_{\sigma'}$ and~$f_{\sigma''}$ are homeomorphisms,~so~is~$f_{\sigma}$.

The second type of codimension-$2$ strata of~$X$ corresponds to decompositions of~$S$ into one hyperideal pentagon and $2|\chi|-3$ hyperideal triangles.
The link of such a stratum $\sigma$ is a pentagon (a so-called \emph{pentagon move} between triangulations, see Figure~\ref{fig:pentagon}).
\begin{figure}[ht!]
\labellist
\small\hair 2pt
\pinlabel $\underline{\alpha}_1$ [l] at 76 157 
\pinlabel $\underline{\alpha}_2$ [l] at 105 81
\pinlabel $\underline{\alpha}_3$ [l] at 252 31
\pinlabel $\underline{\alpha}_4$ [l] at 338 184
\pinlabel $\underline{\alpha}_5$ [l] at 146 259
\pinlabel $\underline{\alpha}_1$ [l] at 126 36 
\pinlabel $\underline{\alpha}_2$ [l] at 256 81
\pinlabel $\underline{\alpha}_3$ [l] at 298 154
\pinlabel $\underline{\alpha}_4$ [l] at 212 259
\pinlabel $\underline{\alpha}_5$ [l] at 21 184
\pinlabel $0$ [l] at 512 173
\pinlabel $\boldsymbol{f}_\sigma(\alpha_1)$ [l] at 377 116
\pinlabel $\boldsymbol{f}_\sigma(\alpha_2)$ [l] at 487 51
\pinlabel $\boldsymbol{f}_\sigma(\alpha_3)$ [l] at 608 116
\pinlabel $\boldsymbol{f}_\sigma(\alpha_4)$ [l] at 543 246
\pinlabel $\boldsymbol{f}_\sigma(\alpha_5)$ [l] at 445 246
\pinlabel $\Delta$ [l] at 84 50
\pinlabel $\Delta'$ [l] at 282 50
\pinlabel $\boldsymbol{f}_\sigma(\Delta)$ [l] at 459 120
\pinlabel $\boldsymbol{f}_\sigma(\Delta')$ [l] at 530 120
\pinlabel $\boldsymbol{f}_\sigma$ [l] at 386 58
\endlabellist
\centering
\begin{changemargin}{-0.1cm}{0cm}
\includegraphics[width=12cm]{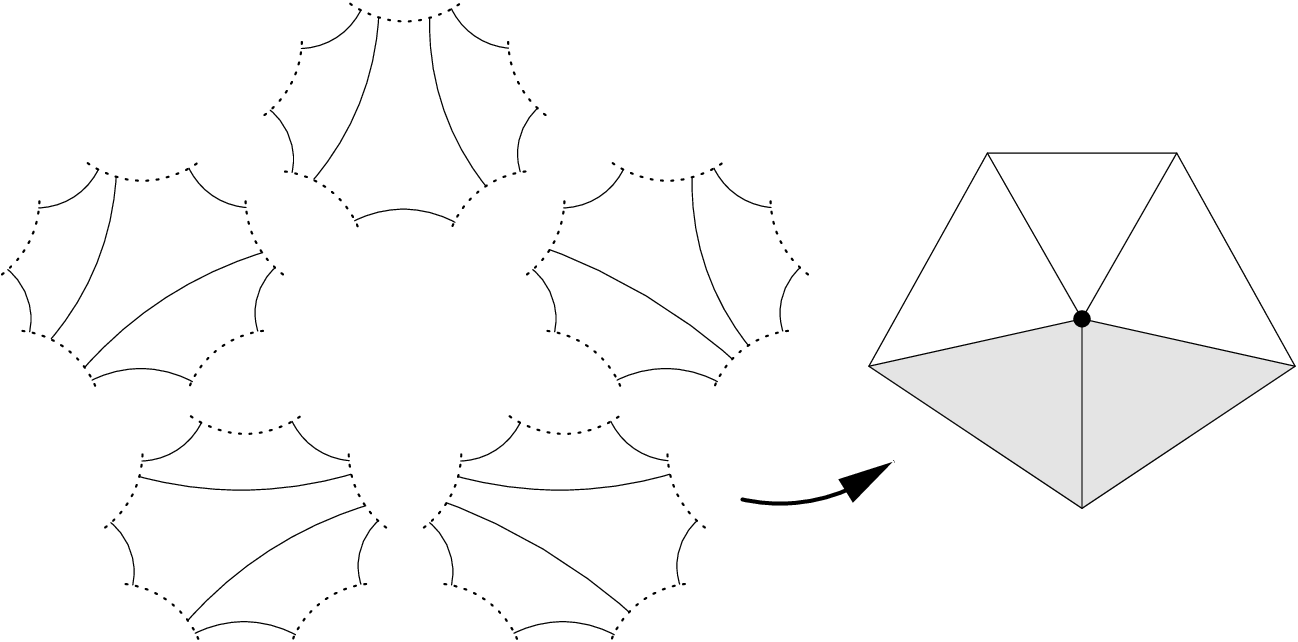}
\end{changemargin}
\caption{A pentagon move and its image under~$\boldsymbol{f}_\sigma$, where~$\sigma$ is a cellulation containing one pentagon and many triangles. In the source (left panel), the geodesic arcs $\underline{\alpha}_i$ cross various boundary components of the convex core of~$S$ (dashed) at near-right angles, creating hyperideal triangulations such as $\Delta$ or $\Delta'$.}
\label{fig:pentagon} 
\end{figure}
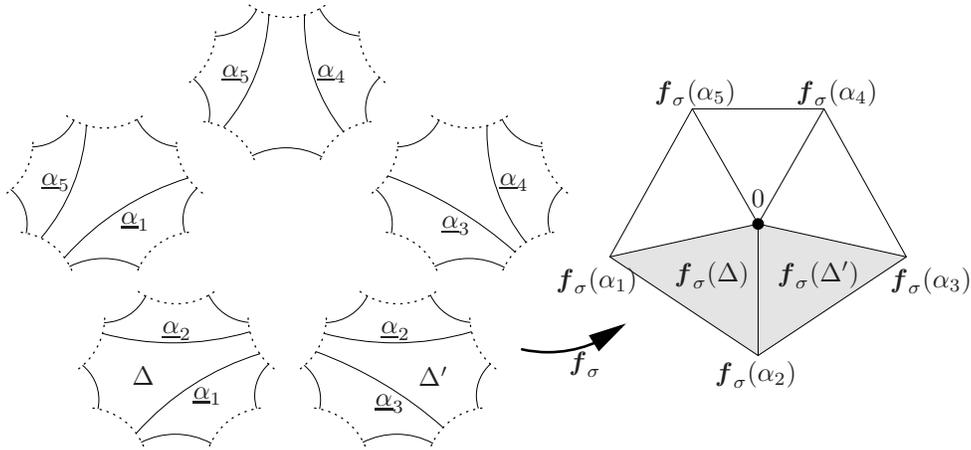
Consider an affine chart of $\PP(H^1_{\rho}(\Gamma,\g))$ containing $f(X)$ (Proposition~\ref{prop:unit-peripheral}) and equip it with a Euclidean metric so that the simplices $f(\Delta)$ in this chart become endowed with dihedral angles at their codimension-$2$ faces.
By Claim~\ref{claim:main}.(1), the projectivized link map $f_\sigma$ is a (piecewise projective) map from a circle to a circle, of degree either 1 or 2, because the image of each of the five segments in the source circle covers between 0 and $\pi$ of the target circle: five numbers $(\theta_i)_{i\in \ZZ/5\ZZ}$ in that range $(0,\pi)$ can add to either $2\pi$ or $4\pi$, but to no other multiple of~$2\pi$.
By Remark~\ref{rem:connected-choices}, the degree (1 or 2) remains constant as we change the positions of the geodesic arcs and waists, because one cannot pass continuously from $2\pi$ to~$4\pi$.
However, by Claim~\ref{claim:main}.(2) there is at least one choice of geodesic arcs~and waists for which the degree is~$1$.
Indeed, by convexity of $f(\Delta)\cup f(\Delta')$, one pair of consecutive numbers $\theta_i, \theta_{i+1}\in (0,\pi)$ has sum $\leq \pi$: the remaining three numbers cannot add to $\geq 3\pi$, so a total of $4\pi$ is impossible.
The degree is~$1$ for this choice of geodesic arcs and waists, hence for all choices, and $f_{\sigma}$ is a homeomorphism.
Therefore $f$ is a local homeomorphism at~$x$.

For $x$ in a stratum $\sigma$ of~$X$ of codimension $d\geq 3$, we argue by induction on~$d$.
The projectivized link map $f_{\sigma}$ is a map from a $(d-1)$-sphere to a $(d-1)$-sphere~$\SS^{d-1}$.
It is a local homeomorphism by induction, hence a covering.
But any connected covering of $\SS^{d-1}$ is a homeomorphism since $\SS^{d-1}$ is simply connected when $d\geq 3$.
Therefore $f$ is a local homeomorphism~at~$x$.

The fact that the restriction of $\F$ to $\mathrm{C}X$ is a local homeomorphism follows by the same argument as for the restriction of $f$ to~$X$.
Indeed, as in the proof of Proposition~\ref{prop:mainsteps}.(1), an infinitesimal perturbation of the widths of an actual (noninfinitesimal) strip deformation is the same as an infinitesimal deformation performed on the deformed surface.
\end{proof}

\section{Formalism for infinitesimal strip deformations}\label{sec:formalism}

In Section~\ref{sec:reduction} we explained how Theorems \ref{thm:main} and~\ref{thm:main-macro} reduce to Claim~\ref{claim:main}.
Before we prove Claim~\ref{claim:main} in Section~\ref{sec:codim01}, we introduce some notation and formalism that will also be useful in Section~\ref{sec:crooked-planes}.

\subsection{Killing vector fields on~$\HH^2$}\label{subsec:Killing-fields}

We identify the $3$-dimensional Minkowski space $\RR^{2,1}$ with the Lie algebra~$\g$, equipped with the symmetric bilinear form $\langle\cdot,\cdot\rangle$ equal to half the Killing form.
Note that $\g$ also naturally identifies with the space of \emph{Killing vector fields} on~$\HH^2$, \ie vector fields whose flow preserves the hyperbolic metric: an element $V\in\g$ defines the Killing field
$$p \longmapsto \frac{\D}{\D t}\Big|_{t=0}\, (e^{tV}\cdot p)\ \in T_p\HH^2 ,$$
and any Killing field is of this form for a unique $V\in\g$.
We shall write $V(p) \in T_p \HH^2$ for the vector at $p$ of the Killing field $V \in \g$.
Recall that $V\in\g$ is said to be \emph{hyperbolic} (\resp \emph{parabolic}, \resp \emph{elliptic}) if the one-parameter subgroup of~$G$ generated by~$e^V$ is hyperbolic (\resp parabolic, \resp elliptic).
We view the hyperbolic plane $\HH^2$ as a hyperboloid in~$\RR^{2,1}$ and its boundary at infinity $\partial_{\infty}\HH^2$ as the projectivized light cone:
\begin{eqnarray}\label{eqn:H2subsetR21}
\HH^2 & = & \big\{ v\in\RR^{2,1} ~|~ v_1^2 + v_2^2 - v_3^2 = -1,\ v_3>0\big\} ,\\
\partial_{\infty}\HH^2 & = & \big\{ [v]\in\PP(\RR^{2,1}) ~|~ v_1^2 + v_2^2 - v_3^2 = 0\big\} .\notag
\end{eqnarray}
The geodesic lines of~$\HH^2$ are the intersections of the hyperboloid with the linear planes of~$\RR^{2,1}$.
For any $p\in\HH^2$, the tangent space $T_p\HH^2$ naturally identifies with the linear subspace $p^{\perp}\subset\RR^{2,1}$ of vectors orthogonal to~$p$; this linear subspace is the translate back to the origin of the affine plane tangent at~$p$ to the hyperboloid $\HH^2 \subset \RR^{2,1}$.
For any $V\in\g\simeq\RR^{2,1}$ (seen as a Killing field on~$\HH^2$),
\begin{equation}\label{eqn:cross-product-killing}
V(p) = V \wedge p \ \in T_p\HH^2 = p^{\perp} \subset \RR^{2,1} \simeq \g,
\end{equation}
where $\wedge$ is the natural Minkowski cross product on~$\RR^{2,1}$:
$$(v_1,v_2,v_3) \wedge (w_1,w_2,w_3) := (- v_2 w_3 + v_3 w_2 \, , \,  - v_3 w_1 + v_1 w_3 \, , \, v_1 w_2 - v_2 w_1).$$
Note that for any $v,w\in\RR^{2,1}$ the vector $v\wedge w\in\RR^{2,1}$ is orthogonal to both $v$ and~$w$, and $(v,w,v\wedge w)$ is positively oriented (\ie satisfies the ``right-hand rule'').
Here is an easy consequence of~\eqref{eqn:cross-product-killing}.

\begin{lemma}\label{lem:inf-transl-in-R21}
An element $V\in\RR^{2,1}\simeq\g$, seen as a Killing vector field on~$\HH^2$, may be described as follows:
\begin{enumerate}
  \item If $V$ is timelike (\ie $\langle V, V \rangle < 0$), then it is elliptic: it is an infinitesimal rotation of velocity $\pm\sqrt{|\langle V, V \rangle |}$ centered at $\frac{\pm V}{\sqrt{|\langle V, V \rangle|}}\in\nolinebreak\HH^2\subset\nolinebreak\RR^{2,1}$.
  The velocity is positive if $V$ is future-pointing and negative otherwise.
  \item If $V$ is lightlike (\ie $\langle V, V \rangle = 0$ and $V\neq 0$), then it is parabolic with fixed point $[V]\in\partial_{\infty}\HH^2\subset\PP(\RR^{2,1})$.
  \item If $V$ is spacelike (\ie $\langle V, V\rangle > 0$), then it is hyperbolic: it is an infinitesimal translation of velocity $\sqrt{\langle V, V\rangle}$ with axis $\ell=V^{\perp}\cap\nolinebreak\HH^2\subset\nolinebreak\RR^{2,1}$.
  If $v^+, v^- \in V^{\perp}$ are future-pointing lightlike vectors representing respectively the attracting and repelling fixed points of $V$ in $\partial_{\infty}\HH^2\subset\PP(\RR^{2,1})$, then the triple $(v^+, V, v^-)$ is positively oriented (\ie satisfies the right-hand rule).
  \item[(3')] Let $\ell'$ be a geodesic of~$\HH^2$ whose endpoints in $\partial_{\infty}\HH^2\subset\PP(\RR^{2,1})$ are represented by future-pointing lightlike vectors $w_1,w_2\in\RR^{2,1}$.
  Then $V$ is an infinitesimal translation along an axis \emph{orthogonal} to~$\ell'$ if and only if $V$ is spacelike and belongs to $\operatorname{span}(w_1,w_2)$.
  Endow $\ell'$ with the transverse orientation placing $[w_1]$ on the left; then $V$ translates in the positive (\resp negative) direction if and only if $V \in \RR^*_+ w_1 + \RR^*_- w_2$ (\resp $V \in \RR^*_- w_1 + \RR^*_+ w_2$).
\end{enumerate}
\end{lemma}

\subsection{Bookkeeping for cocycles} \label{subsec:ookkee}

We now introduce a formalism for describing cocycles that will be useful in the proofs of Claim~\ref{claim:main} (Section~\ref{sec:codim01}) and Theorem~\ref{thm:crooked-planes} (Section~\ref{subsec:proof-crooked-plane-conj}).
The basic idea, following Thurston's description of earthquakes \cite{thu86b}, is to consider deformations (to be named~$\varphi$) of the hyperbolic surface that are locally isometric everywhere except along some fault lines, where they are discontinuous.
Each such deformation is characterized up to equivalence by a map (to be named~$\psi$) describing the relative motion of two pieces of the surface adjacent to a fault line. 

Consider a \emph{geodesic cellulation} $\underline{\Delta}$ of the convex cocompact hyperbolic surface $S=\rho(\Gamma)\backslash\HH^2$, consisting of vertices~$\mathscr{V}$, geodesic edges~$\mathscr{E}$, and finite-sided polygons~$\mathscr{T}$, such that the intersection of two edges (\resp polygons), if nonempty, is a vertex (\resp an edge) in their boundary.
The elements of~$\mathscr{E}$ may be geodesic segments, properly embedded geodesic rays, or properly embedded geodesic lines; the elements of~$\mathscr{T}$ may have infinite area.
A particular case of interest is when $\mathscr{E}$ consists of the geodesic representatives of the supporting arcs of a strip deformation $\boldsymbol{f}(x)$; in this case all polygons are hyperideal and $\mathscr{V}=\emptyset$.
Let $\widetilde{\Delta}$ be the the preimage of $\underline{\Delta}$ in the universal cover~$\HH^2$.
The vertices $\widetilde{\mathscr{V}}$, edges $\widetilde{\mathscr{E}}$, and polygons $\widetilde{\mathscr{T}}$ of the cellulation~$\widetilde{\Delta}$ are the respective preimages of $\mathscr{V}, \mathscr{E}$, and $\mathscr{T}$.
In what follows, we refer to the elements of~$\widetilde{\mathscr{T}}$ as the \emph{tiles}.
We denote by $\pm \widetilde{\mathscr{E}}$ the set of edges of $\widetilde{\mathscr{E}}$ endowed with a transverse orientation.
For $e\in\pm\widetilde{\mathscr{E}}$, we denote by $-e$ the same edge with the opposite transverse orientation.

Let us first recall a description of infinitesimal earthquake transformations.
For simplicity, we assume that the fault locus is a finite disjoint union of properly embedded geodesic lines.
We may build a geodesic cellulation $\underline{\Delta}$ such that the union of all edges in~$\mathscr{E}$ contains the fault locus.
Now consider an assignment $\varphi : \widetilde{\mathscr{T}}\to\g$ of infinitesimal motions to the tiles $\widetilde{\mathscr{T}}$ of the lifted cellulation $\widetilde{\Delta}$ (using the interpretation of Section~\ref{subsec:Killing-fields}).
Define a map $\psi: \pm \widetilde{\mathscr{E}} \to \g$ by assigning to any transversely oriented edge $e \in \pm \widetilde{\mathscr{E}}$ the \emph{relative motion} along that edge: $\psi(e) = \varphi(\delta') - \varphi(\delta)\in\g$ where $\delta, \delta' \in \widetilde{\mathscr{T}}$ are the tiles adjacent to~$e$, with $e$ transversely oriented from $\delta$ to~$\delta'$.
The map $\varphi$ defines an infinitesimal left earthquake transformation of~$S$ if $\psi$ is $\rho$-equivariant, \ie
\begin{equation}\label{eqn:rho-equiv}
\psi(\rho(\gamma)\cdot {e}) = \Ad(\rho(\gamma)) \, \psi({e})
\end{equation}
for all $\gamma\in\Gamma$ and ${e}\in\pm\widetilde{\mathscr{E}}$, and if for any ${e} \in \pm \widetilde{\mathscr{E}}$ whose projection lies in the fault locus, $\psi({e})$ is an infinitesimal translation to the left along $e$ (and $\psi({e}) = 0$ if ${e}$ does not project to the fault locus).
It is a simple observation that the infinitesimal deformation of~$S$ described by~$\varphi$ depends only on the relative motion map~$\psi$; that is, $\varphi$ may be recovered from $\psi$ up to a global isometry (see below). 

We now generalize this description of infinitesimal earthquakes and work with a larger class of deformations, for which the relative motion between adjacent tiles is allowed to be any infinitesimal motion.
Consider a $\rho$-equivariant map $\psi: \pm \widetilde{\mathscr{E}} \to \g$ satisfying the following \emph{consistency conditions}:
\begin{itemize}
  \item $\psi(-{e})=-\psi({e})$ for all ${e}\in\pm\widetilde{\mathscr{E}}$;
  \item the total motion around any vertex is zero: if ${e}_1,\ldots,{e}_k\in\pm\widetilde{\mathscr{E}}$ are the transversely oriented edges crossed (in the positive direction) by a loop encircling a vertex $v\in\widetilde{\mathscr{V}}$, then $\sum_{i=1}^k \psi({e}_i) = 0$.
\end{itemize}
Under these conditions, $\psi$ defines a cohomology class $[u] \in H^1_{\rho}(\Gamma,\g)$ as follows.
Choose a base tile $\delta_0\in\widetilde{\mathscr{T}}$ and an element $v_0 \in \g$.
Then $\psi$ determines a map $\varphi : \widetilde{\mathscr{T}} \to \g$ by integration: given a tile $\delta\in\widetilde{\mathscr{T}}$, consider a path $p : [0,1] \to \HH^2$ with initial endpoint $p(0)$ in the interior of~$\delta_0$ and final endpoint $p(1)$ in the interior of~$\delta$, such that $p(t)$ avoids the vertices $\widetilde{\mathscr{V}}$ of~$\widetilde{\Delta}$; we set
$$\varphi(\delta) := v_0 + \sum \psi({e}) ,$$
where the sum is over all transversely oriented edges ${e}\in\pm\widetilde{\mathscr{E}}$ crossed (in the positive direction) by the path $p(t)$.
This does not depend on the choice of~$p$, by the consistency conditions above.
For any tile $\delta\in\widetilde{\mathscr{T}}$,
$$u(\gamma) := \varphi(\rho(\gamma) \cdot \delta) - \Ad(\rho(\gamma))\,\varphi(\delta)$$
defines a $\rho$-cocycle $u : \Gamma \to \g$, independent of~$\delta$: we shall say that $\varphi$ is \emph{$(\rho,u)$-equivariant}.
The cohomology class of $u$ depends only on~$\psi$, not on the choice of $\delta_0$ and~$v_0$: indeed, the map integrating~$\psi$ with initial data $\delta'_0\in\widetilde{\mathscr{T}}$ and $v_0'\in\g$ differs from~$\varphi$ by the constant vector $w_0:=v_0'-\varphi(\delta_0')$, and therefore the cocycle it determines differs from $u$ by the coboundary $u_{w_0}=(\gamma\mapsto w_0-\Ad(\rho(\gamma))\,w_0)$. 
The map $\varphi$ assigns infinitesimal motions to the tiles in~$\widetilde{\mathscr{T}}$; by construction, $\psi({e})=\varphi(\delta')-\varphi(\delta)$ for any tiles $\delta,\delta'$ adjacent to an edge ${e}\in\pm\widetilde{\mathscr{E}}$ transversely oriented from $\delta$ to~$\delta'$.

Let $\Psi(\pm \widetilde{\mathscr{E}},\g)$ be the space of $\rho$-equivariant maps $\psi : \pm\widetilde{\mathscr{E}} \to \g$ satisfying the two consistency conditions above.
The integration process $\psi \mapsto [u]$ we have just described defines an $\RR$-linear map
\begin{equation}\label{eqn:map-L}
\mathscr{L} : \Psi\big(\!\pm\widetilde{\mathscr{E}},\g\big) \longrightarrow H^1_\rho (\Gamma, \g) .
\end{equation}
Note that each tile has trivial stabilizer in~$\Gamma$, because it is finite-sided and $\Gamma$ is torsion-free.
Therefore the map $\mathscr{L}$ is onto, \ie any infinitesimal deformation $[u] \in H^1_{\rho}(\Gamma,\g)$ of $\rho$ is achieved by some assignment $\psi$ of relative motions.
Indeed, choose a representative in $\widetilde{\mathscr{T}}$ for each element of~$\mathscr{T}$, and choose arbitrary values for $\varphi$ on these representatives.
We can extend this to a $(\rho,u)$-equivariant map $\varphi : \widetilde{\mathscr{T}} \to \g$, and define $\psi({e}):=\varphi(\delta')-\nolinebreak\varphi(\delta)$ for any tiles $\delta,\delta'$ adjacent to an edge ${e}\in\pm\widetilde{\mathscr{E}}$ transversely oriented from $\delta$ to~$\delta'$.
This map $\psi$ satisfies the consistency conditions above.

\subsection{Infinitesimal strip deformations}\label{subsec:formalism-strips}

In our main case of interest, the cellulation $\underline{\Delta}$ is a geodesic hyperideal triangulation and the edges $\mathscr{E}$ of $\underline{\Delta}$ are the geodesic representatives of the supporting arcs of an infinitesimal strip deformation.

\begin{remark}
In our geodesic hyperideal triangulations, we do \emph{not} a priori require the extended edges to meet in a single point in $\PP^2(\RR)\smallsetminus\HH^2$.
\end{remark}

Recall that $(\underline{\alpha},p_{\alpha},m_{\alpha})_{\alpha\in\mathscr{A}}$ denotes the choice of geodesic representatives, waists, and widths of strips defining~${\boldsymbol f}$ (see Section~\ref{subsec:arc-complex}).

\begin{definition}\label{def:psi-alpha}
Let $\alpha\in\mathscr{A}$ be an arc of~$S$ with $\underline{\alpha}\in\mathscr{E}$.
The \emph{relative motion map} $\psi_{\alpha}\in\Psi(\pm \widetilde{\mathscr{E}},\g)$ of the infinitesimal strip deformation ${\boldsymbol f}(\alpha)\in H^1_{\rho}(\Gamma,\g)$ is defined as follows:
\begin{itemize}
  \item for any transversely oriented lift $\tilde{\alpha}\in\pm\widetilde{\mathscr{E}}$ of~$\underline{\alpha}$, the element $\psi_{\alpha}(\tilde{\alpha})\in\g$ is the infinitesimal translation of velocity $m_{\alpha}$ along the axis orthogonal to $\tilde{\alpha}$ at (the lift of)~$p_{\alpha}$, in the positive direction;
  \item $\psi_{\alpha}(e)=0$ for any other transversely oriented edge $e\in\pm\widetilde{\mathscr{E}}$.
\end{itemize}
\end{definition}

Recall the map $\mathscr{L}$ from \eqref{eqn:map-L}.
The following observation is elementary but essential for the proof of Claim~\ref{claim:main} in Sections \ref{subsec:independent} and~\ref{subsec:abcd}.

\begin{observation}\label{obs:realize-strips}
The relative motion map $\psi_{\alpha}$ \emph{realizes} the infinitesimal strip deformation $\boldsymbol f(\alpha)$, in the sense that $\mathscr{L}(\psi_{\alpha}) = \boldsymbol f(\alpha)$.
\end{observation}

In Section~\ref{subsec:abcd}, it will be necessary to work simultaneously with two different geodesic hyperideal triangulations $\underline{\Delta}$ and~$\underline{\Delta}'$.
Consider a common refinement $\underline{\Delta}''$ of $\underline{\Delta},\underline{\Delta}'$.
For $\underline{\Delta}',\underline{\Delta}''$, we use notation $\mathscr{E}',\mathscr{E}''$ and $\mathscr{L}',\mathscr{L}''$ similar to Section~\ref{subsec:ookkee}.
Then there are natural inclusion maps
$$\mathscr{I}: \Psi(\pm \widetilde{\mathscr{E}}, \g) \hookrightarrow \Psi(\pm \widetilde{\mathscr{E}}'', \g) \quad\mathrm{and}\quad \mathscr{I}': \Psi(\pm \widetilde{\mathscr{E}}', \g) \hookrightarrow \Psi(\pm \widetilde{\mathscr{E}}'', \g)$$
defined as follows: for any $\psi \in \Psi(\pm \widetilde{\mathscr{E}}, \g)$ and $e''\in\pm\widetilde{\mathscr{E}}''$, set
\begin{itemize}
  \item $\mathscr{I}(\psi)({e}'') := \varepsilon\,\psi({e})$ if $e''$ is contained in an edge $e\in\pm\widetilde{\mathscr{E}}$, with $\varepsilon=1$ (\resp $\varepsilon=-1$)  if the transverse orientations of $e''$ and~$e$ agree (\resp disagree),
  \item $\mathscr{I}({e}'') :=0$ otherwise,
\end{itemize}
and similarly for~$\mathscr{I}'$.
By using these inclusion maps we may compare relative motion maps defined on the two different triangulations $\underline{\Delta}$ and~$\underline{\Delta}'$.
We consider $\psi \in \Psi(\pm \widetilde{\mathscr{E}}, \g)$ and $\psi' \in \Psi(\pm \widetilde{\mathscr{E}}', \g)$ equivalent if $\mathscr{I}(\psi) = \mathscr{I}'(\psi')$.
Here are two simple observations:

\begin{observations}\label{obs:two-triang}
(1) $\mathscr{L}''\circ \mathscr{I} = \mathscr{L}$  and $\mathscr{L}''\circ \mathscr{I}'= \mathscr{L}'$.\\
(2) Let $\alpha\in\mathscr{A}$ be an arc of~$S$ with $\underline{\alpha}\in\mathscr{E}\cap\mathscr{E}'\subset\mathscr{E}''$.
Let $\psi_{\alpha} \in \Psi(\pm \widetilde{\mathscr{E}}, \g)$ and $\psi_{\alpha}' \in \Psi(\pm \widetilde{\mathscr{E}}', \g)$ be the two relative motion maps realizing $\boldsymbol f(\alpha)$ as in Observation~\ref{obs:realize-strips}, so that $\mathscr{L}(\psi_{\alpha}) = \mathscr{L}'(\psi_{\alpha}') = \boldsymbol{f}(\alpha)$.
Then $\mathscr{I}(\psi_{\alpha}) = \mathscr{I}'(\psi_{\alpha}')$.
\end{observations}

Observation~\ref{obs:two-triang}.(2) means that for any arc $\alpha\in\mathscr{A}$ the map $\psi_{\alpha}$ is well defined, up to equivalence, independently of the geodesic hyperideal triangulation~$\underline{\Delta}$ containing~$\underline{\alpha}$.

We will also need to compose (\ie add) infinitesimal strip deformations supported on arcs that intersect.
Suppose that $\alpha$ and $\alpha'$ are two arcs of~$S$ with geodesic representatives $\underline{\alpha}$ and $\underline{\alpha}'$ contained in $\underline{\Delta}$ and~$\underline{\Delta}'$, respectively.
We define the sum $\psi_{\alpha} + \psi_{\alpha'}$ to be
$$\psi_{\alpha} + \psi_{\alpha'} := \mathscr{I}(\psi_{\alpha}) + \mathscr{I}'(\psi_{\alpha'}) \in \Psi(\pm \widetilde{\mathscr{E}}'', \g).$$
Then $\mathscr{L}$ commutes with this operation: by Observation~\ref{obs:two-triang}.(1),
$$\mathscr{L}''(\psi_{\alpha} + \psi_{\alpha'}) = \mathscr{L}(\psi_{\alpha}) + \mathscr{L}'(\psi_{\alpha'}) .$$

\section{Behavior of $f$ at faces of codimension $0$ and~$1$}\label{sec:codim01}

We now prove Claim~\ref{claim:main}, using the formalism of Section~\ref{sec:formalism}.

\subsection{Proof of Claim~\ref{claim:main}.(0)}\label{subsec:independent}

Let $\underline{\Delta}$ be the geodesic hyperideal triangulation of $S=\rho(\Gamma)\backslash\HH^2$ whose edges $\mathscr{E}$ are the geodesic representatives $\underline{\alpha}$, chosen in the definition of the map $\boldsymbol f$ (see Section~\ref{subsec:arc-complex}), of the arcs $\alpha$ of~$\Delta$.
We continue with the notation of Sections \ref{subsec:ookkee} and~\ref{subsec:formalism-strips}.
Let us prove that the infinitesimal strip deformations ${\boldsymbol f}(\alpha)\in H^1_{\rho}(\Gamma,\g)$, for $\underline{\alpha}\in\mathscr{E}$, span all of $H^1_{\rho}(\Gamma,\g)$.
Since $\dim H^1_{\rho}(\Gamma,\g)=\#\mathscr{E}$, it is equivalent to show that the ${\boldsymbol f}(\alpha)$ are linearly independent.
Suppose that
\begin{equation}\label{eqn:dependence}
\sum_{\underline{\alpha}\in\mathscr{E}} c_{\alpha} \, {\boldsymbol f}(\alpha) = 0
\end{equation}
for some $(c_{\alpha})\in\RR^{\mathscr{E}}$, and let us prove that $c_{\alpha}=0$ for all $\underline{\alpha}\in\mathscr{E}$.
By Observation~\ref{obs:realize-strips} and linearity of~$\mathscr{L}$, the left-hand side of \eqref{eqn:dependence} is realized by the $\rho$-equivariant relative motion of the tiles $\psi := \sum_{\underline{\alpha}\in\mathscr{E}} c_{\alpha} \psi_{\alpha} : \pm\widetilde{\mathscr{E}}\to\g$, such that for any transversely oriented lift $\tilde{\alpha}\in\pm\widetilde{\mathscr{E}}$ of any $\underline{\alpha}\in\mathscr{E}$,
\begin{equation}\label{eqn:description-psi}
\psi(\tilde{\alpha}) = c_{\alpha} \, \psi_{\alpha}(\tilde{\alpha})
\end{equation}
(see Definition~\ref{def:psi-alpha}).
Since $\mathscr{L}(\psi)=0$ by \eqref{eqn:dependence}, there is a map $\varphi : \widetilde{\mathscr{T}}\to\nolinebreak\g$ integrating~$\psi$ that is $(\rho,0)$-equivariant (\ie $\rho$-equivariant in the sense of~\eqref{eqn:rho-equiv}).
Indeed, choose an arbitrary base tile $\delta_0$ and an arbitrary motion $v_0\in\g$ of that tile.
The map $\varphi' : \widetilde{\mathscr{T}}\to\g$ determined by $\psi$ and this initial data, as in Section~\ref{subsec:ookkee}, is $(\rho,u_{w_0})$-equivariant for some $\rho$-\emph{coboundary} $u_{w_0}=(\gamma\mapsto w_0 - \Ad(\rho(\gamma))\,w_0)$.
Then the map $\varphi := \varphi' - w_0 : \widetilde{\mathscr{T}} \to \g$ integrates~$\psi$ and is $(\rho,0)$-equivariant.

Consider an edge $\tilde{\alpha}\in\widetilde{\mathscr{E}}$, with adjacent tiles $\delta,\delta'\in\widetilde{\mathscr{T}}$.
The vectors $v:=\varphi(\delta)$ and $v':=\varphi(\delta')$ encode the infinitesimal motions of the respective tiles $\delta,\delta'$.
The vector $v\in\g$ may be decomposed as $v = v_t + v_{\ell}$, where $v_t\in\spa(\tilde{\alpha})\subset\RR^{2,1}$ is called the \emph{transverse motion} and $v_{\ell}\in\spa(\tilde{\alpha})^{\perp}$ the \emph{longitudinal motion}.
By Lemma~\ref{lem:inf-transl-in-R21}.(3), the longitudinal motion $v_{\ell}$ is an infinitesimal translation with axis~$\tilde{\alpha}$.
Similarly, we decompose $v'$ as $v' = v'_t + v'_{\ell}$.
By \eqref{eqn:description-psi} and Lemma~\ref{lem:inf-transl-in-R21}.(3'), if we endow $\tilde{\alpha}$ with the transverse orientation from $\delta$ to~$\delta'$, then $\psi(\tilde{\alpha})=v'-v\in\spa(\tilde{\alpha})$, which means that $v$ and~$v'$ impart the same longitudinal motion to~$\tilde{\alpha}$, \ie $v_{\ell}=v'_{\ell}$.
Thus $\tilde{\alpha}$ receives a well-defined amount $\sqrt{\langle v_{\ell},v_{\ell}\rangle} = \sqrt{\langle v'_{\ell},v'_{\ell}\rangle}$ of longitudinal motion from~$\varphi$, equal to the longitudinal motion of the tile on either side of the edge; this amount is invariant under the action of $\rho(\Gamma)$ because $\varphi$ is $(\rho,0)$-equivariant.
It is sufficient to prove that all longitudinal motions of edges of $\widetilde{\mathscr{E}}$ are zero, because then $\varphi=0$ and $\psi=0$, and so the linear dependence \eqref{eqn:dependence} is trivial.
Indeed, the three linear forms on~$\RR^{2,1}$ that vanish on the three edges bounding a tile form a dual basis of~$\RR^{2,1}$ (because the edges, extended in $\PP(\RR^{2,1})$, have no common intersection point), hence a Killing field that imparts zero longitudinal motion to all three edges must be zero.
We now assume by contradiction that not all longitudinal motions are zero.

Choose for~$\tilde{\alpha}$ an edge receiving maximal longitudinal motion, \ie such that $v_{\ell}=v'_{\ell}$ has maximal Lorentzian norm $\sqrt{\langle v_{\ell},v_{\ell}\rangle}$ among all edges.
Let $A,B,C,D,E,F$ (\resp $A,B,C',D',E',F'$) be the endpoints in $\partial_{\infty}\HH^2\subset\PP(\RR^{2,1})$ of all the edges of $\delta$ (\resp $\delta'$), cyclically ordered (see Figure~\ref{fig:projective}, left).
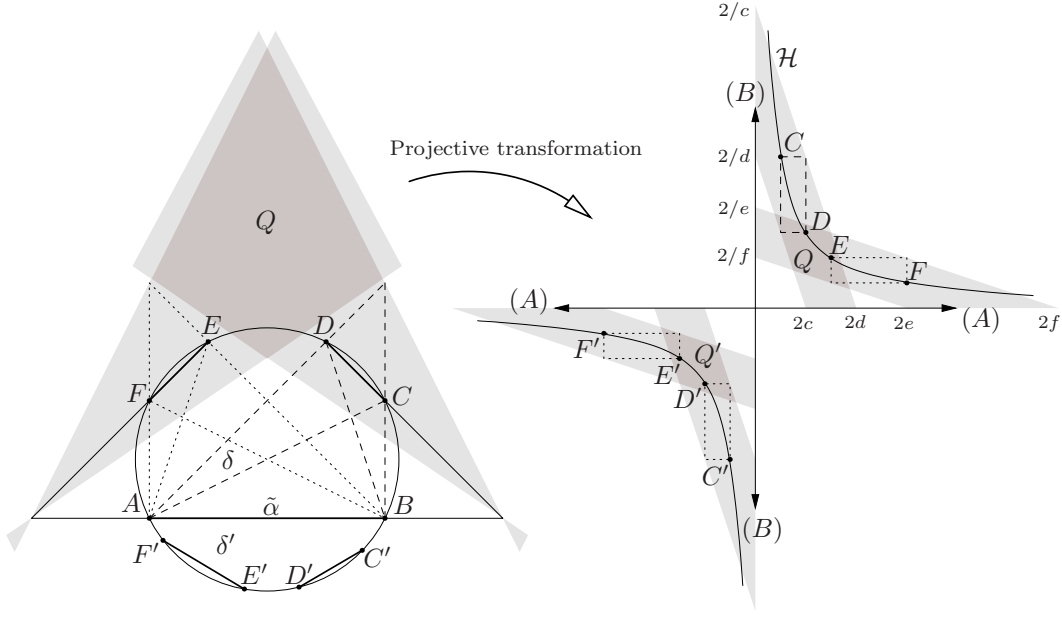
\begin{figure}[ht!]
\labellist
\small\hair 2pt
\pinlabel {$\delta$} at 119 81
\pinlabel {$\delta'$} at 118 36
\pinlabel {$\tilde{\alpha}$} at 142 56
\pinlabel {$A$} at 67 58
\pinlabel {$B$} at 213 57
\pinlabel {$C$} at 212 115
\pinlabel {$D$} at 170 153
\pinlabel {$E$} at 110 153
\pinlabel {$F$} at 70 117
\pinlabel {$C$} at 422 251
\pinlabel {$D$} at 436 209
\pinlabel {$E$} at 446 196
\pinlabel {$F$} at 488 182
\pinlabel {${}_{2c}$} at 427 154
\pinlabel {${}_{2d}$} at 455 155
\pinlabel {${}_{2e}$} at 481 154
\pinlabel {${}_{2f}$} at 559 154
\pinlabel {${}_{2/c}$} at 390 321
\pinlabel {${}_{2/d}$} at 390 243
\pinlabel {${}_{2/e}$} at 390 215
\pinlabel {${}_{2/f}$} at 390 190
\pinlabel {$C'$} at 199 29
\pinlabel {$D'$} at 157 20
\pinlabel {$E'$} at 133 20
\pinlabel {$F'$} at 75 31
\pinlabel {$Q$} at 139 208
\pinlabel {$Q'$} at 376 137
\pinlabel {$C'$} at 380 73
\pinlabel {$D'$} at 365 114
\pinlabel {$E'$} at 354 127
\pinlabel {$F'$} at 311 141
\pinlabel {$Q$} at 427 187
\pinlabel {$\mathcal{H}$} at 418 297
\pinlabel {$(A)$} at 522 156
\pinlabel {$(B)$} at 396 277
\pinlabel {$(A)$} at 280 166
\pinlabel {$(B)$} at 405 43
\pinlabel {\scriptsize{Projective transformation}} at 273 247
\endlabellist
\centering
\begin{changemargin}{-0.6cm}{0cm}
\includegraphics[width=14cm]{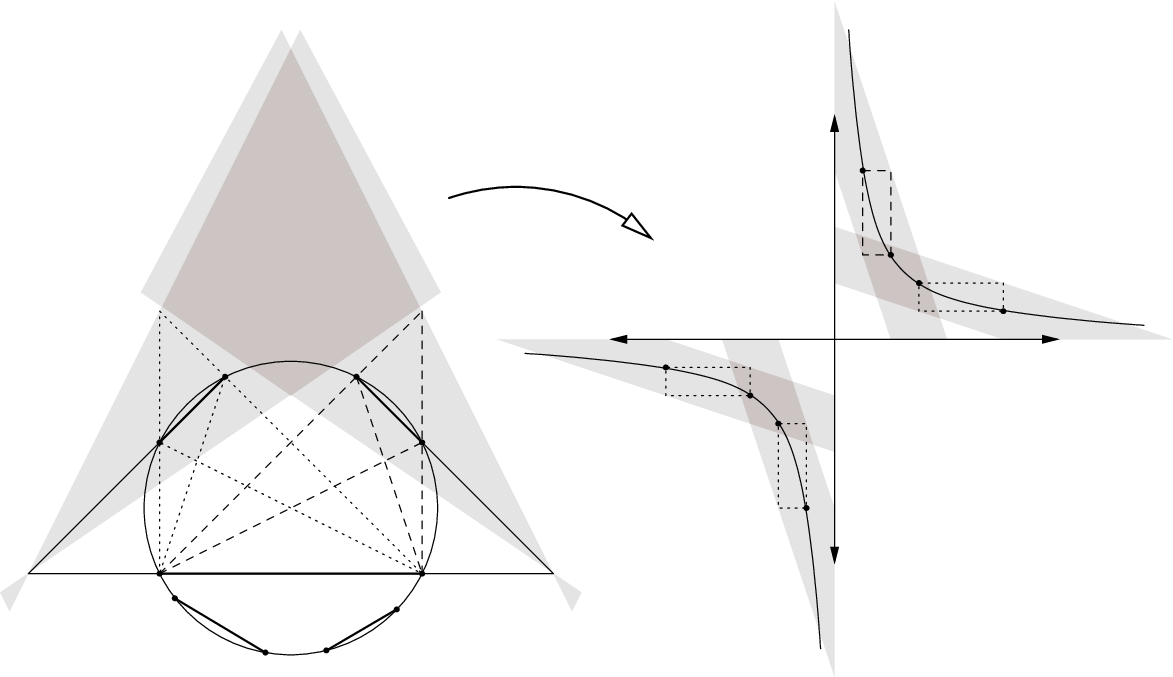}
\end{changemargin}
\caption{Left: view of $\PP(\RR^{2,1})$ and its subset $\partial_{\infty}\HH^2$ in the affine chart $\{ x_3=1\} $.
Right: view in another affine chart, obtained by slicing $\RR^{2,1}$ along a plane parallel to $\spa(A,B)$.}
\label{fig:projective} 
\end{figure}
For convenience, we refer to an edge by its two endpoints, so that \eg $\tilde{\alpha} = AB := \HH^2\cap\spa(A,B)$.
The fact that the longitudinal motion of $AB$ is at least that of $CD$ means that the image $[v]\in\PP(\RR^{2,1})$ of~$v$ lies in a bigon of $\PP(\RR^{2,1})$ bounded by two projective lines through the point $AB\cap CD$, namely the line through $AB\cap CD$ and $AC\cap BD$ and the line through $AB\cap CD$ and $AD\cap BC$.
Of the two regions of $\PP(\RR^{2,1})$ that these lines determine, the correct one is the one containing $CD$ (if $[v]$ is on $CD$ then the longitudinal motion of $CD$ is zero).
We refer to Figure~\ref{fig:projective} (left panel), in which the relevant region is shaded.
Similarly, the fact that the longitudinal motion of $AB$ is at least that of $EF$ means that $[v]$ lies in a region of $\PP(\RR^{2,1})$ bounded by two projective lines through the point $AB\cap EF$, namely the line through $AB\cap EF$ and $AE\cap BF$ and the line through $AB\cap EF$ and $AF \cap BE$.
These two conditions determine a quadrilateral $Q$ of $\PP(\RR^{2,1})$ to which $[v]$ must belong.
Similarly,  $[v']\in\PP(\RR^{2,1})$ must belong to another quadrilateral $Q'$ of $\PP(\RR^{2,1})$, corresponding to the fact that the longitudinal motion of $AB$ is at least that of $C'D'$ and of $E'F'$.

Consider the affine chart of $\PP(\RR^{2,1})$ obtained by slicing $\RR^{2,1}$ along the affine plane parallel to $\spa(A,B)$ passing through $v$ and $v'$; note that this plane contains the origin only if $v_{\ell}=v'_{\ell}=0$ and we have assumed this is not the case.
The corresponding projective transformation is shown in Figure~\ref{fig:projective}, right.
In this new affine chart the points $A$ and~$B$ are at infinity, the points $C,D,E,F$ lie in this order on one branch of a hyperbola $\mathcal{H}$ with asymptotes of directions $A$ (horizontal) and~$B$ (vertical), and the points $C',D',E',F'$ lie in this order on the other branch of~$\mathcal{H}$.
Consider the restriction of the $\RR^{2,1}$ metric to this affine plane.
The two asymptotes, which are lightlike, divide the plane into four quadrants: two of them timelike (namely those containing~$\mathcal{H}$) and two of them spacelike.
We claim that $Q$ and~$Q'$ lie in \emph{opposite}, \emph{timelike} quadrants; this will contradict the fact that the vector of relative motion $v'-v=\psi(\tilde{\alpha})$ is spacelike.
Indeed, by construction the quadrilateral $Q$ is the intersection of two infinite Euclidean strips: the first strip is the union of all translates, along the direction of the line $(CD)$, of the rectangle circumscribed to the segment $[CD]$ with sides parallel to the asymptotes; the second strip is the union of all translates, along the direction of the line $(EF)$, of the rectangle circumscribed to the segment $[EF]$.
Without loss of generality, we may assume that $C, D, E, F$ have respective coordinates $(c,1/c), (d,1/d), (e,1/e), (f,1/f)$ where $0<c<\nolinebreak d<e<\nolinebreak f<+\infty$.
Then the four boundary lines of the two Euclidean strips intersect the horizontal axis at distance $2c < 2d < 2e < 2f$ from the origin, and the vertical axis at distance $2/c > 2/d >  2/e > 2/f$ from the origin.
The quadrilateral~$Q$, which is the intersection of the two strips, therefore lies entirely in the timelike quadrant that contains the branch of $\mathcal{H}$ on which $C, D, E, F$ lie.
Similarly, $Q'$ lies entirely in the opposite quadrant.
Therefore the image of $\psi(AB)=v'-v$ in $\PP(\RR^{2,1})$ is timelike, a contradiction.

\subsection{Proof of Claim~\ref{claim:main}.(1)--(2)}\label{subsec:abcd}

The two hyperideal triangulations $\Delta$ and~$\Delta'$ have all but one arc in common.
Let $\alpha$ (\resp $\alpha'$) be the arc~of~$\Delta$ (\resp $\Delta'$) that is not an arc of $\Delta'$ (\resp $\Delta$).
By Claim~\ref{claim:main}.(0), the sets $\{ {\boldsymbol f}(\beta) \,|\, \beta\text{ an arc of }\Delta\}$ and $\{ {\boldsymbol f}(\beta) \,|\, \beta \text{ an arc of }\Delta'\}$ are both bases of $H^1_{\rho}(\Gamma,\g)$.
Therefore there is, up to scale, exactly one linear relation of the form
\begin{equation}\label{eqn:zerocross}
c_{\alpha}\,{\boldsymbol f}(\alpha) + c_{\alpha'}\,{\boldsymbol f}(\alpha') + \sum_{\substack{\beta\text{ arc of both}\\ \Delta\text{ and }\Delta'}} c_{\beta}\,{\boldsymbol f}(\beta) = 0 ,
\end{equation}
where $c_{\alpha},c_{\alpha'},c_{\beta}\in\RR$.
Claim~\ref{claim:main}.(1) is equivalent to the inequality $c_{\alpha} c_{\alpha'} > 0$.
Given the nondegeneracy guaranteed by Claim~\ref{claim:main}.(0), this inequality will hold in general if it holds for one particular choice of geodesic representatives, waists, and widths of the arcs of~$S$ (using Remark~\ref{rem:connected-choices}).
Therefore, it is sufficient to exhibit some choice of geodesic representatives, waists, and widths for which
\begin{equation}\label{eqn:requested-signs-w}
\left \{ \begin{array}{lcl}
  c_{\alpha} & > & 0,\\
  c_{\alpha'} & > & 0,\\
  c_{\beta} & \leq & 0 \quad \text{ for all other arcs }\beta\text{ of }\Delta\text{ and }\Delta'.
\end{array}\right .
\end{equation}
The last inequality will clearly imply Claim~\ref{claim:main}.(2).

The arcs $\alpha,\alpha'$ are the diagonals of a quadrilateral bounded by four arcs $\beta_1, \beta_2, \beta_3, \beta_4$, with $\alpha$ separating $\beta_1, \beta_2$ from $\beta_3, \beta_4$ and $\alpha'$ separating $\beta_2, \beta_3$ from $\beta_4, \beta_1$.
In the following, we show how to choose geodesic representatives, waists, and widths for the arcs of $\Delta$ and~$\Delta'$ so that \eqref{eqn:zerocross} becomes
\begin{equation} \label{eqn:specialzerocross}
\boldsymbol{f}(\alpha) + \boldsymbol{f}(\alpha') - \sum_{\beta\in\{ \beta_1,\beta_2,\beta_3,\beta_4\} } \boldsymbol{f}(\beta) = 0
\end{equation}
(which implies \eqref{eqn:requested-signs-w}).
In particular, all coefficients $c_{\beta}$ in \eqref{eqn:zerocross} vanish except for $\beta=\beta_i$.

Let $\underline{\smash{\beta}}_1,\ldots,\underline{\smash{\beta}}_4\subset S$ be any geodesic representatives of $\beta_1, \ldots, \beta_4$, respectively, and let $R$ be the hyperideal quadrilateral bounded by these four edges.
Let $\tilde{\beta}_1,\ldots,\tilde{\beta}_4\subset\HH^2$ be lifts of these edges bounding a lift $\widetilde{R}$ of~$R$.
The quadrilateral $\widetilde{R}$ is the intersection of $\HH^2$ with the cone spanned by four spacelike vectors $v_1, \ldots, v_4 \in \RR^{2,1}$, where we index the $v_i$ so that
$$\tilde{\beta}_i  = \HH^2 \cap (\RR_+ v_i + \RR_+ v_{i+1})$$
for all $1\leq i\leq 4$, with indices to be interpreted cyclically modulo~$4$ throughout the section (\ie $v_5=v_1$): see Figure~\ref{fig:claim3-1}.
\begin{figure}[ht!]
\labellist
\small\hair 2pt
\pinlabel {$[v_1]$} at 3 3
\pinlabel {$[v_2]$} at 115 3
\pinlabel $[v_3]$ at 115 115
\pinlabel {$[v_4]$} at 3 115
\pinlabel {\large $\delta_1$} at 65 30
\pinlabel {\large $\delta_2$} at 93 47
\pinlabel {\large $\delta_3$} at 52 88
\pinlabel {\large $\delta_4$} at 30 52
\pinlabel $\tilde{\beta}_1$ at 54 14
\pinlabel $\tilde{\beta}_2$ at 104 64
\pinlabel $\tilde{\beta}_3$ at 65 102
\pinlabel $\tilde{\beta}_4$ at 15 64
\pinlabel $e_1$ at 45 35
\pinlabel $e_2$ at 75 51
\pinlabel $e_3$ at 75 82
\pinlabel $e_4$ at 34 77
\pinlabel {$[v]$} at 59 51
\pinlabel {$\bullet$} at 59 58
\endlabellist
\centering
\includegraphics[width=2.4in]{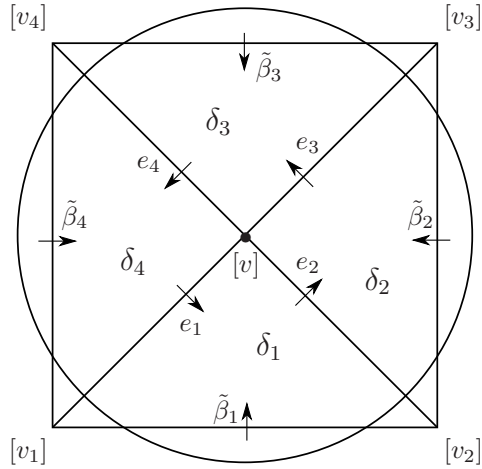}
\caption{View of $\PP(\RR^{2,1})$ in the affine chart $\{ x_3=1\} $. 
The quadrilateral $\widetilde{R}$ is divided into four small tiles $\delta_1, \delta_2, \delta_3, \delta_4$. The set of edges $\mathscr{E}''$ consists of $\mathscr{E}\cap\mathscr{E}'$ together with the four geodesic rays $e_i$ formed by dividing $\tilde{\alpha}$ and $\tilde{\alpha}'$ in half at their intersection point $\{ v\}$. Arrows show transverse orientations.}
\label{fig:claim3-1}
\end{figure}
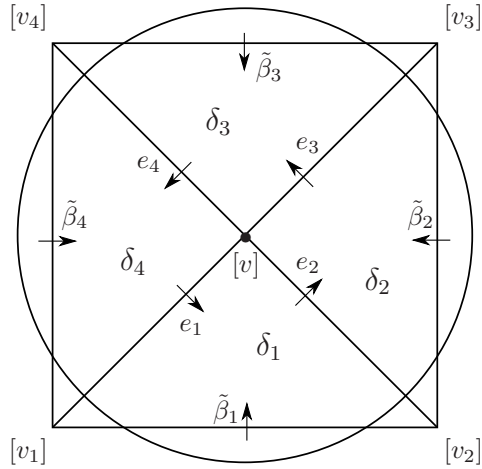
We now choose the geodesic representatives $\underline{\alpha}$ and $\underline{\alpha}'$ so that their lifts $\tilde{\alpha}$ and $\tilde{\alpha}'$ inside $\widetilde{R}$ satisfy
$$\tilde{\alpha} = \HH^2 \cap (\RR_+ v_1 + \RR_+ v_3) \quad\quad\mathrm{and}\quad\quad \tilde{\alpha}' = \HH^2 \cap (\RR_+ v_2 + \RR_+ v_4) .$$
This configuration is achieved, for instance, if all the chosen geodesic representatives of the arcs of~$S$ are perpendicular to the boundary of the convex core: then $v_1, v_2, v_3, v_4$ are dual to the relevant boundary components of the preimage of the convex core in~$\HH^2$.

Let $\underline{\Delta},\underline{\Delta}'$ be geodesic hyperideal triangulations of~$S$ corresponding to $\Delta,\Delta'$, respectively, and containing our chosen geodesic representatives $\underline{\alpha}$, $\underline{\alpha}',\underline{\smash{\beta}}_i$ from above.
We now apply the formalism of Section~\ref{subsec:formalism-strips} to the smallest geodesic cellulation $\underline{\Delta}''$ refining $\underline{\Delta}$ and~$\underline{\Delta}'$.
The vertex set $\mathscr{V}''$ of $\underline{\Delta}''$ is just the one point $\{\underline{v}\}=\underline{\alpha}\cap\underline{\alpha}'$.
The set $\mathscr{T}''$ of polygons consists of the $2|\chi|-2$ hyperideal triangles in $\mathscr{T} \cap \mathscr{T}'$ and of four ``small'' (nonhyperideal) triangles arranged around~$\underline{v}$.
The set $\mathscr{E}''$ of edges consists of $\mathscr{E} \cap \mathscr{E}'$ together with the four geodesic rays formed by cutting $\underline{\alpha}$ and $\underline{\alpha}'$ in half at~$\underline{v}$.
Let $\widetilde{\mathscr{V}}'',\widetilde{\mathscr{E}}'',\widetilde{\mathscr{T}}''$ be the respective preimages in~$\HH^2$ of $\mathscr{V}'',\mathscr{E}'',\mathscr{T}''$.
There are four ``small'' tiles $\delta_1, \delta_2, \delta_3, \delta_4\in\widetilde{\mathscr{T}}''$ that partition~$\widetilde{R}$:
$$\delta_i := \HH^2 \cap (\RR_+v_i + \RR_+ v_{i+1} + \RR_+ v).$$
For any~$i$, the tile $\delta_i$ is bounded by the infinite edge $\tilde{\beta}_i$ together with the two half-infinite edges $e_i$ and~$e_{i+1}$, where
$$e_i := \HH^2 \cap (\RR_+ v_i + \RR_+ v)$$
(see Figure~\ref{fig:claim3-1}).
Note that $\tilde{\alpha} = e_1 \cup e_3$ and $\tilde{\alpha}' = e_2 \cup e_4$.

By multiplying the $v_i\in\RR^{2,1}$ by positive scalars, we may arrange that
\begin{equation}\label{eqn:norm} 
v_1 + v_3 = v_2 + v_4.
\end{equation}
Now, define
$$\varphi(\delta_i) := v_{i+1} - v_i$$
for all~$i$, and extend this to a $\rho$-equivariant (in the sense of \eqref{eqn:rho-equiv}) map $\varphi :\nolinebreak \widetilde{\mathscr{T}}''\to\g$, with value~$0$ outside the $\rho(\Gamma)$-orbits of the~$\delta_i$.
The corresponding $\rho$-equivariant map $\psi :\nolinebreak\pm\widetilde{\mathscr{E}}''\to\g$ describing the relative motion of the tiles, defined as in Section~\ref{subsec:ookkee}, satisfies $\mathscr{L}(\psi) = 0$ because $\varphi$ is $\rho$-equivariant.
By Observation~\ref{obs:realize-strips} and linearity of~$\mathscr{L}$, in order to establish \eqref{eqn:specialzerocross}, it is sufficient to see that for some appropriate choice of the strip waists and widths, we have
\begin{equation}\label{eqn:psi-good-linear-combin}
\psi = \psi_{\alpha} + \psi_{\alpha'} - \sum_{\beta\in\{ \beta_1,\beta_2,\beta_3,\beta_4\} } \psi_{\beta},
\end{equation}
where we interpret $\psi_{\alpha}, \psi_{\alpha'}, \psi_{\beta}$ (Definition~\ref{def:psi-alpha}) as elements of $\Psi(\pm\widetilde{\mathscr{E}}'',\g)$ as in Section~\ref{subsec:formalism-strips}.

We first assume that $\beta_1,\beta_2,\beta_3,\beta_4$ are pairwise distinct.
Endow each $\tilde{\beta}_i$ with the transverse orientation placing $\delta_i$ on the positive side; this makes $\tilde{\beta}_i$ into an element of $\pm\widetilde{\mathscr{E}}''$.
Then
$$\psi(\tilde{\beta}_i) = \varphi(\delta_i) - 0 = v_{i+1} - v_i$$
is, by Lemma~\ref{lem:inf-transl-in-R21}.(3'), an infinitesimal translation along a geodesic of~$\HH^2$ orthogonal to~$\tilde{\beta}_i$ at a point~$\tilde{p}_i$; the translation direction is negative with respect to the transverse orientation.
We choose the waist $p_{\beta_i}\in\underline{\smash{\beta}}_i$ to be the projection to $S=\rho(\Gamma)\backslash\HH^2$ of~$\tilde{p}_i$ and the width $m_{\beta_i} = \sqrt{\langle\psi(\tilde{\beta}_i),\psi(\tilde{\beta}_i)\rangle}$ to be the velocity of the infinitesimal translation $\psi(\tilde{\beta}_i)$.
Then
$$\psi(\tilde{\beta}_i) = -\psi_{\beta_i}(\tilde{\beta}_i)$$
by definition of~$\psi_{\beta_i}$.
Next, we transversely orient the ray $e_i$ from $\delta_{i-1}$ to~$\delta_i$ (see Figure~\ref{fig:claim3-1}).
By \eqref{eqn:norm},
$$\psi(e_i) = \varphi(\delta_i) - \varphi(\delta_{i-1}) = v_{i+1} - 2 v_i + v_{i-1} = v_{i+2} - v_i .$$
By Lemma~\ref{lem:inf-transl-in-R21}.(3'), this implies that $\psi(e_i)$ is an infinitesimal translation along a geodesic of~$\HH^2$ orthogonal to~$e_i$ at some point~$\tilde{q}_i$; the direction of translation is positive with respect to the transverse orientation.
Note that $\psi(e_i) = \psi(-e_{i+2})$, hence $\tilde{q}_i = \tilde{q}_{i+2}$.
We choose the waist $p_{\alpha}\in\underline{\alpha}$ to be the projection to $S$ of $\tilde{q}_1 = \tilde{q}_3$, and the width $m_{\alpha}>0$ to be the velocity of the infinitesimal translation $\psi(\pm e_1) = \psi(\mp e_3)$.
Similarly, we choose the waist and width for $\alpha'$ to be defined by $\psi(\pm e_2) = \psi(\mp e_4)$.
Then
\begin{align*}
\psi(e_1) &= \psi_{\alpha}(e_1) , & \psi(e_2) &= \psi_{\alpha'}(e_2) ,\\
\psi(e_3) &= \psi_{\alpha}(e_3) , & \psi(e_4) &= \psi_{\alpha'}(e_4) .
\end{align*} 
Since $\psi,\psi_{\alpha},\psi_{\alpha'},\psi_{\beta_i}$ all take value~$0$ outside the $\rho(\Gamma)$-orbits of the $\tilde{\beta}_i$ and~$e_i$, we conclude that \eqref{eqn:psi-good-linear-combin} holds.
This establishes \eqref{eqn:specialzerocross}, hence \eqref{eqn:requested-signs-w}, hence Claim~\ref{claim:main}.(1)--(2), in the case that $\beta_1,\beta_2,\beta_3,\beta_4$ are pairwise distinct.

In the case that some of the $\beta_i$ are equal, we still define $\varphi$ as above.
For $1\leq i\leq 4$, if $\beta_i$ is not equal to any other~$\beta_j$, then we choose the waist $p_{\beta_i}$ and the width $m_{\beta_i}$ as above.
If $\beta_i=\beta_j$ for some $1\leq i<j\leq 4$, then $\rho(\gamma)\cdot\tilde{\beta}_i = -\tilde{\beta}_j$ for some $\gamma\in\Gamma$, and
$$\psi(-\tilde{\beta}_j) =  \varphi(\rho(\gamma) \cdot \delta_i) - \varphi(\delta_j) =  \Ad(\rho(\gamma))\,\varphi(\delta_i) - \varphi(\delta_j)$$
is the sum of two infinitesimal translations orthogonal to~$\tilde{\beta}_j$, both positive with respect to the transverse orientation of~$\tilde{\beta}_j$.
Therefore, using Lemma~\ref{lem:inf-transl-in-R21}.(3'), we see that $\psi(-\tilde{\beta}_j)$ is again a positive infinitesimal translation orthogonal to~$\tilde{\beta}_j$.
We choose $\psi_{\beta_i}=\psi_{\beta_j}$ to have waist and width defined by $\psi(-\tilde{\beta}_j)$, so that $\psi_{\beta_j}(\tilde{\beta}_j)=-\psi(\tilde{\beta}_j)$.
Then \eqref{eqn:psi-good-linear-combin} holds as above.
This completes the proof of Claim~\ref{claim:main}.(1)--(2).

Proposition~\ref{prop:mainsteps} is proved, as well as Theorems \ref{thm:main} and~\ref{thm:main-macro}.

\section{Examples}\label{sec:ex}

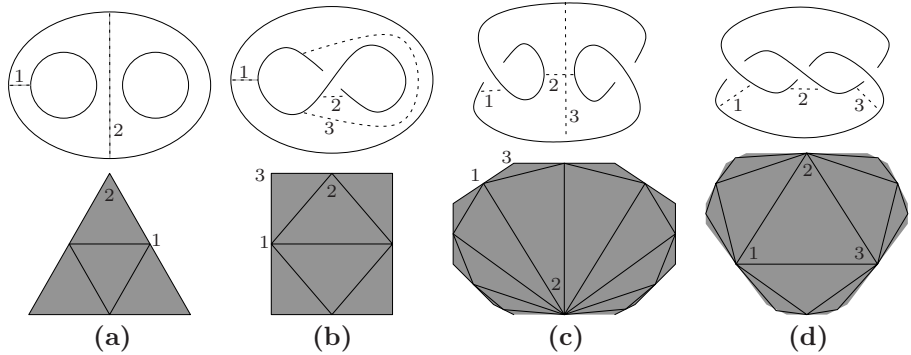
\begin{figure}[ht!]
\labellist
\small\hair 2pt
\pinlabel {\bf (a)} [b] at 47 -18
\pinlabel {\bf (b)} [b] at 145 -18
\pinlabel {\bf (c)} [b] at 249 -18
\pinlabel {\bf (d)} [b] at 357 -18
\pinlabel {${}_1$} at 6 107
\pinlabel {${}_1$} at 67 34
\pinlabel {${}_1$} at 106 109
\pinlabel {${}_1$} at 115 33
\pinlabel {${}_1$} at 215 96
\pinlabel {${}_1$} at 209 63
\pinlabel {${}_1$} at 326 94
\pinlabel {${}_1$} at 333 28
\pinlabel {${}_2$} at 50 83
\pinlabel {${}_2$} at 46 54
\pinlabel {${}_2$} at 147 94
\pinlabel {${}_2$} at 145 56
\pinlabel {${}_2$} at 244 103
\pinlabel {${}_2$} at 245 14
\pinlabel {${}_2$} at 355 97
\pinlabel {${}_2$} at 357 65
\pinlabel {${}_3$} at 144 84
\pinlabel {${}_3$} at 113 63
\pinlabel {${}_3$} at 253 89
\pinlabel {${}_3$} at 223 71
\pinlabel {${}_3$} at 380 94
\pinlabel {${}_3$} at 380 28
\endlabellist
\centering
\includegraphics[width=12cm]{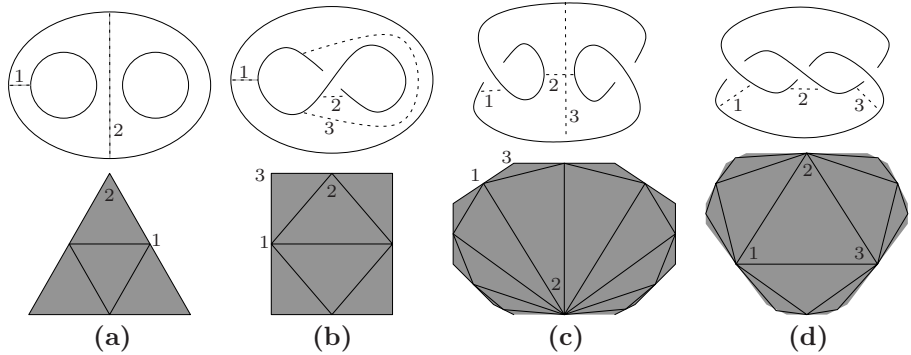}
\vspace{8pt}
\caption{Four surfaces of small complexity (top) and their arc complexes, mapped under $f$ to the closure of $\adm(\rho)$ in an affine chart of $\PP(H^1_{\rho}(\Gamma,\g))$ (bottom). Some arcs are labeled by Arab numerals.}
\label{fig:small} 
\end{figure}

Four noncompact surfaces (two of them orientable) have a 2-dimensional arc complex~$\overline{X}$.
They are represented in Figure~\ref{fig:small}.
Here we summarize some elementary facts about~ $\overline{X}$, and how $\overline{X}$ relates to the geometry of $\adm(\rho)$ when $\rho$ is the holonomy representation of a convex cocompact hyperbolic structure on the surface. 
Margulis spacetimes whose associated hyperbolic surface has one of these four topological types were studied by Charette--Drumm--Goldman: in \cite{cdg10,cdg13}, they gave a similar tiling of $\adm(\rho)$ according to which isotopy classes of crooked planes embed disjointly in the Margulis spacetime. 

\smallskip

\noindent {\bf (a)} \textbf{Thrice-holed sphere}:
The arc complex $\overline{X}$ has $6$ vertices, $9$ edges, $4$ faces.
Its image $f(\overline{X})$ is a triangle whose sides stand in natural bijection with the three boundary components of the convex core of $S=\rho(\Gamma)\backslash\HH^2$: an infinitesimal deformation $u$ of~$\rho$ lies in a side of the triangle if and only if it fixes the length of the corresponding boundary component, to first order.
The set $\adm(\rho)=f(X)$ is the interior of the triangle.

\smallskip

\noindent {\bf (b)} \textbf{Twice-holed projective plane}:
The arc complex $\overline{X}$ has $8$ vertices, $13$ edges, $6$ faces.
Its image $f(\overline{X})$ is a quadrilateral.
The horizontal sides of the quadrilateral correspond to infinitesimal deformations $u$ that fix the length of a boundary component.
The vertical sides correspond to infinitesimal deformations that fix the length of one of the two simple closed curves running through the half-twist.
The set $\adm(\rho)=f(X)$ is the interior of the quadrilateral.

\smallskip

\noindent {\bf (c)} \textbf{Once-holed Klein bottle}:
The arc complex $\overline{X}$ is infinite, with one vertex of infinite degree and all other vertices of degree either $2$ or~$5$.
The closure of $f(\overline{X})$ is an infinite-sided polygon with sides indexed in $\ZZ\cup\{\infty\}$.
The exceptional side has only one point in $f(\overline{X})$, and corresponds to infinitesimal deformations that fix the length of the only nonperipheral, two-sided simple closed curve $\gamma$, which goes through the two half-twists.
The group $\ZZ$ naturally acts on the arc complex $\overline{X}$, via Dehn twists along~$\gamma$.
All nonexceptional sides are contained in $f(\overline{X})$ and correspond to infinitesimal deformations that fix the length of some curve, all these curves being related by some power of the Dehn twist along~$\gamma$.
The set $\adm(\rho)=f(X)$ is the interior of the polygon.

\smallskip

\noindent {\bf (d)} \textbf{Once-holed torus}:
The arc complex $\overline{X}$ is infinite, with all vertices of infinite degree; it is known as the \emph{Farey triangulation}.
The arcs are parameterized by $\PP^1(\QQ)$.
The closure of $f(\overline{X})$ contains infinitely many segments in its boundary. 
These segments, also indexed by $\PP^1(\QQ)$, are in natural correpondence with the simple closed curves.
Only one point of each side belongs to $f(\overline{X})$: namely, the strip deformation along a single arc, which lengthens all curves except the one curve disjoint from that arc.
The group $\GL_2(\ZZ)$ acts on~$\overline{X}$, transitively on the vertices, via the mapping class group of the once-holed torus.
We refer to \cite{glmm} or \cite{PT-note} for more details about $\adm(\rho)$ and its closure in this case.

\begin{remark}
In Examples {\bf (c)} and~{\bf (d)}, where the surface has only one boundary loop $\underline{\gamma}$, the closure of $f(\overline{X})$ in $\PP(H^1_{\rho}(\Gamma,\g))$ does \emph{not} meet the projective line corresponding to infinitesimal deformations that fix the length of~$\underline{\gamma}$.
This is implied by Proposition~\ref{prop:unit-peripheral}.
\end{remark}

\section{Fundamental domains in Minkowski $3$-space}\label{sec:crooked-planes}

In this section, we deduce Theorem~\ref{thm:crooked-planes} (the Crooked Plane Conjecture, assuming convex cocompact linear holonomy) from Theorem~\ref{thm:main} (the parameterization by the arc complex of Margulis spacetimes with fixed convex cocompact linear holonomy).
To begin, we review the construction of crooked planes in Minkowski space, originally due to Drumm~\cite{dru92}.

\subsection{Crooked planes in~$\RR^{2,1}$}\label{subsec:crooked-R21}

A \emph{crooked plane} in~$\RR^{2,1}$, as defined in~\cite{dru92}, is the union of
\begin{itemize}
  \item a \emph{stem}, which is the union of all causal (\ie timelike or lightlike) lines of a given timelike plane that pass through a given point, called the \emph{center};
  \item two \emph{wings}, which are two disjoint open lightlike half-planes whose respective boundaries are the two (lightlike) boundary lines of the~stem.
\end{itemize}
Let us fix some notation.
We see $\HH^2$ as a hyperboloid in~$\RR^{2,1}$ as in \eqref{eqn:H2subsetR21}.
For any future-pointing lightlike vector $v_0\in\RR^{2,1}$, we denote by $\mathcal{W}(v_0)$ the \emph{left wing associated with~$v_0$}: by definition, this is the connected component of $v_0^{\perp}\smallsetminus\RR v_0$ consisting of (spacelike) vectors $w$ that lie ``to the left of~$v_0$ seen from~$\HH^2$'', \ie such that $(v,v_0,w)$ is positively oriented for any $v\in\HH^2 \subset \RR^{2,1}$.
For any geodesic line $\ell$ of~$\HH^2$, with endpoints in $\partial_{\infty}\HH^2\subset\PP(\RR^{2,1})$ represented by future-pointing lightlike vectors $v^+,v^-\in\nolinebreak\RR^{2,1}$, we denote by $\CP(\ell)$ the \emph{left crooked plane centered at $0\in\RR^{2,1}$ associated with~$\ell$}: by definition, this is the union of the stem
$$\mathcal{S}(\ell) := \{ w \in \spa(\ell) \subset \RR^{2,1} ~|~ \langle w, w\rangle \leq 0\} $$
and of the wings $\mathcal{W}(v^+)$ and~$\mathcal{W}(v^-)$ (see Figure~\ref{fig:1CP}).
\begin{figure}[ht!]
\centering
\labellist
\small\hair 2pt
\pinlabel $\mathcal{S}(\ell)$ at 85 25
\pinlabel $\mathcal{W}(v^+)$  at 120 60
\pinlabel $\mathcal{W}(v^-)$ at 50 63
\pinlabel $v^+$  at 78 100
\pinlabel $v^-$ at 108 119
\endlabellist
\includegraphics[scale=1]{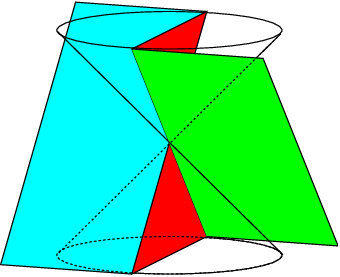}
\caption{The left crooked plane $\CP(\ell)$ in~$\RR^{2,1}$}
\label{fig:1CP}
\end{figure}
A general \emph{left crooked plane} is just a translate $\CP(\ell) + v$ of such a set~$\CP(\ell)$ by some vector $v \in \RR^{2,1}$.
The images of left crooked planes under the orientation-reversing linear map $w\mapsto -w$ are called \emph{right crooked planes}; we will not work directly with them here.

Thinking of $\RR^{2,1}\simeq\g$ as the set of Killing vector fields on~$\HH^2$ as in Section~\ref{subsec:Killing-fields} and using Lemma~\ref{lem:inf-transl-in-R21}, we can describe $\CP(\ell)$ as follows:
\begin{itemize}
  \item the interior of the stem $\mathcal{S}(\ell)$ is the set of elliptic Killing fields on~$\HH^2$ whose fixed point belongs to~$\ell$;
  \item the lightlike line $\RR v^+$, in the boundary of the stem $\mathcal{S}(\ell)$, is $\{ 0\}$ union the set of parabolic Killing fields with fixed point $[v^+]\in\partial_{\infty}\HH^2$, and similarly for~$v^-$;
  \item the wing $\mathcal{W}(v^+)$ is the set of hyperbolic Killing fields with attracting fixed point $[v^+]\in\partial_{\infty}\HH^2$, and similarly for~$v^-$.
\end{itemize}
In other words, $\CP(\ell)\smallsetminus\{ 0\}$ is the set of Killing fields on~$\HH^2$ with a nonrepelling fixed point in~$\overline{\ell}$, where $\overline{\ell}$ is the closure of $\ell$ in $\HH^2\cup\partial_{\infty}\HH^2$.

Any crooked plane divides $\RR^{2,1}$ into two connected components.
Given a transverse orientation of~$\ell$, the \emph{positive crooked half-space} $\CHS^+(\ell)$ (\resp the \emph{negative crooked half-space} $\CHS^-(\ell)$) is the connected component of $\RR^{2,1}\smallsetminus\CP(\ell)$ consisting of nonzero Killing fields on~$\HH^2$ with a nonrepelling fixed point in $(\HH^2\cup\partial_{\infty}\HH^2)\smallsetminus\overline{\ell}$ lying on the positive (\resp negative) side of~$\overline{\ell}$.

\subsection{Disjointness of crooked half-spaces in~$\RR^{2,1}$}

In order to build fundamental domains in~$\RR^{2,1}$ for proper actions of free groups, it is important to understand when two crooked planes are disjoint.
A complete disjointness criterion for crooked planes was given by Drumm--Goldman in \cite{dg99}.
More recently, the geometry of crooked planes and crooked half-spaces was studied in \cite{bcdg13}.
We now recall a sufficient condition due to Drumm.

Let $\ell$ be a transversely oriented geodesic line of~$\HH^2$ and let $v^+,v^-\in\RR^{2,1}$ be future-pointing lightlike vectors representing the endpoints of $\ell$ in $\partial_{\infty}\HH^2$, with $[v^+]$ lying to the left for the transverse orientation.
We shall use the following terminology.

\begin{definition}\label{def:stem-quadrant}
The open cone $\SQ(\ell) := \RR^*_+ v^+ - \RR^*_+ v^-$ of $\mathrm{span}(v^-,v^+)=\mathrm{span}(\ell)$ is called the \emph{stem quadrant} of the transversely oriented geodesic~$\ell$.
\end{definition}

By Lemma~\ref{lem:inf-transl-in-R21}.(3'), the stem quadrant $\SQ(\ell)$ consists of all infinitesimal translations of~$\HH^2$ whose axis is orthogonal to~$\ell$ and oriented in the positive direction. 
The following sufficient condition for disjointness of crooked planes was first proved by Drumm:

\begin{proposition}[Drumm \cite{dru92}]\label{prop:disjoint-CP-R21}
Let $\ell,\ell'$ be two disjoint geodesics of~$\HH^2$, trans\-versely oriented away from each other.\,For any $v\in\mathrm{SQ}(\ell)$ and $v'\in\nolinebreak\mathrm{SQ}(\ell')$,
$$\overline{\CHS^+(\ell)} + v \subset \CHS^-(\ell') + v' ;$$
in particular, the crooked planes $\CP(\ell) + v$ and $\CP(\ell') + v'$ are disjoint.
\end{proposition}

Conversely, for $w\in\RR^{2,1}$, we have $\CP(\ell)\cap (\CP(\ell')+w)=\emptyset$ if and only if $w\in \SQ(\ell')- \SQ(\ell)$ \cite{dg99,bcdg13}.
Thus the space of directions in which one can translate $\CP(\ell')$ to make it disjoint from $\CP(\ell)$ is a convex open cone of~$\RR^{2,1}$ with a quadrilateral basis.

It is clear from the definitions in terms of nonrepelling fixed points of Killing fields that $\overline{\CHS^+(\ell)} \subset \CHS^-(\ell') \cup \{ 0\}$.
Therefore Proposition~\ref{prop:disjoint-CP-R21} is a consequence of the following lemma, applied to $(\ell,v)$ and $(\ell',v')$.

\begin{lemma}\label{lem:parallel-CP-R21}
For any transversely oriented geodesic $\ell$ of~$\HH^2$ and any $v\in\nolinebreak\mathrm{SQ}(\ell)$,
$$\overline{\CHS^+(\ell)} + v \subset \overline{\CHS^+(\ell)} \smallsetminus \{ 0\} .$$
\end{lemma}

\begin{proof}
Let $L^+$ be the closure of the connected component of $(\HH^2\cup\partial_{\infty}\HH^2)\smallsetminus\overline{\ell}$ lying on the positive side of~$\overline{\ell}$ for the transverse orientation, and let $L^-$ be its complement in $\HH^2\cup\partial_{\infty}\HH^2$.
Consider a nonzero Killing vector field $V\in\RR^{2,1}$ (\resp $V'\in\RR^{2,1}$) on~$\HH^2$ with a nonrepelling fixed point in $L^+$ (\resp $L^-$).
The lemma says that if $v\in\RR^{2,1}$ is a hyperbolic Killing field with translation axis orthogonal to~$\ell$, oriented towards $L^+$, then $V+v\notin\{ 0, V'\}$.

Let $\alpha$ be the geodesic line of~$\HH^2$ whose closure contains the nonrepelling fixed points of $V'$ and of~$V$; orient it from the former to the latter.
For $p\in\alpha$, let $\mathrm{pr}_{\alpha} :\nolinebreak T_p\HH^2\rightarrow\nolinebreak\RR$ be the linear form giving the signed length of the projection to~$\alpha$.
By definition of a Killing vector field, for any $Y\in\nolinebreak\RR^{2,1}$ the function $p\mapsto\mathrm{pr}_{\alpha}(Y(p))$ is constant on~$\alpha$; we call its value the \emph{component of $Y$ along~$\alpha$}.
The Killing field $V$ (\resp $V'$) has nonnegative (\resp nonpositive) component along~$\alpha$, because $\alpha$ is oriented towards (\resp away from) the nonrepelling fixed point of $V$ (\resp $V'$).
On the other hand, $v$ has positive component along~$\alpha$: indeed, if the oriented translation axis $\beta$ of~$v$ does not meet~$\alpha$, then this component is $\sqrt{\langle v,v\rangle}\,\cosh d(\alpha,\beta)>0$; otherwise, $\beta$ meets~$\alpha$ at an angle $\theta\in [0,\pi/2)$ and the component is $\sqrt{\langle v,v\rangle}\,\cos\theta > \nolinebreak 0$.
Thus $V+v$ has positive component along~$\alpha$, while $V'$ has nonpositive component, which implies $V+v\notin\{ 0, V'\}$.
\end{proof}

\subsection{Drumm's strategy}\label{subsec:Drumm}

In the early 1990s, Drumm \cite{dru92} introduced a strategy to produce proper affine deformations $u$ of~$\rho$.
We now briefly recall it; see \cite{cg00} for more details.

Begin with a convex cocompact representation $\rho\in\Hom(\Gamma,G)$.
Then~$\rho(\Gamma)$ is a Schottky group, playing ping pong on~$\HH^2$: there is a fundamental domain $\mathcal{F}$ in~$\HH^2$ for the action of $\rho(\Gamma)$ that is bounded by finitely many pairwise disjoint geodesics $\ell_1,\ell'_1,\dots,\ell_r,\ell'_r$, and there is a free generating subset $\{\gamma_1, \ldots, \gamma_r\}$ of~$\Gamma$ such that $\ell'_i=\rho(\gamma_i)\cdot\ell_i$ for all~$i$.
The corresponding left crooked planes centered at the origin in~$\RR^{2,1}$ satisfy $\CP(\ell'_i)=\rho(\gamma_i)\cdot\CP(\ell_i)$.
Now orient transversely each geodesic $\ell_i$ or~$\ell'_i$ away from~$\mathcal{F}$ and translate the corresponding crooked plane $\CP(\ell_i)$ or $\CP(\ell'_i)$ by a vector $v_i$ or $v'_i$ in the corresponding stem quadrant $\SQ(\ell_i)$ or $\SQ(\ell'_i)$.
By Proposition~\ref{prop:disjoint-CP-R21}, the resulting crooked planes are pairwise disjoint and bound a closed region $\mathcal{R}$ in~$\RR^{2,1}$.
The Minkowski isometries that identify opposite pairs of crooked planes generate an affine deformation $\Gamma^{\rho,u}$ of $\rho(\Gamma)$, where $u(\gamma_i)=v'_i-\rho(\gamma_i)\cdot v_i$ for all $1\leq i\leq r$.
(In other words, $u$ comes from propagating the movement of the original crooked planes equivariantly by translating, not only each crooked plane, but the whole closed positive crooked half-space it bounds.)

\begin{remark}
If $S$ is nonorientable, then some elements $\rho(\gamma_i)$ belong to $\SO(2,1)\smallsetminus\SO(2,1)_0$ (corresponding to one-sided loops in~$S$); the associated affine isometries $(\rho(\gamma_i),u(\gamma_i))$ preserve the orientation but reverse the time orientation of~$\RR^{2,1}$.
\end{remark}

By construction, $\mathcal{R}$ is a fundamental domain for the action of $\Gamma^{\rho,u}$ on the union $\Gamma^{\rho,u}\cdot\mathcal{R}$ of all translates of~$\mathcal{R}$.
Drumm proved the following.

\begin{theorem}[Drumm \cite{dru92}]\label{thm:Drumm}
In the setting above, $\Gamma^{\rho,u}\cdot\mathcal{R}=\RR^{2,1}$.
\end{theorem}

In particular, $\Gamma^{\rho,u}$ acts properly on~$\RR^{2,1}$ and $\mathcal{R}$ is a fundamental domain for this action.

In summary, given a fundamental domain in~$\HH^2$ and appropriate motions of the corresponding crooked planes, Drumm's procedure yields a proper cocycle~$u$.
The idea of the proof of Theorem~\ref{thm:crooked-planes} is that Theorem~\ref{thm:main}, correctly interpreted, gives an inverse to Drumm's procedure: a proper cocycle~$u$ is an infinitesimal strip deformation by Theorem~\ref{thm:main}, and this determines disjoint geodesics in~$\HH^2$ (namely the preimage of the support of the strips) that bound a fundamental domain in~$\HH^2$, as well as relative motions of the corresponding left crooked planes in~$\RR^{2,1}$.
These yield a fundamental domain in~$\RR^{2,1}$ bounded by crooked planes.

\subsection{Proof of Theorem~\ref{thm:crooked-planes}: the Crooked Plane Conjecture}\label{subsec:proof-crooked-plane-conj}

By \cite{fg83} and \cite{mes90}, any discrete subgroup of $\OO(2,1)\ltimes\RR^3$ that is not virtually solvable and acts properly discontinuously and freely on~$\RR^{2,1}$ is of the form $\Gamma^{\rho,u}$ as in \eqref{eqn:Gamma-rho-u}, where $\Gamma$ is a free group.
We assume that $\rho$ is convex cocompact.
By \cite{glm09} or \cite{dgk13}, the cocycle $u$ belongs to the admissible cone of~$\rho$ (Definition~\ref{def:GLMcone}).
Note that replacing $u$ with $-u$ amounts to conjugating the $\Gamma^{\rho,u}$-action on~$\RR^{2,1}$ by the orientation-reversing linear map $w\mapsto -w$, which maps left crooked planes to right crooked planes.
Therefore, it is sufficient to consider the case that $u$ belongs to the \emph{positive} admissible cone of~$\rho$, and to prove that in this case there exists a fundamental domain in~$\RR^{2,1}$ for $\Gamma^{\rho,u}$ that is bounded by finitely many \emph{left} crooked planes.

By Theorem~\ref{thm:main}, the cocycle $u$ is an infinitesimal strip deformation supported on some collection $\mathscr{E}$ of geodesic arcs $\underline{\alpha}$ on $S = \rho(\Gamma)\backslash\HH^2$, which cut the surface into topological disks.
We use the notation and formalism of Section~\ref{sec:formalism}.
By Observation~\ref{obs:realize-strips}, the infinitesimal strip deformation $u$ is described by a $(\rho,u)$-equivariant assignment $\varphi : \widetilde{\mathscr{T}}\to\g$ of infinitesimal motions to the tiles, such that for any tiles $\delta,\delta'$ adjacent to an edge $\tilde{\alpha}\in\pm\widetilde{\mathscr{E}}$ transversely oriented from $\delta$ to~$\delta'$, the Killing vector field $\psi(\tilde{\alpha})=\varphi(\delta')-\varphi(\delta)$ is an infinitesimal translation of~$\HH^2$ orthogonal to~$\tilde{\alpha}$, in the positive direction; in other words, $\psi(\tilde{\alpha})$ belongs to the stem quadrant $\mathrm{SQ}(\tilde{\alpha})$ (Definition~\ref{def:stem-quadrant}), by Lemma~\ref{lem:inf-transl-in-R21}.(3').
To any $\tilde{\alpha}\in\widetilde{\mathscr{E}}$ we associate the crooked plane $\mathcal{D}_{\tilde{\alpha}}:=\CP(\tilde{\alpha})+v_{\tilde{\alpha}}$ where
\begin{equation}\label{eqn:average}
v_{\tilde{\alpha}} := \frac{\varphi(\delta) + \varphi(\delta')}{2} .
\end{equation}
One could think of~$v_{\tilde{\alpha}}$ as the motion of the edge $\tilde{\alpha}$ under the infinitesimal deformation~$u$, which we take to be the average of the motions of the adjacent tiles. 
Since $\varphi$ is $(\rho,u)$-equivariant, the map $\tilde{\alpha}\mapsto\mathcal{D}_{\tilde{\alpha}}$ is $(\rho, (\rho,u))$-equivariant, meaning that $\mathcal{D}_{\rho(\gamma)\cdot\tilde{\alpha}} = \rho(\gamma)\cdot\mathcal D_{\tilde{\alpha}} + u(\gamma)$ for all $\tilde{\alpha}\in\widetilde{\mathscr{E}}$ and $\gamma\in\Gamma$.

We claim that the crooked planes $\mathcal{D}_{\tilde{\alpha}}$, for $\tilde{\alpha}\in\widetilde{\mathscr{E}}$, are pairwise disjoint.
Indeed, consider two adjacent edges $\tilde{\alpha},\tilde{\alpha}'$ bounding tiles $\delta,\delta',\delta''$ as in Figure~\ref{fig:diagram-proof-CP}.
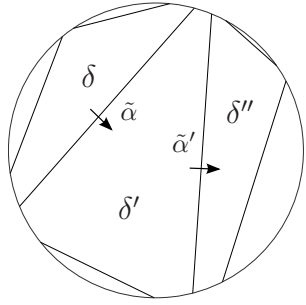
\begin{figure}[ht!]
\centering
\labellist
\small\hair 2pt
\pinlabel {\large $\delta'$} at 120 90
\pinlabel $\tilde{\alpha}$ [l] at 106 185
\pinlabel $\tilde{\alpha}'$ [r] at 192 155
\pinlabel {\large $\delta$}  at 80 220
\pinlabel {\large $\delta''$} at 230 185
\endlabellist
\includegraphics[scale=0.38]{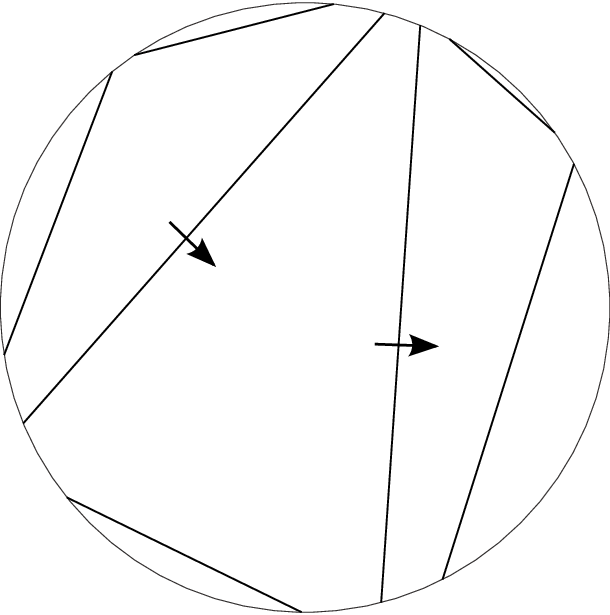}
\caption{Two adjacent edges $\tilde{\alpha},\tilde{\alpha}$ and tiles $\delta,\delta',\delta''$ as in the proofs of Theorems \ref{thm:crooked-planes} and~\ref{thm:AdS-crooked-planes}}
\label{fig:diagram-proof-CP}
\end{figure}
Transversely orient $\tilde{\alpha}$ from $\delta$ to~$\delta'$, and $\tilde{\alpha}'$ from $\delta'$ to~$\delta''$, so that the positive half-plane of $\tilde{\alpha}$ in~$\HH^2$ (\ie the connected component of $\HH^2\smallsetminus\tilde{\alpha}$ lying on the positive side of~$\tilde{\alpha}$ for the transverse orientation) contains that of~$\tilde{\alpha}'$.
Note that
$$v_{\tilde{\alpha}} - \varphi(\delta') = \frac{1}{2}\psi(-\tilde{\alpha}) \quad\quad\mathrm{and}\quad\quad v_{\tilde{\alpha}'} - \varphi(\delta') = \frac{1}{2}\psi(\tilde{\alpha}') .$$
Since $\psi(-\tilde{\alpha}) \in \SQ(-\tilde{\alpha})$ and $\psi(\tilde{\alpha}') \in \SQ(\tilde{\alpha}')$, Proposition~\ref{prop:disjoint-CP-R21} implies
$$\overline{\CHS^+(\tilde{\alpha}')} + v_{\tilde{\alpha}'} - \varphi(\delta') \subset \CHS^-(-\tilde{\alpha}) + v_{\tilde{\alpha}} - \varphi(\delta') ,$$
hence
\begin{equation}\label{eqn:nesting1}
\overline{\CHS^+(\tilde{\alpha}')} + v_{\tilde{\alpha}'} \subset \CHS^-(-\tilde{\alpha}) + v_{\tilde{\alpha}} = \CHS^+(\tilde{\alpha}) + v_{\tilde{\alpha}} .
\end{equation}
In particular, the crooked planes $\mathcal D_{\tilde{\alpha}}$ and $\mathcal D_{\tilde{\alpha}'}$ are disjoint whenever $\tilde{\alpha}, \tilde{\alpha}'$ border the same tile.
Now, let $\tilde{\alpha},\tilde{\alpha}'\in\widetilde{\mathscr{E}}$ be any distinct edges.
Assign transverse orientations so that the positive half-plane of $\tilde{\alpha}$ in~$\HH^2$ contains that of~$\tilde{\alpha}'$.
There is a sequence of transversely oriented edges $\tilde{\alpha} = e_0, e_1, \ldots, e_N = \tilde{\alpha}' \in \pm \widetilde{\mathscr{E}}$ such that any consecutive edges $e_i, e_{i+1}$ border a common tile and the positive half-plane of $e_i$ in~$\HH^2$ contains that of~$e_{i+1}$.
Applying \eqref{eqn:nesting1} to $e_i,e_{i+1}$, we see by induction that
\begin{equation}\label{eqn:nesting2}
\overline{\CHS^+(\tilde{\alpha}')} + v_{\tilde{\alpha}'} \subset \CHS^+(\tilde{\alpha}) + v_{\tilde{\alpha}}.
\end{equation}
In particular, the crooked planes $\mathcal D_{\tilde{\alpha}}$ and $\mathcal D_{\tilde{\alpha}'}$ are disjoint.

To conclude, we note that since the arcs supporting the infinitesimal strip deformation $u$ cut the surface $S=\rho(\Gamma)\backslash\HH^2$ into topological disks, we may choose geodesic arcs $\tilde{\alpha}_1,\tilde{\alpha}'_1,\dots,\tilde{\alpha}_r,\tilde{\alpha}'_r$ from $\widetilde{\mathscr{E}}$ that bound a fundamental domain $\mathcal{F}$ in $\HH^2$ for the action of $\rho(\Gamma)$, and a free generating subset $\{ \gamma_1,\ldots, \gamma_r\} $ of~$\Gamma$ such that $\tilde{\alpha}'_i = \rho(\gamma_i) \cdot \tilde{\alpha}_i$ for all~$i$.
By \eqref{eqn:nesting2}, the crooked planes $\mathcal{D}_{\tilde{\alpha}_i}$ and $\mathcal{D}_{\tilde{\alpha}'_i}$, for $1\leq i\leq r$, are pairwise disjoint and bound a closed, connected region $\mathcal{R}$ in~$\RR^{2,1}$.
For any~$i$, the element $(\rho(\gamma_i), u(\gamma_i))\in\Gamma^{\rho,u}$ identifies $\mathcal{D}_{\tilde{\alpha}_i}$ with~$\mathcal{D}_{\tilde{\alpha}'_i}$.
Therefore, $\mathcal{R}$ is a fundamental domain for the action of $\Gamma^{\rho,u}$ on $\Gamma^{\rho,u}\cdot\mathcal{R}$.
That $\Gamma^{\rho,u}\cdot\mathcal{R}=\RR^{2,1}$ then follows from Theorem~\ref{thm:Drumm}, or alternatively from Lemma~\ref{lem:explode-in-p} below, which implies that the crooked planes $\mathcal{D}_{\tilde{\alpha}}$, for $\tilde{\alpha} \in \widetilde{\mathscr{E}}$, do not accumulate on any set.

\begin{lemma}\label{lem:explode-in-p}
For any $p\in\HH^2$ and any sequence $(\tilde{\alpha}_n)\in\widetilde{\mathscr{E}}^{\NN}$ going to infinity,
$$\inf\big\{ \Vert V_n(p)\Vert ~|~ V_n\in\mathcal{D}_{\tilde{\alpha}_n}\big\} \underset{\scriptstyle n\rightarrow +\infty}{\longrightarrow}  +\infty .$$
\end{lemma}

\begin{proof}
It is enough to treat the case that there exists a sequence $(\delta_n)_{n\in\NN}$ of distinct elements of~$\widetilde{\mathscr{T}}$ (tiles) such that $p\in\delta_0$ and $\tilde{\alpha}_n\in\widetilde{\mathscr{E}}$ is adjacent to $\delta_n$ and~$\delta_{n-1}$ for all $n\geq 1$.
We transversely orient $\tilde{\alpha}_n$ towards~$\delta_n$.
Consider a Killing field $V_n=Y_n+ v_{\tilde{\alpha}_n}\in\mathcal{D}_{\tilde{\alpha}_n}$, where $Y_n\in\nolinebreak\CP(\tilde{\alpha}_n)$.
By definition, the Killing field $Y_n$ admits a nonrepelling fixed point $q_n$ in the closure of~$\tilde{\alpha}_n$ in $\HH^2\cup\partial_{\infty}\HH^2$.
Let $\ell_n$ be the geodesic line of~$\HH^2$ whose closure contains $p$ and~$q_n$; orient it from the former to the latter.
The component (see the proof of Lemma~\ref{lem:parallel-CP-R21}) of~$Y_n$ along~$\ell_n$ is nonnegative.
Moreover, for any $i\leq n$ the line $\ell_n$ crosses $\tilde{\alpha}_i$ in the positive direction (at a point~$r_i$), and so the Killing field $\psi(\tilde{\alpha}_i)$ has nonnegative component along~$\ell_n$, equal to $\Vert\psi(\tilde{\alpha}_i)(r_i)\Vert \, \sin \measuredangle_{r_i}(\ell_n, \tilde{\alpha}_i)$.
In fact, for $2\leq i \leq n-1$ the angle $\measuredangle_{r_i}(\ell_n, \tilde{\alpha}_i)$ is bounded below by a positive constant depending only on the geometry of the tiles, and $\Vert \psi(\tilde{\alpha}_i)(r_i)\Vert$ is at least the width of the infinitesimal strip deformation along~$\tilde{\alpha}_i$.
It follows that the component along~$\ell_n$ of the Killing~field
$$V_n=Y_n+ \varphi(\delta_0) + \psi(\tilde{\alpha}_1) + \psi(\tilde{\alpha}_2)+ \dots + \psi(\tilde{\alpha}_{n-1})+\frac{1}{2}\psi(\tilde{\alpha}_n)$$
goes to infinity as $n\rightarrow +\infty$.
This completes the proof.
\end{proof}

\begin{remark}
In \eqref{eqn:average} above, we could have taken
$$v_{\tilde{\alpha}} := (1-t)\,\varphi(\delta) + t\,\varphi(\delta')$$
for an arbitrary fixed $t\in (0,1)$, not necessarily $t=1/2$; the proof would have worked the same way.
For each arc supporting the strip deformation, there is an interval's worth of parallel crooked planes (for $t$ varying in $(0,1)$), each embedded in the Margulis spacetime; their union is a \emph{parallel crooked slab}, as defined in \cite{cdg13}.
Crooked planes from different parallel crooked slabs never intersect; crooked planes in the same parallel crooked slab are tangent along subsets of their stems.
\end{remark}

\section{Fundamental domains in anti-de Sitter $3$-space}\label{sec:AdS}

In this section we introduce piecewise totally geodesic surfaces in $\AdS$ analogous to the crooked planes of Section~\ref{subsec:crooked-R21}.
We establish a sufficient condition for disjointness similar to Proposition~\ref{prop:disjoint-CP-R21} and prove Theorem~\ref{thm:AdS-crooked-planes}.

\subsection{$\AdSS$ crooked planes}\label{subsec:AdS-crooked-planes}

As in Minkowski space, we define a crooked plane in $\AdS$ to be the union of three pieces (see Figure~\ref{fig:2CP}):
\begin{itemize}
  \item a \emph{stem}, defined to be the union of all causal (\ie timelike or lightlike) geodesics of a given timelike plane of $\AdS$ that pass through a given point, called the \emph{center} of the $\AdSS$ crooked plane;
  \item two \emph{wings}, defined to be two disjoint open lightlike half-planes of $\AdS$ whose respective boundaries are the two (lightlike) boundary lines of the stem.
\end{itemize}
\begin{figure}[ht!]
\centering
\includegraphics[scale=0.8]{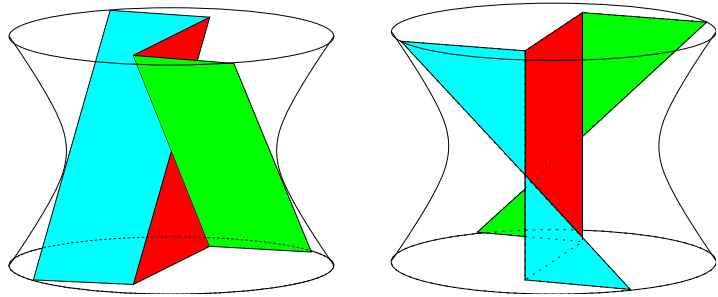}
\caption{A left $\AdSS$ crooked plane, seen in two different affine charts of $\PP^3(\RR)\supset\AdS$. The stem (red) is a bigon whose closure meets the boundary of $\AdS$ in two points. On the left, these two points are at infinity; on the right, the center of the stem is at infinity. Each wing (green or blue) is itself a bigon, bounded by a line contained in the boundary of $\AdS$ and a line of the stem.}
\label{fig:2CP}
\end{figure}
As in Minkowski space, an $\AdSS$ crooked plane centered at the identity is determined by a geodesic line $\ell$ of~$\HH^2$ and a choice of orientation (left or right): we denote by $\ACP(\ell)$ the \emph{left $\AdSS$ crooked plane centered at $e\in G_0$ associated with~$\ell$}, which is described explicitly as follows:
\begin{itemize}
  \item the interior of the stem $\mathsf{S}(\ell)$ of $\ACP(\ell)$ is the set of elliptic elements $h\in G_0$ whose fixed point belongs to~$\ell$;
  \item the boundary of the stem $\mathsf{S}(\ell)$ is $\{ e\}$ union the set of parabolic elements $h\in G_0$ fixing one of the two endpoints $[v^+],[v^-]$ of $\ell$ in $\partial_{\infty}\HH^2$;
  \item the wings of $\ACP(\ell)$ are $\mathsf{W}(v^+)$ and $\mathsf{W}(v^-)$, where $\mathsf{W}(v^+)$ is the set of hyperbolic elements $h\in G_0$ with attracting fixed point $[v^+]$,~and~simi\-larly for~$v^-$.
\end{itemize}
In other words, $\ACP(\ell)\smallsetminus\{ e\}$ is the set of orientation-preserving isometries of~$\HH^2$ (\ie elements of $G_0\simeq\AdS$) with a nonrepelling fixed point in~$\overline{\ell}$, where $\overline{\ell}$ is the closure of $\ell$ in $\HH^2\cup\partial_{\infty}\HH^2$, as in Section~\ref{subsec:crooked-R21}.
Note that this is exactly the image under the exponential~map of the crooked plane $\CP(\ell)\subset\RR^{2,1}$ from Section~\ref{subsec:crooked-R21} (see also \cite{gol13}). We also have $\ACP(g\cdot \ell)=g\ACP(\ell)g^{-1}$ for all $g\in G_0$.

A general \emph{left $\AdSS$ crooked plane} is a $G_0$-translate (on either side) of $\ACP(\ell)$ for some geodesic $\ell$ of~$\HH^2$.
The images of left $\AdSS$ crooked planes under the orientation-reversing isometry $g\mapsto g^{-1}$ of $\AdS$ are called \emph{right $\AdSS$ crooked planes}; we will not work with them in this paper.

Similarly to the Minkowski setting, an $\AdSS$ crooked plane divides $\AdS$ into two connected components (see Figure~\ref{fig:2CP}).
Note that by contrast a timelike geodesic plane in $\AdS$ \emph{does not} divide $\AdS$ into two components: it is one-sided (topologically a M\"obius strip).
Given a transverse orientation of~$\ell$, we denote by $\ACHS^+(\ell)$ (\resp $\ACHS^-(\ell)$) the connected component of $\AdS\smallsetminus\nolinebreak\ACP(\ell)$ consisting of nontrivial elements $g\in G_0$ with a nonrepelling fixed point in $(\HH^2\cup\partial_{\infty}\HH^2)\smallsetminus\overline{\ell}$ lying on the positive (\resp negative) side of~$\overline{\ell}$.

\begin{remark}
The Minkowski space $\RR^{2,1}$ and the anti-de Sitter space $\AdS$ can both be embedded into the $3$-dimensional \emph{Einstein space} $\mathrm{Ein}^3$.
By \cite{gol13}, the closure in $\mathrm{Ein}^3$ of a Minkowski or $\AdSS$ crooked plane is a \emph{crooked surface} in the sense of Frances \cite{fra03}.
Note that Drumm's strategy from Section~\ref{subsec:Drumm} has recently been carried out in the Einstein setting in \cite{cfl13}, although no complete disjointness criterion is known for the moment.
\end{remark}

\subsection{Disjointness for left $\AdSS$ crooked planes}\label{subsec:disjoint-AdS-crooked-planes}

In \cite{dgk-crooked-AdS} we give a complete disjointness criterion for $\AdSS$ crooked planes.
Here we only need the following sufficient condition analogous to Proposition~\ref{prop:disjoint-CP-R21}; as before, $\mathrm{SQ}(\ell)$ denotes the stem quadrant of~$\ell$ (Definition~\ref{def:stem-quadrant}).

\begin{proposition}\label{prop:disjoint-CP-AdS}
Let $\ell,\ell'$ be two disjoint geodesics of~$\HH^2$, transversely oriented away from each other.
For any $g\in\exp(\mathrm{SQ}(\ell))$ and $g'\in\exp(\mathrm{SQ}(\ell'))$,
$$g\,\overline{\ACHS^+(\ell)} \subset g' \ACHS^-(\ell') ;$$
in particular, the crooked planes $g\ACP(\ell)$ and $g'\ACP(\ell')$ are disjoint.
\end{proposition}

It is clear from the definitions in terms of nonrepelling fixed points of isometries of~$\HH^2$ that $\ACHS^+(\ell)\subset\ACHS^-(\ell')$.
Therefore Proposition~\ref{prop:disjoint-CP-AdS} is a consequence of the following lemma, applied to $(\ell,g)$ and $(\ell',g')$.

\begin{lemma}\label{lem:parallel-CP-AdS}
For any transversely oriented geodesic $\ell$ of~$\HH^2$ and any element $g\in\exp(\mathrm{SQ}(\ell))$,
$$g \, \overline{\ACHS^+(\ell)}\subset\ACHS^+(\ell) .$$
\end{lemma}

Note that, in the analogy between the Minkowski and anti-de Sitter settings, Lemma~\ref{lem:parallel-CP-AdS} is slightly stronger than Lemma~\ref{lem:parallel-CP-R21}: for $g\in\exp(\mathrm{SQ}(\ell))$ the intersection $g\ACP(\ell)\cap\ACP(\ell)$ is empty, whereas for $v\in\mathrm{SQ}(\ell)$ the intersection $(\CP(\ell)+v)\cap\CP(\ell)$ is the union of two affine subcones of the stem.
This is the subject of the following remark.

\begin{remark}
In~$\RR^{2,1}$, a crooked plane meets any translate of itself.
In $\AdS$, the crooked plane $\ACP(\ell)$ does not meet its translates $g\ACP(\ell)$ for $g\in\exp(\mathrm{SQ}(\ell))$: these are obtained from $\ACP(\ell)$ by sliding \emph{and tilting} (see Figure~\ref{fig:crookedAdS}).
\end{remark}

\begin{figure}[ht!]
\centering
\labellist
\small\hair 2pt
\endlabellist
\includegraphics[width=8cm]{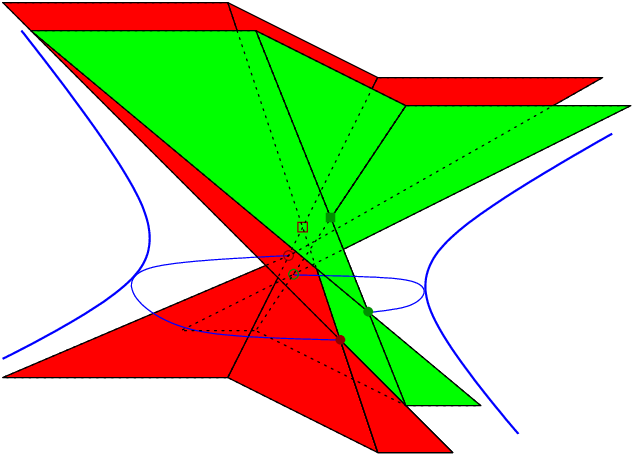}
\caption{Two disjoint $\AdSS$ crooked planes, green and red, in an affine chart of $\PP^3(\RR)$ truncated above and below. The two centers are marked by square dots. The closures of the stems meet the boundary of $\AdS$ in four points marked by round dots. Dual to each center is a copy of~$\HH^2$, part of~whose boundary at infinity of $\AdS$ is also shown (blue ellipse arcs).}
\label{fig:crookedAdS}
\end{figure}

\begin{proof}[Proof of Lemma~\ref{lem:parallel-CP-AdS}]
Let $L^+$ be the closure of the connected component of $(\HH^2\cup\partial_{\infty}\HH^2)\smallsetminus\overline{\ell}$ lying on the positive side of~$\overline{\ell}$ for the transverse orientation, and let $L^-$ be its complement in $\HH^2\cup\partial_{\infty}\HH^2$.
Consider an element $h\in G_0$ (\resp $h'\in G_0$) with a nonrepelling fixed point in $L^+$ (\resp $L^-$).
The lemma says that if $g\in G_0$ is hyperbolic with translation axis orthogonal to~$\ell$, oriented towards $L^+$, then $gh\neq h'$.

Note that for any $p\in L^+\cap\HH^2$ and $p'\in L^-\cap\HH^2$,
\begin{equation}\label{eqn:g-moves-p-away}
d(g\cdot p,p') - d(p,p') \geq \eta := 2\log\cosh\bigg(\frac{\lambda(g)}{2}\bigg) > 0,
\end{equation}
where $\lambda(g)>0$ is the translation length of $g$ in~$\HH^2$; see below for a proof.
This inequality remains true when $p$ is either a point of $L^+\cap\HH^2$ or a horoball centered in $L^+\cap\partial_{\infty}\HH^2$, and $p'$ is either a point of $L^-\cap\HH^2$ or a horoball centerered in $L^-\cap\partial_{\infty}\HH^2$, with $p$ and~$p'$ disjoint: indeed, the (signed) distance function to a given horosphere $q$ of~$\HH^2$ is a Busemann function, of the form $\lim_{t\to +\infty} d(\cdot,q_t)-t$ where $(q_t)_{t\geq 0}$ is a geodesic ray from $q$ to its center in $\partial_{\infty}\HH^2$; by continuity, the inequality $d(g\cdot q_t,q'_t) -  d(q_t,q'_t)\leq -\eta$ passes to the limit for $q_t\in L^+$ and $q_t\in L^-$.
Thus, if $p$ (\resp $p'$) is a singleton or horoball of $\HH^2$ with $h\cdot p\subset p$ (\resp $h'\cdot p'\subset p'$), then
$$d(gh\cdot p, h'\cdot p') \geq d(g\cdot p, p') \geq d(p,p') + \eta .$$
It follows that $gh$ cannot equal~$h'$.
\end{proof}

\begin{proof}[Proof of \eqref{eqn:g-moves-p-away}]
Note that $p', p, g\cdot p$ project in that order to the oriented translation axis $\mathcal{A}$ of~$g$.
Thus, without loss in generality, we may assume that $p\in\ell$.
Let $p''$ be the intersection point of $\ell$ with the geodesic line through $g\cdot p$ and~$p'$.
Choose points $q, q' \in \ell$ so that $q, p, p'', q'$ lie in that order along~$\ell$ (possibly $p=p''$).
We refer to Figure~\ref{fig:proof8-1}.
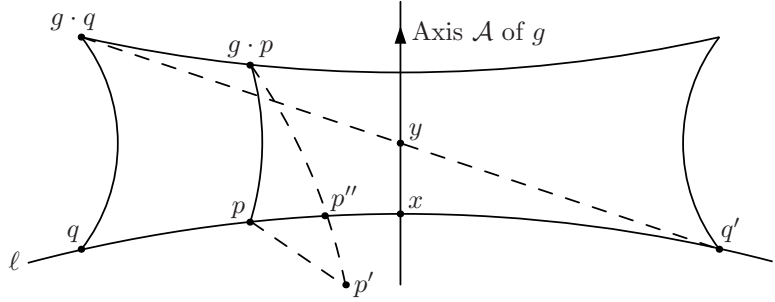
\begin{figure}[ht!]
\centering
\labellist
\small\hair 2pt
\pinlabel {Axis $\mathcal{A}$ of $g$} at 115 66
\pinlabel {$g\cdot q$} at 12 69
\pinlabel {$q$} at 12 15
\pinlabel {$g\cdot p$} at 57 62
\pinlabel {$p$} at 54 21
\pinlabel {$x$} at 99 23
\pinlabel {$p''$} at 81 24
\pinlabel {$p'$} at 86 2
\pinlabel {$q'$} at 179 16
\pinlabel {$y$} at 99 41
\pinlabel {$\ell$} at -3 8
\endlabellist
\includegraphics[width=10cm]{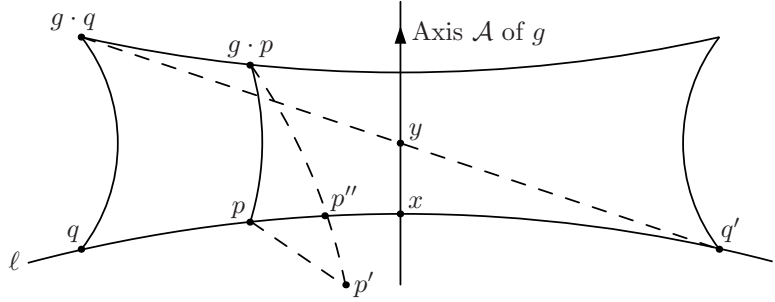}
\caption{Illustration of the proof of \eqref{eqn:g-moves-p-away}}
\label{fig:proof8-1}
\end{figure}
By the triangle inequality,
\begin{eqnarray*}
& & d(g\cdot p,p') - d(p,p')\\
& \geq & d(g\cdot p,p') - d(p,p'') - d(p'',p') = d(g\cdot p,p'') - d(p,p'')\\
& \geq & d(g\cdot p,q') - d(q',p'') - d(p,p'') = d(g\cdot p,q') - d(p,q')\\
& \geq & d(g\cdot q,q') - d(g\cdot q,g\cdot p) - d(p,q') = d(g\cdot q,q') - d(q,q').
\end{eqnarray*}
Define $\{x\}=\ell\cap\mathcal{A}$ and now take $q$ and~$q'$ to be at the same distance $t>0$ from~$x$, on opposite sides.
Let $y$ be the intersection point of $\mathcal{A}$ with the line through $g\cdot q$ and~$q'$.
Then $d(x,y)=\lambda(g)/2$, and so the cosine formula in the right-angled triangle $yxq'$ yields
$$d(g\cdot q,q') = 2 \, d(y,q') = 2 \, \mathrm{arccosh}\big(e^{\eta/2} \cosh t\big),$$
where $\eta:=2\log\cosh(\lambda(g)/2)$.
On the other hand, $d(q,q')=2t$, so it is sufficient to see that $2\, \mathrm{arccosh}(e^{\eta/2}\cosh t)-2t\geq\eta$ for all $t\geq 0$.
This follows from the fact that the function $t\mapsto 2\,\mathrm{arccosh}(e^{\eta/2}\cosh t)-t$ is decreasing on~$\RR_+^{\ast}$, with limit $\eta$ at infinity.
\end{proof}

\subsection{Crooked fundamental domains in $\AdSS$ obtained from strip deformations}\label{subsec:proof-crooked-plane-conj-AdS}

We now deduce Theorem~\ref{thm:AdS-crooked-planes} from Theorem~\ref{thm:main-macro}.

Let $\rho,j\in\Hom(\Gamma,G)$ be the holonomy representations of two convex cocompact hyperbolic structures on a fixed surface.
Assume that $\Gamma^{\rho,j}$ acts properly discontinuously on $\AdS$.
By \cite{kasPhD}, up to switching $j$ and~$\rho$ we may assume that $j$ is ``uniformly longer'' than~$\rho$ in the sense that \eqref{eqn:propcritAdS} is satisfied.
This implies \cite{thu86a} that the surface is not compact.
Note that switching $j$ and~$\rho$ amounts to conjugating the $\Gamma^{\rho,j}$-action on $\AdS=G_0$ by the orientation-reversing isometry $g\mapsto g^{-1}$, which maps left $\AdSS$ crooked planes to right $\AdSS$ crooked planes.
Assuming that $j$ is ``uniformly longer'' than~$\rho$, we shall prove the existence of a fundamental domain in $\AdS$ for $\Gamma^{\rho,j}$ that is bounded by finitely many \emph{left} $\AdSS$ crooked planes.

By Theorem~\ref{thm:main-macro}, we can realize~$j$ as a strip deformation of~$\rho$ supported on some collection $\mathscr{E}$ of geodesic arcs $\underline{\alpha}$ on the hyperbolic surface $S := \rho(\Gamma)\backslash\HH^2$, which cut the surface into topological disks.
We proceed as in the infinitesimal case, again using the notation of Section~\ref{sec:formalism}.
Similarly to Observation~\ref{obs:realize-strips}, the strip deformation taking $\rho$ to~$j$ is described by an assignment $\Phi : \widetilde{\mathscr{T}}\to G_0$ of motions to the tiles satisfying the following:
\begin{itemize}
  \item $\Phi$ is $(\rho,(\rho,j))$-equivariant: for all $\gamma\in\Gamma$ and $\delta\in\widetilde{\mathscr{T}}$,
  $$\Phi(\rho(\gamma)\cdot\delta) = j(\gamma) \, \Phi(\delta) \, \rho(\gamma)^{-1} ;$$
  \item for any transversely oriented geodesic $\tilde{\alpha} \in \pm\widetilde{\mathscr{E}}$, bordering tiles $\delta, \delta'$ on the negative and positive sides respectively, the relative displacement
  $$\Psi(\tilde{\alpha}) := \Phi(\delta)^{-1}\Phi(\delta') \in G_0$$
  is hyperbolic with translation axis orthogonal to~$\tilde{\alpha}$, oriented in the~po\-sitive direction; equivalently, $\Psi(\tilde{\alpha}) \in \exp (\SQ(\tilde{\alpha}))$ by Lemma~\ref{lem:inf-transl-in-R21}.(3').
  The translation length of~$\Psi(\tilde{\alpha})$ is the width of the strip to be inserted in~$S$ along the projection of~$\tilde{\alpha}$.
\end{itemize}
To any $\tilde{\alpha}\in\widetilde{\mathscr{E}}$ we associate the $\AdSS$ crooked plane $\mathsf{D}_{\tilde{\alpha}} := g_{\tilde{\alpha}} \ACP(\tilde{\alpha})$ where
\begin{equation}\label{eqn:averageAdS}
g_{\tilde{\alpha}} := \Phi(\delta)\,\sqrt{\Phi(\delta)^{-1}\,\Phi(\delta')}=\Phi(\delta')\,\sqrt{\Phi(\delta')^{-1}\,\Phi(\delta)} .
\end{equation}
Here the square root of a hyperbolic element denotes the hyperbolic element with the same oriented axis but half the translation length.
One could think of $g_{\tilde{\alpha}}$ as the motion of the edge $\tilde{\alpha}$ under the strip deformation, which we take to be the average of the motions of the adjacent tiles.
Since $\Phi$ is $(\rho,(\rho,j))$-equivariant, $\Psi$ is $(\rho,(\rho,\rho))$-equivariant and
$$\mathsf{D}_{\rho(\gamma)\cdot\tilde{\alpha}} = j(\gamma) \, \mathsf{D}_{\tilde{\alpha}} \, \rho(\gamma)^{-1}$$
for all $\gamma\in\Gamma$ and $\tilde{\alpha}\in\widetilde{\mathscr{E}}$.
We claim that the $\mathsf{D}_{\tilde{\alpha}}$, for $\tilde{\alpha}\in\widetilde{\mathscr{E}}$, are pairwise disjoint.
Indeed, consider two adjacent edges $\tilde{\alpha},\tilde{\alpha}'$ bounding tiles $\delta,\delta',\delta''$ as in Figure~\ref{fig:diagram-proof-CP}.
Transversely orient $\tilde{\alpha}$ from $\delta$ to~$\delta'$, and $\tilde{\alpha}'$ from $\delta'$ to~$\delta''$ so that the positive half-plane of $\tilde{\alpha}$ in~$\HH^2$ (\ie the connected component of $\HH^2\smallsetminus\tilde{\alpha}$ lying on the positive side of~$\tilde{\alpha}$ for the transverse orientation) contains that of~$\tilde{\alpha}'$.
Then
$$\Phi(\delta')^{-1} g_{\tilde{\alpha}} = \sqrt{\Psi(\tilde{\alpha})^{-1}} = \sqrt{\Psi(-\tilde{\alpha})} \quad\mathrm{and}\quad \Phi(\delta')^{-1}g_{\tilde{\alpha}'} = \sqrt{\Psi(\tilde{\alpha}')} .$$
Since $\Psi(-\tilde{\alpha})\in\exp(\SQ(-\tilde{\alpha}))$ and $\Psi(\tilde{\alpha}')\in\exp(\SQ(\tilde{\alpha}'))$, Proposition~\ref{prop:disjoint-CP-AdS}~implies
$$\Phi(\delta')^{-1}g_{\tilde{\alpha}'}\overline{\ACHS^+(\tilde{\alpha}')} \subset \Phi(\delta')^{-1}g_{\tilde{\alpha}}\ACHS^-(-\tilde{\alpha}) ,$$
hence
\begin{equation}\label{eqn:AdS-nesting1}
g_{\tilde{\alpha}'}\overline{\ACHS^+(\tilde{\alpha}')} \subset g_{\tilde{\alpha}}\ACHS^-(-\tilde{\alpha}) = g_{\tilde{\alpha}} \ACHS^+(\tilde{\alpha}) .
\end{equation}
This shows in particular that the crooked planes $\mathsf{D}_{\tilde{\alpha}}$ and $\mathsf{D}_{\tilde{\alpha}'}$ are disjoint whenever $\tilde{\alpha}, \tilde{\alpha}'$ border the same tile.
As in the Minkowski setting, induction and \eqref{eqn:AdS-nesting1} allow us to conclude that for any edges $\tilde{\alpha}, \tilde{\alpha}' \in \pm \widetilde{\mathscr{E}}$, transversely oriented so that the positive half-plane of $\tilde{\alpha}$ in~$\HH^2$ contains that of~$\tilde{\alpha}'$,
\begin{equation}\label{eqn:AdS-nesting2}
\overline{\ACHS^+(\tilde{\alpha}')} \subset \ACHS^+(\tilde{\alpha}).
\end{equation}
A candidate fundamental domain $\mathsf{R}$ is then defined as the intersection of the crooked half-spaces corresponding to the half-planes defining a fundamental domain in $\HH^2$ for $\rho(\Gamma)$.
The proof concludes by showing that $\Gamma^{\rho, j}\cdot\mathsf{R}=\AdS$.
This is implied by the following analogue of Lemma~\ref{lem:explode-in-p}, which shows that the crooked planes $\mathsf{D}_{\tilde{\alpha}}$, for $\tilde{\alpha}\in\widetilde{\mathscr{E}}$, do not accumulate on any set.


\begin{lemma}\label{lem:explode-in-p-AdS}
For any $p\in\HH^2$ and any sequence $(\tilde{\alpha}_n)\in\widetilde{\mathscr{E}}^{\NN}$ going to infinity,
$$\inf\big\{ d(p,h_n\cdot p) ~|~ h_n\in\mathsf{D}_{\tilde{\alpha}_n}\big\} \underset{\scriptstyle n\rightarrow +\infty}{\longrightarrow} +\infty .$$
\end{lemma}

\begin{proof}
Let $\tilde{\alpha}, \tilde{\alpha}' \in \widetilde{\mathscr{E}}$ be distinct edges bounding a common tile~$\delta'$, transversely oriented as in Figure~\ref{fig:diagram-proof-CP}.
By construction, $g_{\tilde{\alpha}}^{-1}g_{\tilde{\alpha}'}=\sqrt{\Psi(\tilde{\alpha})}\sqrt{\Psi(\tilde{\alpha}')}$.
Applying \eqref{eqn:g-moves-p-away} twice, we see that for any $p \in \tilde{\alpha}$ and $p' \in \tilde{\alpha}'$,
\begin{equation}\label{eqn:g-alpha-alpha'}
d(g_{\tilde{\alpha}} \cdot p, g_{\tilde{\alpha}'} \cdot p') \geq d(p,p') + \eta(\tilde{\alpha}) + \eta(\tilde{\alpha}') ,
\end{equation}
where $\eta(\tilde{\alpha}) := 2 \log \cosh(\lambda(\sqrt{\Psi(\tilde{\alpha})})/2) > 0$ depends only on the width of the strip along~$\tilde{\alpha}$.
The inequality remains true when $p$ is either a point of~$\tilde{\alpha}$ or a horoball centered at an endpoint of~$\tilde{\alpha}$, and $p'$ is either a point of~$\tilde{\alpha}'$ or a horoball centerered at an endpoint of~$\tilde{\alpha}'$, with $p$ and~$p'$ disjoint.

To prove the lemma, it is enough to treat the case that $(\tilde{\alpha}_n)\in\widetilde{\mathscr{E}}^{\NN}$ is a sequence of distinct edges, each separated from the next by just one tile, and that $p=p_0\in\tilde{\alpha}_0$.
For $n\geq 1$, let $p_n$ be any point of~$\tilde{\alpha}_n$ or any horoball centered at an endpoint of~$\tilde{\alpha}_n$, disjoint from~$p_0$.
The shortest path from $g_{\tilde{\alpha}_0} \cdot p_0$ to $g_{\tilde{\alpha}_n} \cdot p_n$ intersects each $g_{\tilde{\alpha}_i}\cdot \tilde{\alpha}_i$ at a point $g_{\tilde{\alpha}_i}\cdot p_i$.
Applying \eqref{eqn:g-alpha-alpha'} to each subsegment, we find
\begin{eqnarray*}
d(p_0,p_n) & \leq & \sum_{i=1}^n\, d(p_{i-1},p_i) \;\;\leq\;\; \sum_{i=1}^n\, \big( d(g_{\tilde{\alpha}_{i-1}}\cdot p_{i-1}, g_{\tilde{\alpha}_i}\cdot p_i)-2\eta_0 \big)\\
&=& d(g_{\tilde{\alpha}_0}\cdot p_0, g_{\tilde{\alpha}_n}\cdot p_n) - 2n\eta_0,
\end{eqnarray*}
where $\eta_0>0$ is the smallest of the finitely many values $\eta(\tilde{\alpha})$. 
Up to conjugation we may assume $g_{\tilde{\alpha}_0} = e$.
Let $h = g_{\tilde{\alpha}_n} h' \in \mathsf{D}_{\tilde{\alpha}_n}$, where $h' \in \ACP(\tilde{\alpha}_n)$.
By definition of~$\ACP(\tilde{\alpha}_n)$, there is a singleton $p_n\subset\tilde{\alpha}_n$ or a horoball $p_n$ centered at an endpoint of~$\tilde{\alpha}_n$ such that $h'\cdot p_n\subset p_n$.
Then
$$d(p_0, h \cdot p_n) \geq d(p_0, g_{\tilde{\alpha}_n} \cdot p_n) \geq d(p_0, p_n) + 2n \eta_0 .$$
By the triangle inequality, $d(p_n, h \cdot p_n) \geq 2n\eta_0$.
The result follows.
\end{proof}

\begin{remark}
In \eqref{eqn:averageAdS}, we could have taken
$$g_{\tilde{\alpha}} := \Phi(\delta)\,\big(\Phi(\delta)^{-1}\,\Phi(\delta')\big)^t=\Phi(\delta')\,\big(\Phi(\delta')^{-1}\,\Phi(\delta)\big)^{1-t}\in G_0$$
for an arbitrary fixed $t\in (0,1)$, not necessarily $t=1/2$; the proof would have worked the same way.
(Here $g^t = \exp(t\log g)$.)
\end{remark}

\appendix
\section{Realizing the zero cocycle}\label{sec:remarks}

This appendix is a complement to the discussion of Section~\ref{subsec:abcd}, whose notation and setup we continue with (in particular, we refer to Figure~\ref{fig:claim3-1}).
It is not needed for the proofs of Theorems \ref{thm:main} and~\ref{thm:main-macro}.

In Section~\ref{subsec:abcd} we gave a realization, through the map $\mathscr{L}$ of \eqref{eqn:map-L}, of the zero cocycle as a linear combination of the infinitesimal strip deformation maps $\psi_{\alpha},\psi_{\alpha'},\psi_{\beta_i}$ by choosing specific geodesic representatives, waists, and widths.
In this appendix, we keep the same geodesic representatives (whose extensions to $\PP(\RR^{2,1})$ intersect in triples as in Figure~\ref{fig:claim3-1}) but vary the waists; we determine all possible realizations of the zero cocycle supported on the arcs $\alpha, \alpha', \beta_1, \beta_2, \beta_3, \beta_4$ and discuss their geometric significance in relation with Conjecture~\ref{conj}.

\subsection{Generalized infinitesimal strip deformations}

Let $\underline{\Delta}$ be a geodesic hyperideal triangulation of~$S$, with set of edges~$\mathscr{E}$.
Recall the notation $\Psi(\pm\widetilde{\mathscr{E}},\g)$ from Section~\ref{subsec:ookkee}.

\begin{definition} \label{def:gen-strip}
A relative motion map $\psi\in \Psi(\pm\widetilde{\mathscr{E}},\g)$ is called a \emph{generalized infinitesimal strip deformation} if $\psi(\tilde{\alpha})\in \spa(\tilde{\alpha})$ for any transversely oriented edge $\tilde{\alpha} \in \pm\widetilde{\mathscr{E}}$.
The \emph{support} of~$\psi$ is the set of arcs $\underline{\alpha}\subset S$ such that $\psi$ is nonzero on the lifts to~$\HH^2$ of $\underline{\alpha}$.
We also refer to the cohomology class in $T_{[\rho]}\T$ induced by $\psi$ as a generalized infinitesimal strip deformation.
\end{definition}

Infinitesimal strip deformations in the sense of Definition~\ref{def:inf-strip-deform} are generalized infinitesimal strip deformations for which $\psi(\tilde{\alpha})$ lies in one particular spacelike quadrant of the timelike plane $\spa(\tilde{\alpha}) \subset \RR^{2,1}$. 
We will call these strip deformations \emph{positive} (and their opposites \emph{negative}) to distinguish them among generalized strip deformations. 
Generalized infinitesimal strip deformations can also be neither positive nor negative: if for instance $\psi(\tilde{\alpha})\in\g$ is timelike, then the relative motion of the two tiles adjacent to $\tilde{\alpha}$ is an infinitesimal rotation centered at a point of~$\tilde{\alpha}$.

Generalized infinitesimal strip deformations supported on a single arc~$\underline{\alpha}$ form a linear $2$-plane in $H^1_{\rho}(\Gamma,\g)\simeq T_{[\rho]}\T$.

\subsection{The point $w\in\RR^{2,1}$ associated with a realization of the zero cocycle}

We now work in the setting of Section~\ref{subsec:abcd}: the hyperideal triangulations $\underline{\Delta}, \underline{\Delta}'$ differ by a single diagonal exchange and have a common refinement $\underline{\Delta}''$.
We have four spacelike vectors $v_1,v_2,v_3,v_4\in\RR^{2,1}$, scaled so that
\begin{equation}\label{eqn:norm-v_i-bis}
v_1 + v_3 = v_2 + v_4 ,
\end{equation}
and our chosen geodesic representatives $\underline{\alpha},\underline{\alpha}',\underline{\smash{\beta}}_i\in\underline{\Delta}''$ lift to $\tilde{\alpha},\tilde{\alpha}',\tilde{\beta}_i$ with
$$\left \{ \begin{array}{lcl}
  \tilde{\alpha} & = & \HH^2 \cap (\RR_+ v_1 + \RR_+ v_3) ,\\
  \tilde{\alpha}' & = & \HH^2 \cap (\RR_+ v_2 + \RR_+ v_4) ,\\
  \tilde{\beta}_i  & = & \HH^2 \cap (\RR_+ v_i + \RR_+ v_{i+1})
\end{array}\right .$$
for all $1\leq i\leq 4$ (with the convention that $v_5=v_1$).
For simplicity, we henceforth assume that $\underline{\smash{\beta}}_1,\dots,\underline{\smash{\beta}}_4$ are all distinct in the quotient surface $S$.
Recall that $\tilde{\alpha}={e}_1 \cup {e}_3$ and $\tilde{\alpha}'={e}_2\cup {e}_4$.
All edges $e_i$ and $\tilde{\beta}_i$ carry transverse orientations shown in Figure~\ref{fig:claim3-1}.

For any $w\in\RR^{2,1}$, we define a $(\rho, 0)$-equivariant map $\varphi_w : \widetilde{\mathscr{T}}''\to\g$ supported on the $\rho(\Gamma)$-orbits of the ``small'' tiles $\delta_1, \delta_2, \delta_3, \delta_4$ by
\begin{equation}\label{eqn:phi_w}
\varphi_w (\delta_i) := w\wedge (v_i\wedge v_{i+1})
\end{equation}
for all $1\leq i \leq 4$, where $\wedge$ is the Minkowski cross-product of Section~\ref{subsec:Killing-fields}.
As in Section~\ref{subsec:ookkee}, any $(\rho, 0)$-equivariant map $\varphi : \widetilde{\mathscr{T}}''\to\g$ defines a relative motion map $\psi : \widetilde{\mathscr{E}}''\rightarrow\g$ given by
$$\psi({e})=\varphi(\delta')-\varphi(\delta)$$
for any tiles $\delta,\delta'$ adjacent to an edge $e\in\pm\widetilde{\mathscr{E}}''$ which is transversely oriented from $\delta$ to~$\delta'$.

\begin{lemma}\label{lem:descript-phi}
Let $\varphi : \widetilde{\mathscr{T}}''\to\g$ be a $(\rho, 0)$-equivariant map with $\varphi=0$ outside the $\rho(\Gamma)$-orbits of $\delta_1,\delta_2,\delta_3,\delta_4$.
The associated map $\psi : \widetilde{\mathscr{E}} \to \g$ is a generalized infinitesimal strip deformation (Definition~\ref{def:gen-strip}) if and only if $\varphi=\varphi_w$ for some (unique) $w\in\RR^{2,1}$.
\end{lemma}

In this case the support of the generalized infinitesimal strip deformation is contained in $\{\underline{\alpha}, \underline{\alpha}', \underline{\smash{\beta}}_1, \underline{\smash{\beta}}_2, \underline{\smash{\beta}}_3, \underline{\smash{\beta}}_4\}$.
Note that Lemma~\ref{lem:descript-phi} holds regardless of how the other geodesic representatives $\underline{\smash{\eta}}$ for $\eta\in (\Delta\cap\Delta')\smallsetminus\{ \beta_1,\beta_2,\beta_3,\beta_4\}$ are chosen.

\begin{proof}
Since $\varphi=0$ outside the $\rho(\Gamma)$-orbits of the~$\delta_i$, we have $\psi=0$ outside the $\rho(\Gamma)$-orbits of the $\tilde{\beta}_i$ and~$e_i$ for $1\leq i\leq 4$.
By Definition~\ref{def:gen-strip}, the fact that $\psi$ is a generalized infinitesimal strip deformations is equivalent to
\begin{equation}\label{eqn:prop-psi}
\left \{ \begin{array}{l}
  \psi(\tilde{\beta}_i) \in \spa(v_i,v_{i+1}),\\
  \psi({e}_i) \in \spa(v_i,v_{i+2}) ,\\
  \psi({e}_i) = - \psi({e}_{i+2})
\end{array}\right .
\end{equation}
for all~$i$.
We first check that the space of $\rho$-equivariant maps $\varphi : \widetilde{\mathscr{T}}''\to\g$ for which $\psi$ satisfies \eqref{eqn:prop-psi} has dimension $\leq 3$.
By construction, $\psi(\tilde{\beta}_i)=\varphi(\delta_i)$ and $\psi({e}_i)=\varphi(\delta_i) - \varphi(\delta_{i-1})$.
If $\psi$ satisfies \eqref{eqn:prop-psi}, then we may write
$$\varphi(\delta_i) = a_i v_i - b_{i+1} v_{i+1} \in \spa(v_i,v_{i+1})$$
for all~$i$, where $a_1,b_1,\dots,a_4,b_4\in\RR$.
Since \eqref{eqn:norm-v_i-bis} is the only relation between $v_1,v_2,v_3,v_4$, the condition $\varphi(\delta_i)-\varphi(\delta_{i-1})\in\spa(v_i,v_{i+2})$ is equivalent to $b_{i+1}=a_{i-1}$, and so we may eliminate the $b_i$ and write
$$\varphi(\delta_i) = a_i v_i - a_{i-1} v_{i+1}$$
where $a_1,a_2,a_3,a_4\in\RR$.
The condition $\psi({e}_i)=-\psi({e}_{i+2})$ is equivalent to $\varphi(\delta_1) + \varphi(\delta_3) = \varphi(\delta_2) + \varphi(\delta_4)$, which amounts to
$$(a_1+a_3)\,v_1 - (a_2+a_4)\,v_2 + (a_3+a_1)\,v_3 - (a_4+a_2)\,v_4 = 0 ,$$
\ie $a_1+a_3=a_2+a_4$ by \eqref{eqn:norm-v_i-bis}.
The space of quadruples $(a_1,a_2,a_3,a_4)$ satisfying this condition has dimension~$3$, as announced.

The map $w\mapsto\varphi_w$ is linear and injective, for its kernel in~$\RR^{2,1}$ is contained in $\bigcap_{i=1}^4 \spa(v_i\wedge v_{i+1}) = \{ 0\}$.
Therefore, it only remains to see that for any $w\in\RR^{2,1}$, the map $\psi$ associated with~$\varphi_w$ satisfies \eqref{eqn:prop-psi}.
We have 
\begin{equation}\label{eqn:psi-beta_i}
\psi(\tilde{\beta}_i)=\varphi_w(\delta_i)-0=w\wedge (v_i \wedge v_{i+1}) \in \spa(v_i,v_{i+1}).
\end{equation}
Using \eqref{eqn:norm-v_i-bis} and the skew-symmetry of $\wedge$, we also have
\begin{equation}\label{eqn:psi-e_i}
\psi({e}_i) = \varphi_w(\delta_i)-\varphi_w(\delta_{i-1})=w\wedge (v_i\wedge (v_{i+1}+v_{i-1}))= w\wedge (v_i\wedge v_{i+2}),
\end{equation}
hence $\psi({e}_i)\in\spa(v_i,v_{i+2})$ and $\psi({e}_i)=-\psi({e}_{i+2})$.
\end{proof}

\subsection{The case of timelike~$w$}

We now consider the map $\varphi_w$ of \eqref{eqn:phi_w} when $w\in\RR^{2,1}$ is timelike.
We see $\HH^2$ as a hyperboloid in~$\RR^{2,1}$ as in \eqref{eqn:H2subsetR21}.

\begin{lemma}\label{lem:spacelike}
If $w\in\RR^{2,1}$ is timelike, then the map $\psi : \widetilde{\mathscr{E}}\to\g$ associated with~$\varphi_w$ is
a generalized infinitesimal strip deformation whose support is exactly $\{\underline{\alpha}, \underline{\alpha}', \underline{\smash{\beta}}_1, \underline{\smash{\beta}}_2, \underline{\smash{\beta}}_3, \underline{\smash{\beta}}_4\}$.
The six (nonzero) vectors $\psi(\tilde{\alpha}), \psi(\tilde{\alpha}'), \psi(\tilde{\beta}_i)$ are all spacelike, \ie correspond to either positive or negative infinitesimal strip deformations.
The corresponding waists on $\tilde{\alpha}, \tilde{\alpha}', \tilde{\beta}_i$ are the respective orthogonal projections of $\HH^2\cap\RR w$ in~$\HH^2$.
\end{lemma}

\begin{proof}
From \eqref{eqn:psi-beta_i} and \eqref{eqn:psi-e_i} we know that $\psi(\tilde{\beta}_i)$ and $\psi({e}_i)$ are orthogonal to~$w$, hence are zero or spacelike if $w$ is timelike.
To see that they are nonzero, we note that the planes $\spa(v_i,v_{i+1})$ and $\spa(v_i,v_{i+2})$ are timelike, hence the vectors $v_i \wedge v_{i+1}$ and $v_i \wedge v_{i+2}$ are spacelike; in particular, these vectors are not collinear to~$w$, and we conclude using \eqref{eqn:psi-beta_i} and \eqref{eqn:psi-e_i}.

By Lemma~\ref{lem:inf-transl-in-R21}, the translation axes of the (positive or negative) infinitesimal strip deformation along $\tilde{\alpha}$, $\tilde{\alpha}'$, $\tilde{\beta}_i$ are the intersections of $\HH^2$ with the orthogonal planes in~$\RR^{2,1}$ to $\psi(e_1)$, $\psi(e_2)$, $\psi(\tilde{\beta}_i)$, respectively.
By \eqref{eqn:psi-beta_i} and \eqref{eqn:psi-e_i}, these all go through $\HH^2\cap\RR w$.
\end{proof}

\subsection{The map $\varphi$ of Section~\ref{subsec:abcd}}

Recall that the map $\varphi$ we constructed in Section~\ref{subsec:abcd} is given by $\varphi(\delta_i)=v_{i+1}-v_i$ for all~$i$.
By \eqref{eqn:psi-good-linear-combin}, the associated map $\psi : \widetilde{\mathscr{E}} \to \g$ satisfies the hypotheses of Lemma~\ref{lem:descript-phi}, hence $\varphi=\varphi_{w_0}$ for some $w_0\in\RR^{2,1}$.
We now determine~$w_0$.
(This will not be needed afterwards.)

\begin{lemma}\label{lem:pis}
The vector $w_0\in\RR^{2,1}$ is timelike and satisfies $\HH^2\cap\RR w_0=\{p\}$, where $p\in\HH^2$ is the intersection point of the common perpendicular to $\tilde{\beta}_1$ and $\tilde{\beta}_3$ in~$\HH^2$ with the common perpendicular to $\tilde{\beta}_2$ and $\tilde{\beta}_4$.
\end{lemma}

\begin{proof}
Since $\varphi(\delta_i)=v_{i+1}-v_i$, we have $\varphi(\delta_i)+\varphi(\delta_{i+2})=0$ by \eqref{eqn:norm-v_i-bis}.
By \eqref{eqn:psi-beta_i} we have $\varphi(\delta_i)=w_0 \wedge (v_i\wedge v_{i+1})$.
Therefore, $w_0\wedge (v_i\wedge v_{i+1}+v_{i+2}\wedge v_{i+3})=0$, and so $w_0$ is a multiple of $v_i\wedge v_{i+1}+v_{i+2}\wedge v_{i+3}$.
In particular, $w_0$ belongs to the plane $\RR (v_i\wedge v_{i+1})+\RR (v_{i+2}\wedge v_{i+3})$, whose intersection with~$\HH^2$ is the common perpendicular to $\tilde{\beta}_i$ and $\tilde{\beta}_{i+2}$.
This holds for all $1\leq i\leq 4$.
\end{proof}

In general, $\HH^2\cap\RR w_0$ is \emph{not} the point $\tilde{\alpha}\cap\tilde{\alpha}'$.

\subsection{Link with Conjecture~\ref{conj}}

The following lemma provides evidence for Conjecture~\ref{conj}.

\begin{lemma}\label{lem:evidence-conj}
Let the map ${\boldsymbol f} : \overline{X}\rightarrow H^1_{\rho}(\Gamma,\g)$ of Section~\ref{subsec:arc-complex} be defined with respect to our choice of geodesic representatives of the arcs of $\Delta\cup\Delta'$, with waists on $\tilde{\alpha}, \tilde{\alpha}', \tilde{\beta_i}$ that are all orthogonal projections of a common point $p\in\HH^2$, and with infinitesimal widths $m_{\alpha}, m_{\alpha'}, m_{\beta_i}$ all equal to~$1$.
Then the image of ${\boldsymbol f}$ looks salient at the codimension-2 face shared by ${\boldsymbol f}(\Delta)$ and ${\boldsymbol f}(\Delta')$, when seen from the origin of $H^1_\rho(\Gamma,\g)$.
\end{lemma}

\begin{proof}
Recall the linear relation \eqref{eqn:zerocross}:
$$c_{\alpha} \boldsymbol{f}(\alpha) + c_{\alpha'} \boldsymbol{f}(\alpha') + \sum_{\substack{\beta\text{ arc of both}\\ \Delta\text{ and }\Delta'}} c_{\beta}\,{\boldsymbol f}(\beta) \, = \, 0 \, \in H^1_{\rho}(\Gamma,\g).$$ 
By Claim~\ref{claim:main}.(0) the coefficients $c_{\alpha}, c_{\alpha'}, c_{\beta}$ are unique up to scale, and by Claim~\ref{claim:main}--(1) we may take $c_{\alpha}$ and $c_{\alpha'}$ to be positive.
The fact that the image of ${\boldsymbol f}$ looks salient at the codimension-2 face shared by ${\boldsymbol f}(\Delta)$ and ${\boldsymbol f}(\Delta')$ is then expressed by
\begin{equation}\label{eqn:salient}
c_{\alpha} + c_{\alpha'} + \sum_{\substack{\beta\text{ arc of both}\\ \Delta\text{ and }\Delta'}} c_{\beta} <0.
\end{equation}
Let us prove that this inequality holds.

By definition, the coefficients $c_{\alpha}, c_{\alpha'}, c_{\beta}$ encode a realization of the zero cocycle by (positive and negative) infinitesimal strip deformations with the given geodesic representatives and waists.
By Lemma~\ref{lem:spacelike} and uniqueness of $c_{\alpha}, c_{\alpha'}, c_{\beta}$, this realization is of the form $\varphi_w$ for some timelike $w\in\RR^{2,1}$ with $\HH^2\cap\RR w=\{ p\}$.
(In particular, $c_{\beta}=0$ for $\beta\notin \{\alpha, \alpha', \beta_1, \beta_2, \beta_3, \beta_4\}$.)
Moreover, if $p$ varies continuously in~$\HH^2$ while $w$ remains in the same component of the timelike cone of~$\RR^{2,1}$, then the signs (positive or negative) of the infinitesimal strip deformations defining~$\varphi_w$ remain constant, similar to \eqref{eqn:requested-signs-w}.
Therefore, since the infinitesimal widths $m_{\alpha}, m_{\alpha'}, m_{\beta_i}$ are all equal to~$1$, the relation \eqref{eqn:zerocross} has the form
$$\Vert x_1 - x_2 \Vert \, {\boldsymbol f}(\alpha') + \Vert x_2 - x_3 \Vert \, {\boldsymbol f}(\alpha) - \sum_{i=1}^4 \Vert x_i\Vert \, {\boldsymbol f}(\beta_i) = 0 ,$$
where we set $x_i:=\varphi(\delta_i)$, and $\Vert x \Vert:=\sqrt{\langle x, x \rangle} >0$ for spacelike $x\in\RR^{2,1}$.
Note that the points $v_i\wedge v_{i+1}$, for $1\leq i\leq 4$, form a parallelogram in~$\RR^{2,1}$: indeed,
$$v_1\wedge v_2 + v_3\wedge v_4 - v_2\wedge v_3 - v_4\wedge v_1 = (v_1+v_3)\wedge (v_2+v_4)=0$$
by \eqref{eqn:norm-v_i-bis}.
By Lemma~\ref{lem:descript-phi}, it follows that the $x_i=w\wedge (v_i\wedge v_{i+1})$, for $1\leq i\leq 4$, form a parallelogram in the spacelike plane $w^{\perp} \subset \RR^{2,1}$.
As a consequence,
$$\Vert x_1 - x_2 \Vert + \Vert x_2 - x_3 \Vert - \sum_{i=1}^4 \Vert x_i\Vert <0,$$
and so \eqref{eqn:salient} holds.
\end{proof}

Note, however, that if the waists are chosen arbitrarily, so that they are not all projections of a common point $p\in\HH^2$, then typically none of the terms $c_{\beta}$ of \eqref{eqn:zerocross} vanish, and the signs of the terms other than $c_{\alpha}$ and $c_{\alpha'}$ may vary: the conclusion of Lemma~\ref{lem:evidence-conj} might then fail.

\vspace{0.5cm}

\end{document}